\documentclass[reqno,12pt,letterpaper]{amsart}
\usepackage{amsmath,amssymb,amsthm,mathrsfs,graphicx,url}
\usepackage[usenames,dvipsnames]{color}
\usepackage[colorlinks=true,linkcolor=Red,citecolor=Green]{hyperref}

\newtheorem{theo}{Theorem}
\newtheorem{lemm}{Lemma}[section]

\numberwithin{equation}{section}

\DeclareMathOperator{\ad}{ad}
\DeclareMathOperator{\cl}{cl}
\DeclareMathOperator{\comp}{comp}

\DeclareMathOperator{\even}{even}

\DeclareMathOperator{\Imag}{Im}
\DeclareMathOperator{\loc}{loc}
\DeclareMathOperator{\rank}{rank}
\DeclareMathOperator{\Op}{Op}
\DeclareMathOperator{\Real}{Re}
\DeclareMathOperator{\sgn}{sgn}
\DeclareMathOperator{\supp}{supp}
\DeclareMathOperator{\Vol}{Vol}
\DeclareMathOperator{\WF}{WF}

\def\WFh{\WF_{\hbar}}
\def\Hh{H_{\hbar}}
\def\Psic{\Psi^{\comp}}
\def\Psie{\Psi^{\comp}_{1/2}}
\def\Syme{S^{\comp}_{1/2}}
\def\Resh{\mathcal O(h^\infty)_{\Psi^{-\infty}}}
\def\lxir{\langle\xi\rangle}

\title[Fractal Weyl laws for asymptotically hyperbolic manifolds]%
{Fractal Weyl laws for\\
asymptotically hyperbolic manifolds}
\author{Kiril Datchev}
\email{datchev@math.mit.edu}
\address{Department of Mathematics, 77 Massachusetts Avenue, MIT,
Cambridge, MA 02139}
\author{Semyon Dyatlov}
\email{dyatlov@math.berkeley.edu}
\address{Department of Mathematics, Evans Hall, University of California,
Berkeley, CA 94720}

\begin{document}

\begin{abstract}
For asymptotically hyperbolic manifolds with hyperbolic trapped sets
we prove a fractal upper bound on the number of resonances near the essential spectrum,
with power determined by the dimension of the trapped set.
This covers the case of general convex 
cocompact quotients (including the case of connected trapped sets)
where our result implies a bound on the number of zeros of the
Selberg zeta function in disks of arbitrary size along the imaginary axis.
Although no sharp fractal lower bounds are known, the case of quasifuchsian groups,
included here, is most likely to provide them.
\end{abstract}

\maketitle

%%%%%%%%%%%%%%%%%%%%%%%%%%%%%%%%%%%%%%%%%%%%%%%%%%%%%%%%%%%%%%%%%%%%%%%%%%%%%%%%
%                                 INTRODUCTION                                 %
%%%%%%%%%%%%%%%%%%%%%%%%%%%%%%%%%%%%%%%%%%%%%%%%%%%%%%%%%%%%%%%%%%%%%%%%%%%%%%%%
\addtocounter{section}{1}
\addcontentsline{toc}{section}{1. Introduction}

Let $M=\Gamma \backslash \mathbb{H}^n$, $n\ge 2$, be a convex cocompact quotient 
of hyperbolic space,
i.e. a conformally compact
 manifold of
constant negative curvature. Let
$\delta_\Gamma \in [0,n-1)$ be the Hausdorff dimension of its limit set, which
by Patterson--Sullivan theory equals the abscissa of convergence of
its Poincar\'e series \cite{p1,sul}. Let $Z_\Gamma(s)$ be the Selberg 
zeta function:
\[
Z_\Gamma(s) = \exp\left(- \sum_{\gamma \in \mathcal{P}} \sum_{m=1}^\infty \frac 1 m \frac {e^{-sml(\gamma)}}{\det(\textrm{Id} - P_\gamma)}\right),
\]
where $\mathcal{P}$ is the set of primitive closed geodesics on $M$,
$l(\gamma)$ is the length of $\gamma$, and $P_\gamma$ is the
Poincar\'e one-return map of $\gamma$ in $T^*M$. The sums
converge absolutely for $\Real s > \delta_\Gamma$ (so $Z_\Gamma$ is
nonvanishing there), and the function $Z_\Gamma$ extends
holomorphically to $\mathbb{C}
\setminus(-\mathbb{N}_0 \cup((n-1)/2 - \mathbb{N}))$ \cite{fr,bun,pp}.
Let $m_\Gamma(s)$ be the order of vanishing of $Z_\Gamma$ at $s$.

%
%%%%%%%%%%%%%%%%%%%%%%%%%%%%%% BEGIN PROP %%%%%%%%%%%%%%%%%%%%%%%%%%%%%%
%
\begin{theo}\label{l:theorem-weak}
For any $R>0$ there exists $C>0$ such that if $t \in \mathbb{R}$, then
\begin{equation}\label{e:0}
\sum_{\left|s-it\right|<R} m_\Gamma(s) \le C(1 + |t|)^{\delta_\Gamma}.
\end{equation}
\end{theo}
%
%%%%%%%%%%%%%%%%%%%%%%%%%%%%%%% END PROP %%%%%%%%%%%%%%%%%%%%%%%%%%%%%%%
%
This was proved by Guillop\'e--Lin--Zworski~\cite{glz} in the case when
$\Gamma$ is Schottky. In this paper we consider general convex
cocompact quotients (see Figure~\ref{f:puredim} for examples), 
and, as we explain below, also give a
generalization to the case of nonconstant curvature.

%
%%%%%%%%%%%%%%%%%%%%%%%%%%%%%% BEGIN FIGURE %%%%%%%%%%%%%%%%%%%%%%%%%%%%%%
\begin{figure}[h]
\vspace{.1cm}
\includegraphics[width=3cm]{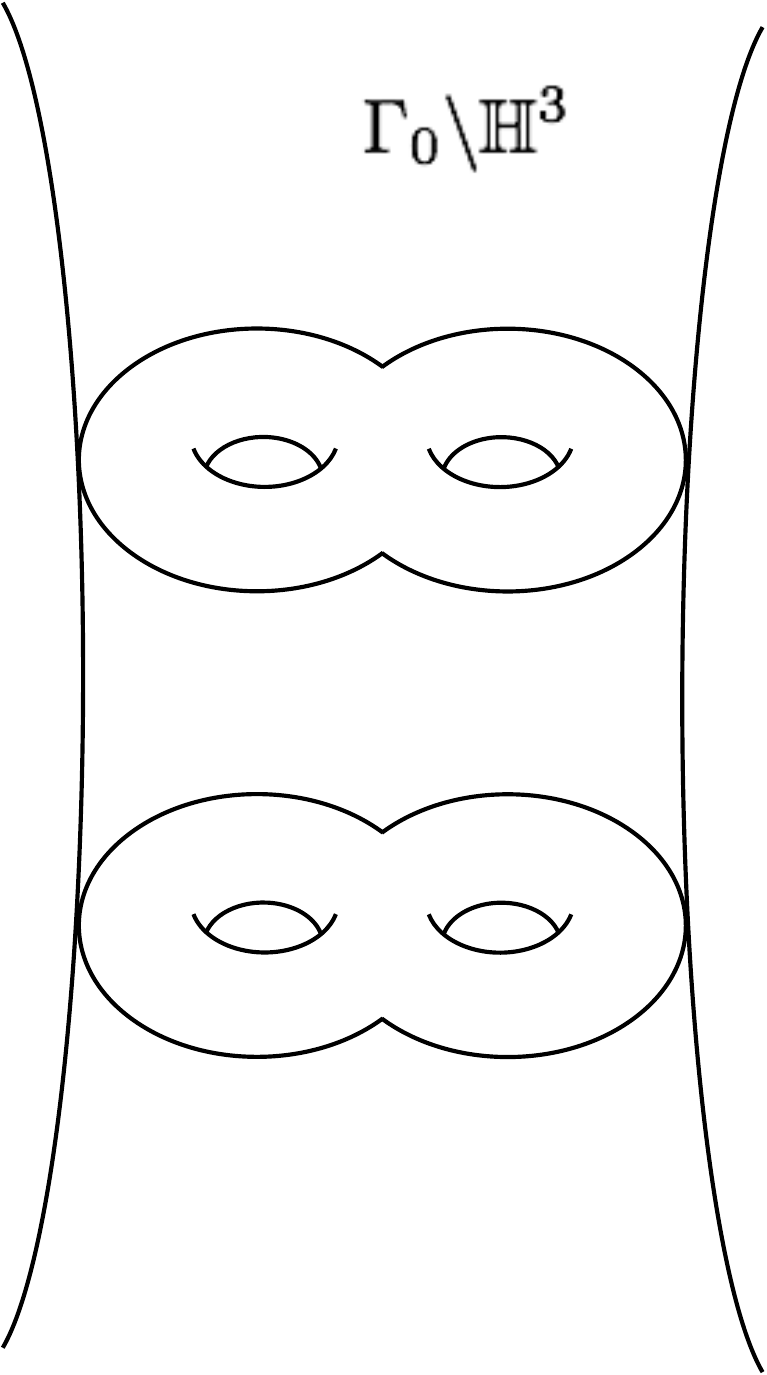}
\hspace{1cm}
\includegraphics[width=4.5cm]{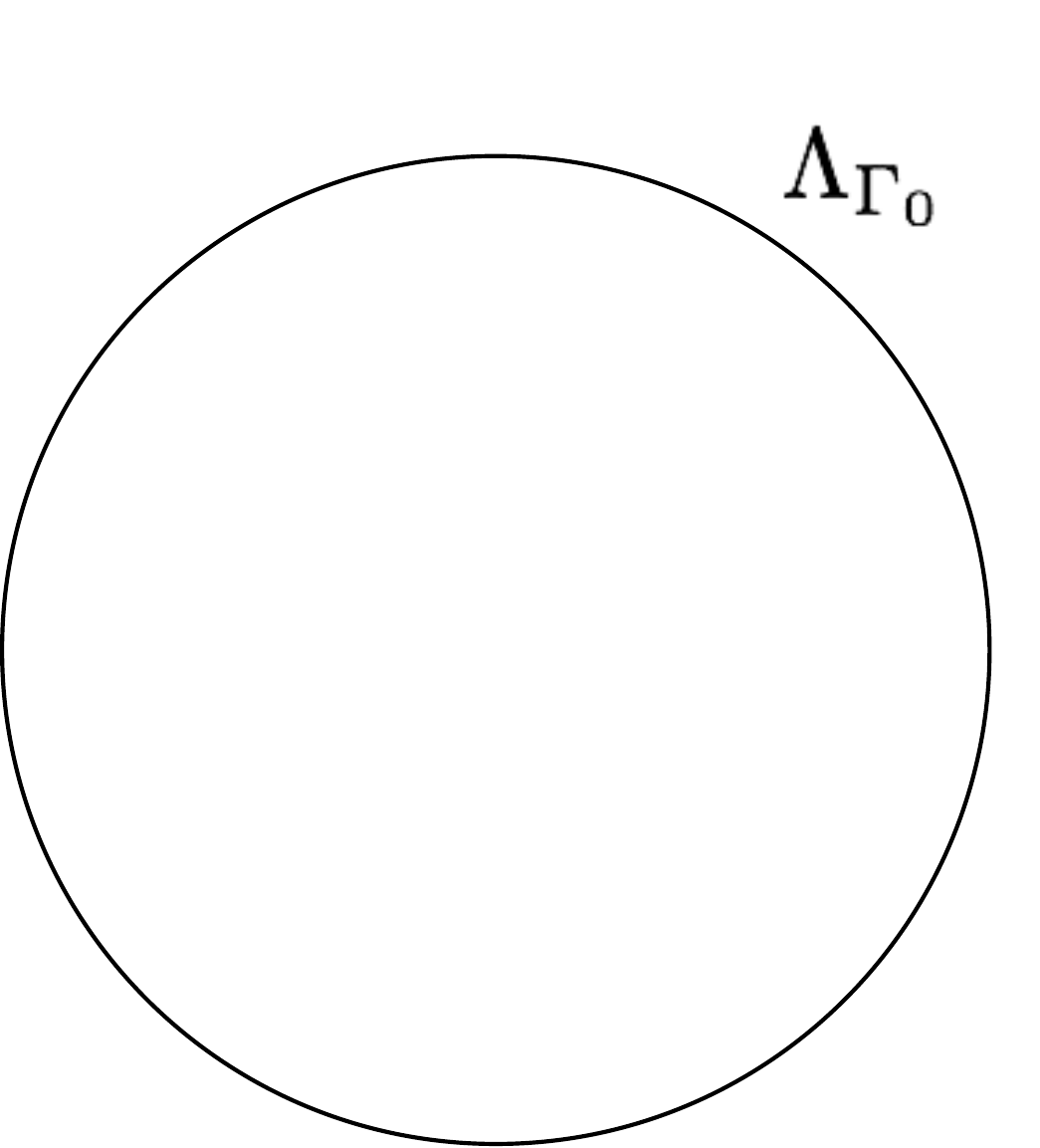}

\vspace{.3cm}

\includegraphics[width=3cm]{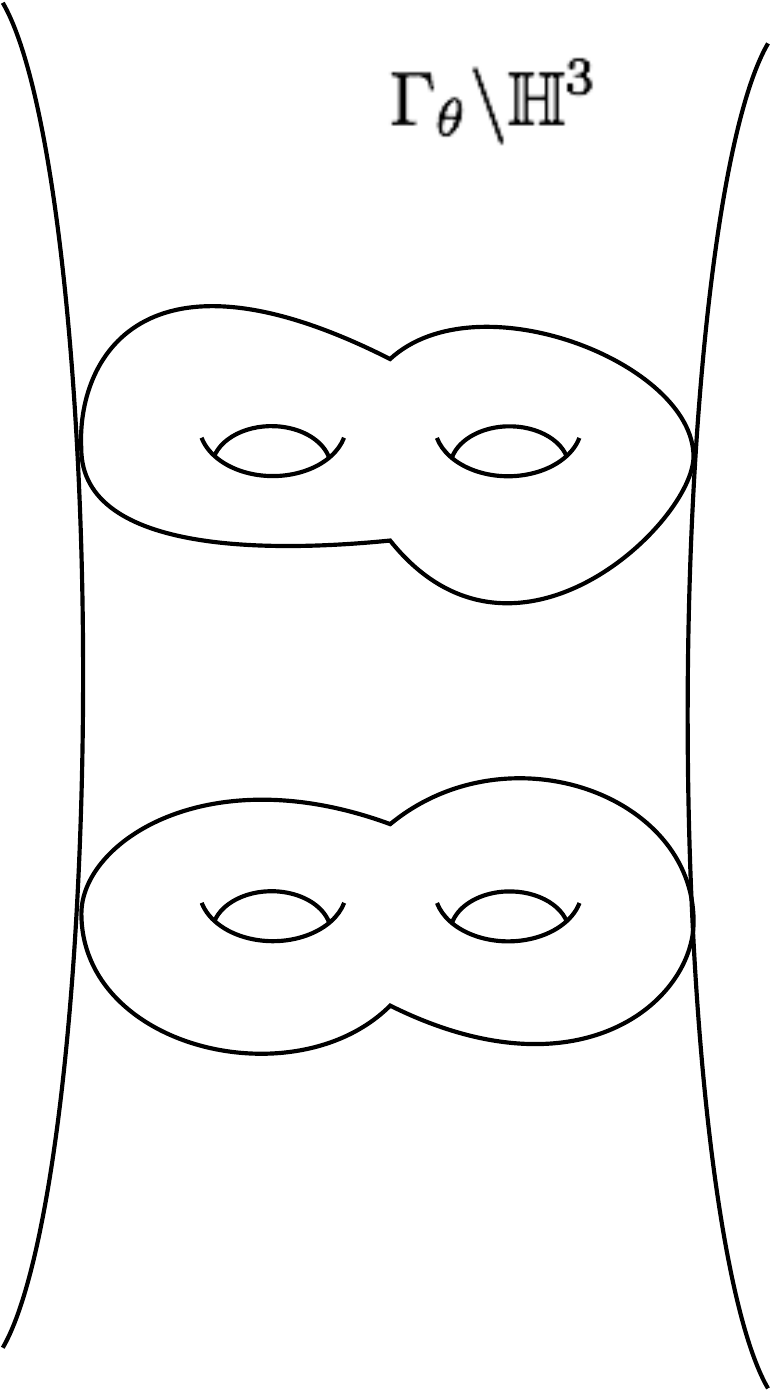}
\hspace{1cm}
\includegraphics[width=4.5cm]{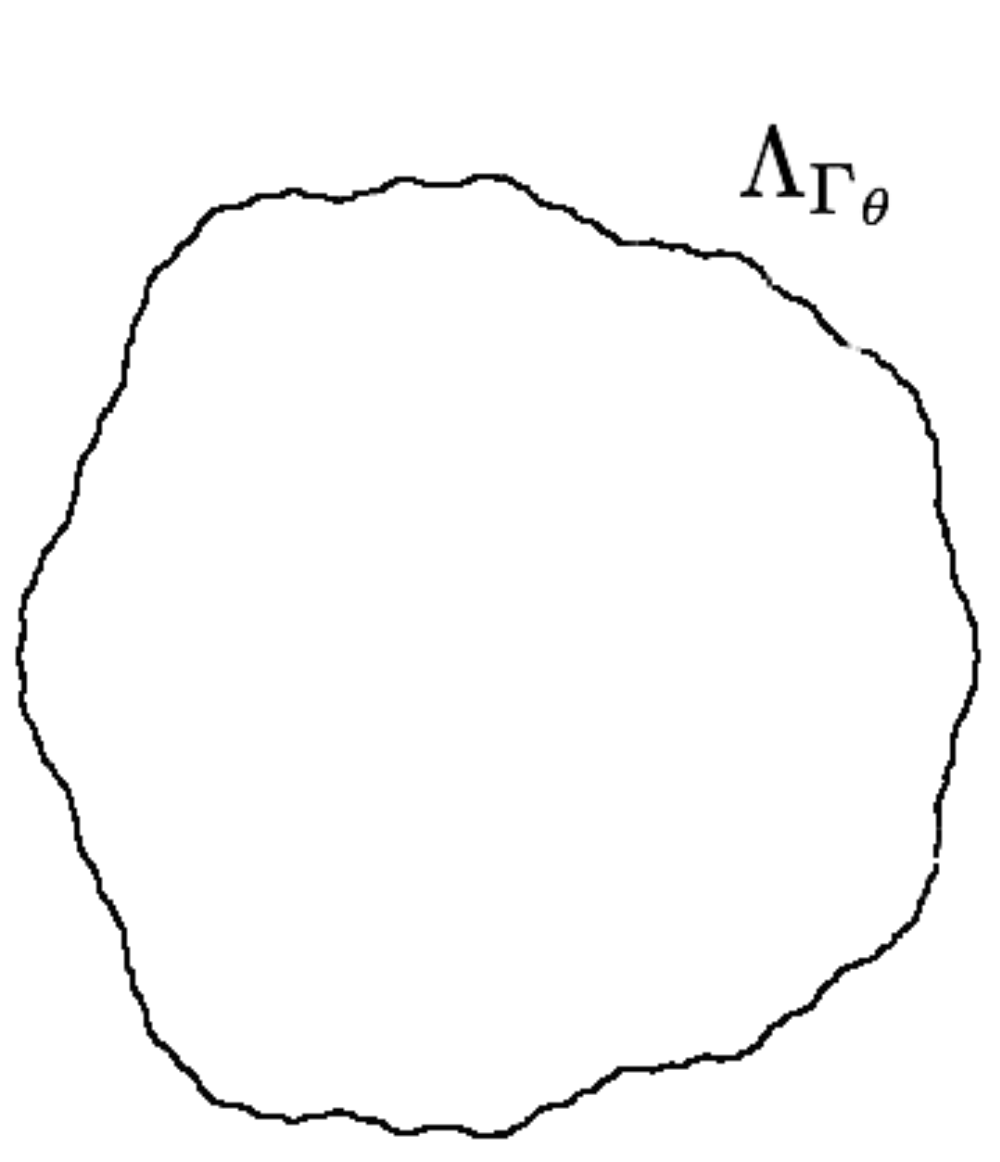}
\caption{
Let $\Gamma_0$ be a cocompact Fuchsian group. Then 
$\Gamma_0 \backslash \mathbb H^3$ is convex cocompact,
its limit 
set $\Lambda_{\Gamma_0}$ is a circle and $\delta_{\Gamma_0} = 1$.
For $\Gamma_\theta$ a  \textit{quasifuchsian bending} of $\Gamma_0$,
$\Gamma_\theta \backslash \mathbb H^3$ is convex cocompact,
 but its limit
set $\Lambda_{\Gamma_\theta}$ is  a \textit{quasicircle} 
and
$\delta_{\Gamma_\theta}>1$.
Theorems~\ref{l:theorem-weak} and~\ref{l:theorem-strong} apply,
but $\Gamma_\theta \backslash \mathbb H^3$ is not Schottky,
$\Lambda_{\Gamma_\theta}$ is connected,
 and
the trapped set is of pure fractal dimension. See 
Appendix~\ref{s:quasifuchsian}.}
\label{f:puredim}
\end{figure}
%
%%%%%%%%%%%%%%%%%%%%%%%%%%%%%%% END FIGURE %%%%%%%%%%%%%%%%%%%%%%%%%%%%%%%

Theorem \ref{l:theorem-weak} follows from Theorem
\ref{l:theorem-strong} below which holds in a general geometric setting. 
The novelty of our approach lies in combining recent results of
Vasy~\cite{v1} on effective meromorphic continuation
with the technology of~\cite{sj-z} for resonance
counting. (For a direct presentation of Vasy's construction in the explicit
setting of the hyperbolic cylinder, see Figure~\ref{f:vasy} and Appendix~\ref{s:picture}.)
A particular challenge comes from constructing
Lyapunov/escape
functions compatible with both approaches.
We also simplify the counting argument on $h$-size scales
by replacing the complicated second microlocalization of~\cite{sj-z} by suitably adapted
functional calculus.
The authors of~\cite{sj-z} also considered the
use of functional calculus in their treatment of $h$-size neighborhoods of the energy surface,
but chose a fully microlocal approach.

%
%%%%%%%%%%%%%%%%%%%%%%%%%%%%%% BEGIN FIGURE %%%%%%%%%%%%%%%%%%%%%%%%%%%%%%
\begin{figure}
\includegraphics{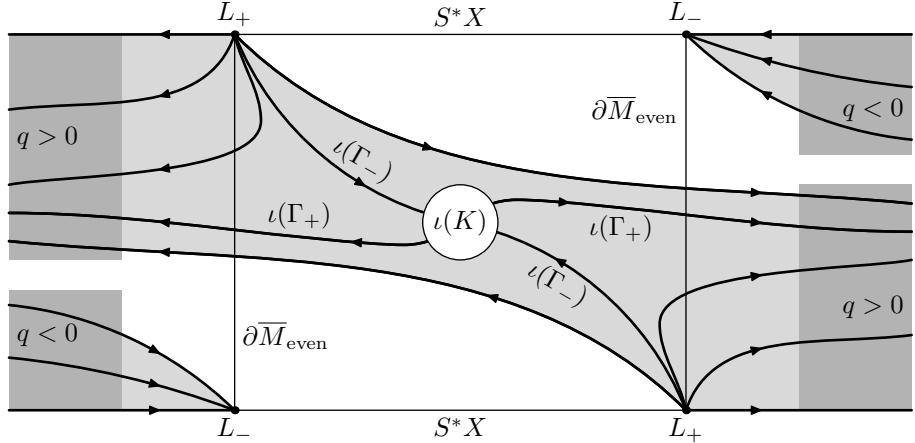}
\caption{A schematic presentation of the dynamical complexity
for the hyperbolic cylinder $M=(-1,1)_r\times \mathbb S^1_{\tilde y}$.
Here $\iota(K)$ denotes the trapped set, which
connects to infinity through the incoming and outgoing tails $\iota(\Gamma_\pm)$.
The vertical lines $\partial \overline M_{\even}$ correspond
to the two funnel ends $\{r=\pm 1\}$; the horizontal variable is $r$ and
the vertical variable is the compactification $\zeta/\langle\zeta\rangle$
of the momentum $\zeta$ dual to $r$. The lighter shaded regions are the components
$\Sigma_\pm$ of the energy surface,
with $\Sigma_+$ the bigger region and $\Sigma_-$ the union
of two small ones. The darker shaded regions are the sets
where the complex absorbing operator
$Q$ is elliptic. See Appendix~\ref{s:picture} for more details.}
\label{f:vasy}
\end{figure}
%
%%%%%%%%%%%%%%%%%%%%%%%%%%%%%%% END FIGURE %%%%%%%%%%%%%%%%%%%%%%%%%%%%%%%
%
To state the main theorem, let $(M,g)$ be an asymptotically 
hyperbolic manifold; i.e. $M$ is the interior of a compact manifold with boundary $\overline
M$ and $g$ is a metric on $M$ such that
\begin{equation}
  \label{e:as-hyp}
g={d\tilde x^2+g_1\over \tilde x^2}
\end{equation}
near $\partial \overline M$, where $\tilde x$ is a  defining
function for $\partial \overline M$ and 
$g_1$ a 2-cotensor with $g_1|_{\partial \overline M}$ a metric.
Suppose $g_1$ is even in the sense 
of being smooth in $\tilde x^2$ (see also~\cite[Definition 1.2]{gui}).

Let $\Delta_g$ be the positive Laplacian on $(M,g)$.
Its resolvent, $(\Delta_{g} - s(n-1-s))^{-1}$, is
meromorphic $L^2(M) \to L^2(M)$ for $\Real s >
\frac{n-1}2$ and the essential spectrum  is the line $\Real s =
\frac{n-1}2$. As an operator from $L^2_{\textrm{comp}}$ to $L^2_{\textrm{loc}}$ 
 it extends meromorphically to $\mathbb{C}$,
with finite rank poles called \textit{resonances}, as shown by Mazzeo--Melrose~\cite{mm}
and Guillarmou~\cite{gui} for asymptotically
hyperbolic manifolds, Guillop\'e--Zworski~\cite{gzasy}
when curvature
near infinity is constant, and Vasy~\cite{v1} for asymptotically hyperbolic
manifolds with $g_1$ even.

When $M = \Gamma \backslash \mathbb{H}^n$,
Bunke--Olbrich \cite{bun} and Patterson--Perry \cite[Theorems 1.5,
1.6]{pp} show  that zeros of $Z_\Gamma$ and \textit{scattering poles} 
coincide on $\mathbb{C}
\setminus(-\mathbb{N}_0 \cup({n-1\over 2} - \mathbb{N}))$. 
Guillop\'e--Zworski~\cite{gz} and Borthwick--Perry~\cite[Theorem 1.1]{bp} show that scattering poles 
 and  resonances coincide 
off a discrete subset of $\mathbb R$.
  Below, we bound the density of
resonances near the essential spectrum. Theorem
\ref{l:theorem-weak} follows from the correspondence between
resonances  and  zeros of $Z_\Gamma$ (the
discrete set where the correspondence fails, which has been further studied in e.g. \cite{grahamzworski, gui2}, is a subset of $\mathbb{R}$ and hence
irrelevant here).

We assume further that the geodesic flow on $M$ is
hyperbolic on its trapped set in the following sense of Anosov. 
(In fact, it is sufficient to assume hyperbolicity in the weaker sense of~\cite[\S 5]{sj} and~\cite[\S 7]{sj-z}.
The latter includes the case of normally hyperbolic trapped sets, see for
example~\cite{w-z}.)
Let $p_0 \in C^\infty(T^*M)$ be the (shifted) geodesic Hamiltonian:
\[
p_0(\rho) = g^{-1}(\rho,\rho) - 1,
\]
where $g^{-1}$ is the dual metric to $g$. Let $H_{p_0}$ be the
Hamilton vector field of $p_0$ and  $\exp(tH_{p_0})$
its flow. Define the \textit{trapped set} 
and its intersection with the energy surface $p_0^{-1}(0)$ by
\[
\widetilde K =  \big\{\rho \in T^*M \setminus 0 \mid \{\exp(tH_{p_0})\rho \mid  t \in \mathbb{R}\} \textrm{ is bounded}\big\},  \  K = \widetilde K \cap p_0^{-1}(0).
\]
Note that the homogeneity of $p_0$ in the fibers implies that
$\widetilde K$ is conic in the fibers,
and that \eqref{e:as-hyp}
implies that $K$ is compact. 
Our assumption is that for any $\rho \in
K$, the tangent space to $p_0^{-1}(1)$ at
$\rho$ splits into flow, unstable, and stable subspaces
\cite[Definition 17.4.1]{kh}:
%%%%%%%%%%%%%%%%%%%%%%%%%%%%%%
\begin{enumerate}
\item $T_\rho p_0^{-1}(1) = \mathbb{R} H_{p_0} \oplus E^+_\rho \oplus E_\rho^-, \ \dim E^\pm_\rho = n-1$,
\item $d \exp(tH_{p_0})E^\pm_\rho =  E^\pm_{\exp(tH_{p_0})\rho}$ for all $t \in \mathbb{R}$,
\item there exists $\lambda > 0$ such that $\|d \exp(tH_{p_0}) v\| \le C e^{-\lambda|t|}\|v\|$ for all $v \in E^{\mp}_\rho$, $\pm t \ge 0$.
\end{enumerate}
%%%%%%%%%%%%%%%%%%%%%%%%%%%%%%
Here we consider the differential $d$ of $\exp(tH_{p_0}) \colon T^*M \to
T^*M$ as a map $T_\rho T^*M \to T_{\exp(tH_{p_0})\rho} T^*M$.

Recall  a bounded subset $B$ of an $N$-dimensional manifold has upper Minkowski dimension
\begin{equation}\label{e:minkdef}
\inf\{d \mid \exists C>0, \forall \varepsilon \in (0,1], \Vol(B_\varepsilon) \le C \varepsilon^{N-d}\},
\end{equation}
where $B_\varepsilon$ is the $\varepsilon$-neighborhood of $B$.
The dimension is~\emph{pure} if the infimum is attained. 

We state our main theorem for the semiclassical, nonnegative Laplacian with
spectral parameter $E = h^2s(n-1-s) - 1$, and define the multiplicity of a
pole at $E \in \mathbb{C}$ by
\[
m_h(E) = \rank \left[\oint_{E} (h^2 \Delta_g -1- E')^{-1} dE' \colon L^2_\textrm{comp}(M) \to L^2_\textrm{loc}(M)\right].
\]
%
%%%%%%%%%%%%%%%%%%%%%%%%%%%%%% BEGIN PROP %%%%%%%%%%%%%%%%%%%%%%%%%%%%%%
%
\begin{theo}[Main theorem]\label{l:theorem-strong}
Let $(M,g)$ be asymptotically hyperbolic in the sense of \eqref{e:as-hyp}, and suppose $g_1$ is even and  the geodesic flow is hyperbolic on $K$. Let $2\nu_0+1$ be the upper Minkowski dimension of $ K$.
For any $\nu > \nu_0$, $c_0 > 0$ there exist $c_1, h_0>0$ such that 
\begin{equation}\label{e:1}
\sum_{ |E| < c_0h}m_h(E)  \le c_1 h^{-\nu},
\end{equation}
for $h \in (0,h_0]$. If $ K$ is of pure dimension, we may take $\nu =\nu_0$.
\end{theo}
%
%%%%%%%%%%%%%%%%%%%%%%%%%%%%%%% END PROP %%%%%%%%%%%%%%%%%%%%%%%%%%%%%%%
%

Our assumptions hold for $M$ convex cocompact,
 but they only concern the
asymptotic structure of $g$  at infinity, and the dynamics
of the geodesic flow at $K$. The latter assumption
holds when $M$ has (possibly variable)  negative 
curvature~\cite[Theorem 3.9.1]{kli}.

When $M$ is convex cocompact, we have
 $\dim K = 2\delta_\Gamma + 1$, and
the Minkowski dimension is pure and equals the Hausdorff
dimension (\cite[\S3]{sul}, see also~\cite[\S8.1]{nic},~\cite[Corollary 1.5]{bj}). 
Theorem~\ref{l:theorem-weak}  follows from Theorem~\ref{l:theorem-strong} 
and the fact that resonances of the Laplacian and zeros of the 
Selberg zeta function agree with multiplicities in the domains we study.

Theorem \ref{l:theorem-strong}  is to be compared with the Weyl law for
eigenvalues of a compact manifold:
\begin{equation}
  \label{e:weyl1}
\sum_{|E| < a}  m_h(E) = (2 \pi h)^{-n} \Vol\left(p_0^{-1}[-a,a]\right) + 
\mathcal O(h^{-n+1}), \quad a \in (0,1),
\end{equation}
where $n = \dim M$, and with the corresponding  bound when one considers smaller domains
\begin{equation}
  \label{e:weyl2}
\sum_{|E| < c_0h} m_h(E)  \le c_1 h^{-n+1},
\end{equation}
see for example \cite[(1.1)]{sj-z}.
The notation in \eqref{e:weyl1} and \eqref{e:weyl2} is as in \eqref{e:1} but
the manifold is compact, so $m_h(E) \ne 0$ only for real $E$ and
gives the multiplicity of the eigenvalue at $E$.

If 
$\Gamma$ is cocompact Fuchsian, as in Figure~\ref{f:puredim},
the Laplacian on $\Gamma \backslash \mathbb H^3$ is unitarily equivalent to 
$\bigoplus_{j \in \mathbb N_0} D_r^2 + \lambda_j \operatorname{sech}^2 r + 1$,
where $\lambda_j \ge 0$ are the eigenvalues of 
$\Gamma \backslash \mathbb H^2$.
By \cite[Appendix]{gzjfa},
  the scattering poles of
$\Gamma \backslash \mathbb H^3$ are  
$s_{j,k} =  \sqrt{1/4 - \lambda_j} + 1/2 - k$
where  $j \in \mathbb N$, 
$k \in \mathbb N_0$,
and the square root has values in $(-1/2,0] \cup i \mathbb R$.
The
Weyl law for $\Gamma \backslash \mathbb H^2$
implies that  in this case \eqref{e:0} and \eqref{e:1}
 can be improved to asymptotics:
\[
\sum_{|s-it|<R} m_\Gamma(s) = C|t|^{\delta_\Gamma}(1 + o(1)), \qquad \sum_{ |E| < c_0h}m_h(E)  = c_1 h^{-\nu}(1 + o(1)), \qquad \delta_\Gamma = \nu = 1.
\]

The first polynomial upper bounds on the resonance counting function
are due to Melrose \cite{mel}, and the first 
bounds involving geometric data (the dimension of the trapped
set) are due to Sj\"ostrand \cite{sj}.  Sj\"ostrand's result
was proved for convex cocompact surfaces by Zworski
\cite{z} (with an improvement by Naud \cite{n})
 and for analytic scattering
manifolds by Wunsch--Zworski \cite{wz}.  Guillop\'e--Lin--Zworski
\cite{glz} proved Theorem~\ref{l:theorem-strong}
 for convex cocompact Schottky groups, including 
general convex cocompact surfaces. 
Sj\"ostrand--Zworski \cite{sj-z} obtained the first result allowing 
$C^\infty$ perturbations, and our analysis of the
trapped set partially follows theirs,  with  modifications needed
due to the asymptotically hyperbolic infinity.
 As in~\cite{glz} and~\cite{sj-z},
we count resonances in $ \mathcal O(h)$ size regions,
rather than in the larger regions of ~\cite{sj}.  Most recently,
Nonnenmacher--Sj\"ostrand--Zworski~\cite{nszj, nszj2} studied general
topologically one-dimensional hyperbolic flows and proved the analog
of Theorem~\ref{l:theorem-strong} for scattering by several convex obstacles.

Much less is known about lower bounds.
In~\cite{g-s}, G\'erard--Sj\"ostrand studied semiclassical
problems with $K$  a single periodic orbit
and proved that resonances lie asymptotically on a lattice, implying
in particular sharp upper and lower bounds with $\nu = 0$. 
Similar
results hold for  spherically symmetric problems, as studied by S\'a
Barreto--Zworski~\cite{sabzwo}, and (under a stronger
hyperbolicity condition) for perturbations of them,
as studied by the second author~\cite{zeeman}. Most recently, Weyl asymptotics for resonances in strips have been obtained under more restrictive dynamical assumptions in work in preparation by Faure--Tsuji~\cite{ft} and the second author~\cite{dyatlov}. Nonnenmacher--Zworski
\cite{nz} proved asymptotics with
fractal $\nu$ for toy models of open quantum systems, and Lu--Sridhar--Zworski \cite{lsz}
give numerical evidence supporting an asymptotic in the case of obstacle scattering. 
Jakobson--Naud \cite{jn} proved
logarithmic lower bounds with exponent related to $\delta_\Gamma$
for convex cocompact surfaces, 
 and in the arithmetic case
their  bounds are fractal, but the power is different from the
one for the upper bound.

For a broader introduction to the subject of the distribution of resonances for systems
with hyperbolic classical dynamics, we refer the reader to the recent review
paper of Nonnenmacher \cite{nonn}, which describes many of
these and other results. It also includes a discussion of the
theoretical and experimental physics literature on the subject,
which supports the optimality of our upper bound (although without rigorous proofs).

An interesting
open problem is to prove the analog of Theorem~\ref{l:theorem-weak}
or~\ref{l:theorem-strong}  for manifolds
with cusps; when $M = \Gamma \backslash \mathbb{H}^n$
this means $\Gamma$ has parabolic
elements. If cusps have mixed rank
 the problem is harder;  Guillarmou--Mazzeo 
\cite{gm} show that the resolvent is meromorphic in $\mathbb C$
but continuation of the zeta function is not known.
If $n=2$ all cusps have full rank, and  the main difficulty
 comes from the fact that $K$ extends into the cusp and is 
not compact.

%%%%%%%%%%%%%%%%%%%%%%%%%%%%%%%%%%%%%%%%%%%%%%%%%%%%%%%%%%%%%%%%%%%%%%%%%%%%%%%%
\subsection*{Outline of the proof of Theorem~\ref{l:theorem-strong}}
(See \S\ref{s:prelim} for a review of semiclassical notation and terminology.)
We begin with Vasy's construction~\cite{v1,v2} 
of a Fredholm semiclassical pseudodifferential operator $P(z) - iQ$, 
on a compact 
manifold $X$, with $M$  diffeomorphic to an open subset of $X$, and
 such that the resonances of the semiclassical Laplacian on
$M$ are a subset of the poles of $(P(z) - iQ)^{-1}$; here $1+E=h^2(n-1)^2/4+(1+z)^2$.
The operator $Q$ is
 supported away
from $\overline M \subset X$, and $P(z)$ is a differential operator
such that $P(z)|_M$ is $\Delta_{g}$ conjugated and weighted (see 
\eqref{e:p1def}).
We review this construction in \S\ref{s:ah.vasy} and summarize the main 
results in Lemma~\ref{l:vasy-outline}; see Figure~\ref{f:vasy} for
the global dynamics of the corresponding Hamiltonian system.
 We will show that for any $C_0>0$
there is $h_0>0$ such that $(P(z) - iQ)^{-1}$ has  $\mathcal O(h^{-\nu})$ many poles in 
$\{|\Real z| \le C_0h,\, |\Imag z| \le C_0h\}$ for $h \in (0,h_0]$.

To do this we introduce the conjugated operator
\[
P_t(z)=e^{-tF}T_s(P(z)-iQ)T_s^{-1}e^{tF},
\]
where  $t>0$ is a large parameter (independent of $h$) and $T_s, F$ are 
 pseudodifferential operators on $X$, with
$F$ in an exotic calculus. We will show that this operator is invertible up
to a remainder of finite rank $\mathcal O(h^{-\nu})$, and then conclude
using Jensen's formula (\S\ref{s:outline}).
This follows from the estimate
\begin{equation}\label{e:intro-main}
\|u\|_{\Hh^{1/2}(X)} \le Ch^{-1} \|(P_t(z) - ithA)u\|_{\Hh^{-1/2}(X)},
\end{equation}
where $A = A_R + A_E$, with  $A_R$ of rank 
$\mathcal O(h^{-\nu})$ and $\|A_E\|_{\Hh^{1/2} \to \Hh^{-1/2}} = \mathcal O(\tilde h)$, where $0<\tilde h \ll 1$
depends on $t$ but not $h$.
This, \eqref{e:intro-main}, is the
 main estimate of the paper (see \eqref{e:main-estimate}).

To prove \eqref{e:intro-main}, we begin with the Taylor expansion
\begin{equation}\label{e:taylorexpintro}
P_t(z) = T_s(P(z)-iQ)T_s^{-1} + t[T_s(P(z)-iQ)T_s^{-1},F] + \mathcal O_t(h\tilde h),
\end{equation}
(this is \eqref{e:conj-dec}). We then use a microlocal partition of unity 
(Lemma~\ref{l:localization}) to divide the phase space $\overline T^*X$
into regions where the principal symbols of the various terms are elliptic.
We construct $T_s$, $F$, and $A$ such that these regions cover $\overline T^*X$, after
which we prove \eqref{e:intro-main} using a positive commutator argument. In fact, because
$P_t(z) - ithA = P(z) - iQ + \mathcal O(h\log(1/h))$, it suffices
to check that the intersection $\{\lxir^{-2}p=0\}\cap\{\lxir^{-2}q=0\}$ of the characteristic sets
of $P(0)$ and $Q$ is covered.

More specifically, $T_s$ is an elliptic pseudodifferential operator whose order $s$
is large enough that the principal symbol of $h^{-1} \Imag T_s(P(z)-iQ)T_s^{-1}$
is elliptic  near the
intersection of $\{\lxir^{-2}p =0\}$ with the fiber infinity of $\overline T^*X$.
We include a neighborhood of the \textit{radial points} of $P(z)$ (i.e. the 
fixed points of
its symbol's Hamiltonian vector field) 
 and this (\S\ref{s:ah.escape})
is  where our analysis near the spatial infinity of $M$ is different from 
that of \cite{v1,v2}. 
The ellipticity of the principal
symbol  of $h^{-1} \Imag T_s(P(z)-iQ)T_s^{-1}$ (with a favorable sign) is proved in
 Lemma~\ref{l:radial-conj}, and it is
used in the positive commutator argument in Lemma~\ref{l:estimate-l-pm}.
 
The operator $F$ has the form
\[
F = \widehat F + M \log(1/h) F_0,
\]
where $M>0$ and  $F_0, \widehat F$ are quantizations of {\em escape functions}, that is functions
monotonic along bicharacteristic flowlines of $P(0)$ in certain regions of $\{\lxir^{-2}p =0\}$. 
For $F_0$ this monotonicity holds outside of a fixed neighborhood of the radial points 
and
of the trapped set $\iota(K)$, 
(here $\iota\colon T^*M \to \overline 
T^*X$ is the slightly modified inclusion map 
defined in \eqref{e:iotadef}; note that no function can be monotonic 
along bicharacteristic flowlines
at the radial points or at $\iota(K)$)
giving ellipticity of the
principal symbol of $h^{-1} \Imag [T_s(P(z)-iQ)T_s^{-1},F_0]$ in that region.
For $\widehat F$ this monotonicity holds at points in a fixed neighborhood of $\iota(K)$
which lie
at least $ \sim (h/\tilde h)^{1/2}$ away from $\iota(K)$ (recall $\tilde h$ 
is small but independent
of $h$)%
\footnote{We need $\tilde h$ because if a symbol's derivatives  grow
like $h^{-1/2}$, then its quantization does not have an asymptotic expansion in powers of $h$.
As we show in \S\ref{s:1/2}, including $\tilde h$ gives an expansion in powers of $\tilde h$.}.
We need the monotonicity up to a neighborhood of $\iota(K)$ which is as small
as possible so that the correction term $A$, needed for  the 
global estimate \eqref{e:intro-main},
can be microsupported in a small
enough set that $A$ is of rank $\mathcal O(h^{-\nu})$ plus a small remainder.
The escape function for $\widehat F$ is taken  directly
from \cite{sj-z}, and it is here that we use the assumption that the 
geodesic flow is hyperbolic on $K$. The escape functions are constructed in 
Lemmas~\ref{l:escape} and~\ref{l:f-hat}.

The operator $A$ has the form
\begin{equation}\label{e:pttildeintro}
A = \chi((h/\tilde h)\widehat P) \widetilde A,
\end{equation}
where $\widehat P$ is an elliptic self-adjoint operator whose principal symbol
agrees with that of $P(0)$ near $\iota(K)$, $\chi \in C_0^\infty(\mathbb R)$ is
$1$ near $0$, 
and $\widetilde A$ is microsupported in a neighborhood of size 
 $\mathcal O(h/\tilde h)^{1/2}$ of $\iota(K)$. The operator $\widetilde A$ is
 pseudodifferential in the exotic class of \S\ref{s:1/2}, but 
 $\chi((h/\tilde h)\widehat P)$ is not even pseudolocal: it propagates
 semiclassical singularities along bicharacteristics of $\widehat P$
 (which near the microsupport of $\widetilde A$ are the same
 as bicharacteristics of $P(0)$). This type of operator is treated in \cite{sj-z}
 using a \emph{second microlocal} pseudodifferential calculus.
 In \S\ref{s:2nd} we use a more basic approach based on the Fourier
 inversion formula,
 \[
 \chi((\tilde h/h)\widehat P) = \frac 1 {2\pi} \int \hat \chi(t) e^{it(h/\tilde h)\widehat P} dt,
 \]
and functional calculus to prove only the  few microlocal properties of $ \chi((\tilde h/h)\widehat P)$
we need.
We use them in the positive commutator argument in 
Lemma~\ref{l:estimate-k-aux}, and we show that $A = A_R + A_E$, 
with $\rank A_R = \mathcal O(h^{-\nu})$ and $\|A_E\| = \mathcal O(\tilde h)$,
in Lemma~\ref{l:approximation} and \S\ref{s:ultimate3}.

\subsection*{Structure of the paper}

\begin{itemize}
\item In \S\ref{s:outline} we state the main properties of the extended manifold $X$
and the modified Laplacian $P(z) - iQ$ from \cite{v1,v2} in Lemma~\ref{l:vasy-outline} and then
use Jensen's formula to  reduce the proof
of Theorem~\ref{l:theorem-strong} to the
main lemma, Lemma~\ref{l:main}. 
\item In \S\ref{s:prelim} we review the
notation used in the paper and properties of semiclassical pseudodifferential
and Fourier integral operators. 
\item In \S\ref{s:ah.vasy} we review the
construction of  $X$ and  $P(z) - iQ$ and the proof of
Lemma~\ref{l:vasy-outline}.
In \S\ref{s:ah.escape} we introduce the conjugation by $T_s$ and prove
 estimates near the radial points, and then define the escape function $f_0$
used to separate the analysis near infinity from the analysis near the trapped set.
\item In \S\ref{s:1/2} we review the part of the exotic $\Psi_{1/2}$ 
 calculus of
\cite{sj-z} needed here. In \S\ref{s:2nd} we 
study microlocal properties of operators of the form $\chi((h/\tilde h)\widehat P)$ 
as in \eqref{e:pttildeintro}.
\item In \S\ref{s:approximation} we prove that operators of
the form of $A$ in \eqref{e:pttildeintro} can
be written $A = A_R + A_E$, with 
$\|A_E\|_{L^2 \to L^2} = \mathcal O(\tilde h)$ and $\rank A_R = \mathcal 
O_{\tilde h}(h^{-\nu})$.
\item In \S\ref{s:ultimate} we prove Lemma~\ref{l:main}.
In \S\ref{s:ultimate0} we use the results of \S\S\ref{s:prelim}--\ref{s:exotic} to 
prove positive commutator estimates for the modified conjugated operator
$P_t(z) - ithA$. In \S\ref{s:ultimate1} and 
\S\ref{s:ultimate2} we use  these to prove
semiclassical resolvent estimates, and
in \S\ref{s:ultimate3} we apply the results of \S\ref{s:approximation} to the
operator $A$ from \eqref{e:pttildeintro}.
\item In Appendix~\ref{s:quasifuchsian} we 
construct the quasifuchsian group in Figure~\ref{f:puredim}.
\item In Appendix~\ref{s:picture} we give a direct presentation of Vasy's construction 
from \S\ref{s:ah.vasy}, and of the escape functions from
Lemmas~\ref{l:escape} and~\ref{l:f-hat},  when $M$ is the hyperbolic cylinder.

\end{itemize}
For the convenience of the reader interested in our approach to second
microlocalization but not in the analysis near infinity, \S\ref{s:2nd} is independent 
of \S\ref{s:ah} (as are \S\ref{s:1/2} and \S\ref{s:approximation}).

%%%%%%%%%%%%%%%%%%%%%%%%%%%%%%%%%%%%%%%%%%%%%%%%%%%%%%%%%%%%%%%%%%%%%%%%%%%%%%%%
%                                   SECTION 2                                  %
%%%%%%%%%%%%%%%%%%%%%%%%%%%%%%%%%%%%%%%%%%%%%%%%%%%%%%%%%%%%%%%%%%%%%%%%%%%%%%%%
\section{Proof of Theorem~\ref{l:theorem-strong}}
  \label{s:outline}

We start by reviewing in Lemma~\ref{l:vasy-outline} Vasy's recent description~\cite{v2} of the
scattering resolvent and resonances on an asymptotically hyperbolic even
space $(M,g)$; this allows us to replace the  Laplacian on $M$ by a Fredholm pseudodifferential operator $P(z) - iQ$ on a compact manifold $X$, adapted to proving semiclassical estimates. See \S\ref{s:ah.vasy} for
details.
%
%%%%%%%%%%%%%%%%%%%%%%%%%%%%%% BEGIN PROP %%%%%%%%%%%%%%%%%%%%%%%%%%%%%%
%
\begin{lemm}
  \label{l:vasy-outline}
Assume that $(M,g)$ is asymptotically hyperbolic and even.
Then there exist a compact manifold $X$ without boundary, an order 2 semiclassical
differential operator $P(z)$ depending holomorphically on  $z$,
and an order 2 semiclassical pseudodifferential operator $Q$
depending holomorphically on $z$ such that for $h$ small enough,
\begin{enumerate}
\item for $\Imag z>-C_0h$ and $s>C_0$, 
the family of operators
$$
P(z)-iQ:\{u\in \Hh^{s+1/2}(X)\mid P(0)u\in \Hh^{s-1/2}(X)\}\to \Hh^{s-1/2}(X)
$$
is Fredholm of index zero, and a preimage of a smooth function under this operator
is again smooth.
Here the domain is equipped with the norm
$\|u\|=\|u\|_{\Hh^{s+1/2}}+\|P(0)u\|_{\Hh^{s-1/2}}$. The inverse
$$
(P(z)-iQ)^{-1}:\Hh^{s-1/2}(X)\to \Hh^{s+1/2}(X)
$$
is meromorphic in $z \in \{\Imag z>-C_0h\}$ with poles of finite rank;
\item the set of poles of $(P(z)-iQ)^{-1}$ in $\{\Imag z>-C_0h\}$ contains (including multiplicities) the set of poles of the continuation of 
\[
(h^2(\Delta_g - (n-1)^2/4) - (z+1)^2)^{-1} \colon L_{\comp}^2(M) \to L^2_{\loc}(M)
\]
from $\{\Imag z>0\}$ to $\{\Imag z>-C_0h\}$.
\end{enumerate}
\end{lemm}
%
%%%%%%%%%%%%%%%%%%%%%%%%%%%%%%% END PROP %%%%%%%%%%%%%%%%%%%%%%%%%%%%%%%
%
To prove Theorem \ref{l:theorem-strong} we will apply Lemma \ref{l:vasy-outline} (and Lemma \ref{l:main} below) with $C_0$ a large constant multiple of $c_0$. Throughout the paper we will work in the domain $\{|\Real z| \le C_0 h, \ \Imag z \ge - C_0h\}$. In the main lemma, we define a modified, conjugated operator $\widetilde P_t(z)$, and prove semiclassical estimates for it. See Section~\ref{s:prelim.basics} for the semiclassical notation.
%
%%%%%%%%%%%%%%%%%%%%%%%%%%%%%% BEGIN PROP %%%%%%%%%%%%%%%%%%%%%%%%%%%%%%
%
\begin{lemm}
  \label{l:main} (Main lemma)
Assume that $(M,g)$ and $\nu$ satisfy the assumptions
of Theorem~\ref{l:theorem-strong}. Let $X,P(z),Q, C_0$ be as in Lemma~\ref{l:vasy-outline}, and let
$T_s\in\Psi^s(X)$ be any (semiclassically) elliptic operator.
We introduce a parameter $\tilde h>0$; the estimates below hold for $\tilde h$ small enough
and $h$ small enough depending on $\tilde h$, and the constants in these estimates
are independent of $\tilde h$, $h$, and $z$ in the specified range, except for $C(\tilde h)$ below.

Then there exist $t>0$ and compactly microlocalized polynomially bounded
operators $A,F$ depexnding on $\tilde h$, with
$e^{\pm tF}-1$ polynomially bounded and compactly microlocalized and
%%%%%%%%%%%%%%%%%%%%%%%%%%%%%%
\begin{enumerate}
  \item the modified conjugated operator
\begin{equation}\label{e:p-t-tilde-intro}
\widetilde P_t(z)=e^{-tF}T_s(P(z)-iQ)T_s^{-1}e^{tF}-ithA
\end{equation}
satisfies the estimate
\begin{equation}\label{e:main-estimate}
|\Real z|\leq C_0h,\
|\Imag z|\leq C_0h,\
u\in C^\infty(X)\Longrightarrow \|u\|_{\Hh^{1/2}}\leq Ch^{-1}\|\widetilde P_t(z) u\|_{\Hh^{-1/2}};
\end{equation}
  \item if $\tilde h,\varepsilon$ are small enough and $h$ is small enough depending on $\tilde h,\varepsilon$,
then we have the improved estimate in the upper half-plane:
\begin{equation}\label{e:improved-estimate}
|\Real z|\leq C_0h,\
C_0h\leq\Imag z\leq\varepsilon,\
u\in C^\infty(X)\Longrightarrow
\|u\|_{\Hh^{1/2}}\leq {C\over \Imag z}\|\widetilde P_t(z)u\|_{\Hh^{-1/2}};
\end{equation}
  \item we can write $A=A_R+A_E$, where $A_R,A_E$
  are compactly microlocalized and for some constant $C(\tilde h)$ independent of $h$,
$$
\|A_R\|_{\Hh^{1/2}\to \Hh^{-1/2}}=\mathcal O(1),\
\|A_E\|_{\Hh^{1/2}\to \Hh^{-1/2}}=\mathcal O(\tilde h),\
\rank A_R\leq C(\tilde h) h^{-\nu}.
$$
\end{enumerate}
%%%%%%%%%%%%%%%%%%%%%%%%%%%%%%
\end{lemm}
%
%%%%%%%%%%%%%%%%%%%%%%%%%%%%%%% END PROP %%%%%%%%%%%%%%%%%%%%%%%%%%%%%%%
%
The conjugations by $T_s$ and $F$ modify $P(z) - iQ$ to make it semiclassically elliptic away from the trapped set, without disturbing the poles of the resolvent. The parameter $t$ is a coupling constant for the positive commutator term $t[T_s(P(z)-iQ)T_s^{-1},F]$ (see \eqref{e:taylorexpintro} above and see \S\ref{s:ultimate0} below for details), taken to be large enough that it overcomes the error term arising from the unfavorable sign of $\Imag z$. The  correction term $-ithA$ (which makes $\widetilde P_t(z)$ invertible) is compactly microlocalized in a small enough neighborhood of the trapped set in the energy surface that it can be approximated by an operator of rank $\mathcal O(h^{-\nu})$; it does affect the poles of the resolvent, but as we will see below it can remove no more than $\mathcal O(h^{-\nu})$ of them in the set $\{|\Real z| < C_0h/2, \ |\Imag z| < C_0h/2\}$.
%
%%%%%%%%%%%%%%%%%%%%%%%%%%%%%% BEGIN PROOF %%%%%%%%%%%%%%%%%%%%%%%%%%%%%%
%
\begin{proof}[Proof of Theorem~\ref{l:theorem-strong} assuming Lemmas~\ref{l:vasy-outline} and~\ref{l:main}]
We follow~\cite[\S6.1]{sj-z}.
Fix $s>C_0$. For $\Imag z \ge-C_0h$, the operators
$$
P_t(z)=e^{-tF}T_s(P(z)-iQ)T_s^{-1}e^{tF}
$$
and $\widetilde P_t(z)$ are Fredholm of index zero
$$
\{u\in \Hh^{1/2}\mid P(0)u\in \Hh^{-1/2}\}\to \Hh^{-1/2}.
$$

For $P_t(z)$, this follows immediately from Lemma~\ref{l:vasy-outline}
as $e^{tF}-1$ is compactly microlocalized and thus $e^{tF}$ preserves
Sobolev spaces, while $T_s^{-1}$ maps $\Hh^{-1/2}\to \Hh^{s-1/2}$ and
$\Hh^{s+1/2}\to \Hh^{1/2}$. (Note however that the norm of $e^{tF}$
grows as $h \to 0$.) For $\widetilde P_t(z)$, we additionally use that $A$
is compactly microlocalized and thus a compact perturbation of
$P_t(z)$.

Let $\varepsilon>0$ be a small constant and consider the rectangle
$$
R_0=\{\Real z\in [-C_0h,C_0h],\ \Imag z\in [-C_0h,\varepsilon]\}.
$$
It follows from parts~1 and~2 of Lemma~\ref{l:main} that $\widetilde
P_t(z)^{-1}$ has no poles in $R_0$ and satisfies
$$
z\in R_0\Longrightarrow \|\widetilde P_t(z)^{-1}\|_{\Hh^{-1/2}\to \Hh^{1/2}}\leq {C\over \max(h,\Imag z)}.
$$
We now use the decomposition~$A=A_R+A_E$ from~Lemma~\ref{l:main}(3).
Since $\|A_E\|=\mathcal O(\tilde h)$, for $\tilde h$ small enough 
$\|(-ithA_E)\widetilde P_t(z)^{-1} \|_{\Hh^{-1/2}\to\Hh^{-1/2}}\leq 1/2$ and we get
\begin{equation}\label{e:maine}
z\in R_0\Longrightarrow \|(P_t(z)-ithA_R)^{-1}\|_{\Hh^{-1/2}\to \Hh^{1/2}}\leq {C\over \max(h,\Imag z)}.
\end{equation}
Now,
$$
(P_t(z)-ithA_R)^{-1}P_t(z)=1+ith(P_t(z)-ithA_R)^{-1}A_R=:1+K(z),\
z\in R_0.
$$
Therefore, the poles of $P_t(z)^{-1}$ (and hence of $(P(z)-iQ)^{-1}$) are
contained, including multiplicities, in the zeros of
$$
k(z)=\det(1+K(z));
$$
indeed, see \cite[Proposition 5.16]{sres} for a general statement and
see \cite[\S D.1]{e-z} for a discussion of the theory of Grushin problems
which is used there.

By Lemma~\ref{l:main}(3), $K(z)$ has norm $\mathcal O(1)$ and
rank $\mathcal O(h^{-\nu})$. Therefore,
\begin{equation}\label{e:k-upper}
|k(z)|\leq e^{Ch^{-\nu}},\
z\in R_0.
\end{equation}
Since $\|ithA_R\|=\mathcal O(h)$, we have $\|K(z_0)\|\leq 1/2$ for $z_0=iC_1h$ with $C_1>0$ large enough,  so
\begin{equation}\label{e:k-lower}
|k(z_0)|\geq e^{-Ch^{-\nu}}.
\end{equation}
Define rectangles $R_2\subset R_1\subset R_0$ by
$$
\begin{gathered}
R_1=\{\Real z\in (-C_0h,C_0h),\
\Imag z\in (-C_0h, 4C_1h)\},\\
R_2=\{\Real z\in (-C_0h/2,C_0h/2),\
\Imag z\in (-C_0h/2, 2C_1h)\}.
\end{gathered}
$$
Then the estimate~\eqref{e:k-upper} holds in $R_1$, while $z_0\in
R_2$. By the Riemann mapping theorem, there exists a unique conformal
map $w(z)$ from $R_1$ onto the ball $B_1=\{|w|<1\}$ with $w(z_0)=0$
and $w'(z_0)\in \mathbb R^+$. We then see that $w(R_2)\subset
B_2=\{|w|<1-\delta\}$ for some $\delta>0$ independent of $h$ (as
$h^{-1}R_1$ and $h^{-1}R_2$ do not depend on $h$). One can now apply
Jensen's formula (see for example~\cite[\S3.61, equation
(2)]{tit}) to the function $k_1(w)=k(z(w))$ on $B_1$: if $n(r)$ is the
number of zeros of $k_1$ in $\{|w|<r\}$, then
$$
\int_0^{1-\delta/2} {n(r)\over r}\,dr
={1\over 2\pi}\int_0^{2\pi}\log|k_1((1-\delta/2)e^{i\theta})|\,d\theta
-\log|k_1(0)|\leq 2Ch^{-\nu}.
$$
Therefore,
$$
n(1-\delta)\leq {2\over\delta}\int_{1-\delta}^{1-\delta/2}{n(r)\over r}\,dr
=\mathcal O(h^{-\nu}).
$$
This estimates the number of zeroes of $k_1(w)$ in $B_2$, thus the number
of zeroes of $k(z)$ in $R_2$, and thus the number of resonances in $R_2$,
as needed.
\end{proof}
%
%%%%%%%%%%%%%%%%%%%%%%%%%%%%%%% END PROOF %%%%%%%%%%%%%%%%%%%%%%%%%%%%%%%
%

%%%%%%%%%%%%%%%%%%%%%%%%%%%%%%%%%%%%%%%%%%%%%%%%%%%%%%%%%%%%%%%%%%%%%%%%%%%%%%%%
%                                   SECTION 3                                  %
%%%%%%%%%%%%%%%%%%%%%%%%%%%%%%%%%%%%%%%%%%%%%%%%%%%%%%%%%%%%%%%%%%%%%%%%%%%%%%%%
\section{Semiclassical preliminaries}
  \label{s:prelim}

%%%%%%%%%%%%%%%%%%%%%%%%%%%%%%%%%%%%%%%%%%%%%%%%%%%%%%%%%%%%%%%%%%%%%%%%%%%%%%%%
%                                  SECTION 3.1                                 %
%%%%%%%%%%%%%%%%%%%%%%%%%%%%%%%%%%%%%%%%%%%%%%%%%%%%%%%%%%%%%%%%%%%%%%%%%%%%%%%%
\subsection{Notation and pseudodifferential operators}
  \label{s:prelim.basics}

In this section, we review certain notions of semiclassical analysis;
for a comprehensive introduction to this area, the reader is referred
to~\cite{e-z} or~\cite{d-s}. We will use the notation
of~\cite[\S2]{v2}, with some minor changes. Consider a (possibly noncompact) manifold
$X$ without boundary.  Following~\cite{mel2}, the symbols we consider will be defined on the
\emph{fiber-radial compactification} $\overline T^*X$ of the cotangent
bundle; its boundary, called the fiber infinity, is associated with
the spherical bundle $S^*X$ and its interior is associated with
$T^*X$. Denote by $(x,\xi)$ a typical element of $\overline T^*X$;
here $x\in X$ and $\xi$ is in the radial compactification of $T^*_x
X$.  We fix a smooth inner product on the fibers of $T^*X$; if
$|\cdot|$ is the norm on the fibers generated by this inner product,
let
$$
\lxir=(1+|\xi|^2)^{1/2}.
$$
Then $\lxir^{-1}$ is a boundary defining function on $\overline T^*X$.
(The smooth structure of $\overline T^*X$ is independent of the choice
of the inner product.)

Let $k\in \mathbb R$. A smooth function $a(x,\xi)$ on $T^*X$ is called
a classical symbol of order $k$, if $\lxir^{-k} a$ extends to a smooth
function on $\overline T^*X$. We denote by $S^k_{\cl}(X)$ the algebra
of all classical symbols.  If $a$ also depends on the semiclassical
parameter $h>0$, it is called a classical semiclassical symbol of
order $k$, if there exists a sequence of functions $a_j(x,\xi)\in
S^{k-j}_{\cl}(X)$, $j=0,1,\dots$, such that $a\sim\sum_j h^ja_j$ in
the following sense: for each $J$, the function
$h^{-J}\lxir^{J-k}\big(a-\sum_{j<J} h^j a_j\big)$ extends to a smooth
function on $\overline T^*X\times [0,h_0)$, for $h_0>0$ small
enough. In this case, $a_0$ is called the (semiclassical)
\emph{principal part} of $a$: it captures $a$ to leading order in both $h$ and $\langle \xi \rangle^{-1}$. The semiclassical symbol $a$ is
classical if and only if $\tilde a=\lxir^{-k}a$ extends to a smooth
function on $\overline T^*X\times [0,h_0)$, and for each differential
operator $\partial^j$ of order $j$ on $\overline T^*X$, the
restriction of $\partial^j\tilde a$ to $S^*X$ is a polynomial of
degree no more than $j$ in $h$.

For real-valued $a\in S^k_{\cl}(X)$, we denote by $H_a$ the
Hamiltonian vector field generated by $a$ with respect to the standard
symplectic form on $T^*X$.  Then $\lxir^{1-k}H_a$ can be
extended to a smooth vector field on $\overline T^*X$ and this
extension preserves the fiber infinity $S^*X$.

The class $S^k_{\hbar,\cl}(X)$ of classical semiclassical symbols is
closed under the standard operations of semiclassical symbol calculus
(multiplication, adjoint, change of coordinates); therefore, one can
consider the algebra $\Psi^k(X)$ of semiclassical pseudodifferential
operators with symbols in $S^k_{\hbar,\cl}(X)$. (See~\cite[Chapters~4,
9, and~14]{e-z} for more information.) We do not use the
$\Psi^k_{\cl}(X)$ notation, as we will only operate with classical
operators until \S\ref{s:1/2}, where a different class will be
introduced. We require that all elements of $\Psi^k$ be properly
supported operators, so that they act $C^\infty(X)\to C^\infty(X)$ and
$C_0^\infty(X)\to C_0^\infty(X)$; in particular, we can multiply two
such operators.  If $\Hh^s(X)$ denotes the semiclassical Sobolev space
of order~$s$ (recall that when $X$ is compact, one choice of norm is $\|u\|_{\Hh^s(X)} = \|(1 + h^2\Delta_X)^{s/2}u\|_{L^2(X)}$, where $\Delta_X$ is the Laplacian on $X$), then each $A\in\Psi^k$ is bounded
$H^s_{\hbar,\comp}(X)\to H^{s-k}_{\hbar,\comp}(X)$ uniformly in
$h$. Here $H^s_{\hbar,\comp}$ consists of compactly supported
distributions lying in $\Hh^s$; in this article, we will mostly work
on compact manifolds, where this space is the same as $\Hh^s$.

For $A\in\Psi^k$, the principal part $\sigma(A)$ of the symbol of $A$
is independent of the quantization procedure. We call $\sigma(A)$ the
\emph{principal symbol} of $A$, and $\sigma(A)=0$ if and only if
$A\in h\Psi^{k-1}$. We remark that $\sigma(A)$ is the restriction of the standard semiclassical principal symbol on the larger standard semiclassical algebra. The principal symbol enjoys the multiplicativity
property $\sigma(AB)=\sigma(A)\sigma(B)$ and the commutator identity
$\sigma(h^{-1}[A,B])=-i\{\sigma(A),\sigma(B)\}$, where
$\{\cdot,\cdot\}$ is the Poisson bracket.

We define the (closed) semiclassical \emph{wavefront set}
$\WFh(A)\subset \overline T^*X$ as follows: if $a$ is the full symbol
of $A$ in some quantization, then $(x_0,\xi_0)\not\in\WFh(A)$ if and
only if there exists a neighborhood $U$ of $(x_0,\xi_0,h=0)$ in
$\overline T^*X\times [0,h_0)$ such that for each $N$, $h^{-N}\lxir^N
a$ is smooth in $U$. Since the operations of semiclassical symbol
calculus are defined locally modulo $\mathcal
O(h^\infty\lxir^{-\infty})$, this definition does not depend on the
choice of quantization, and we also have $\WFh(AB)\subset
\WFh(A)\cap\WFh(B)$ and $\WFh(A^*)=\WFh(A)$.  We say that $A=B$
\emph{microlocally} in some set $V\subset \overline T^*X$ if $\WFh(A-B)\cap
V=\emptyset$. If $\WFh(A)$ is a compact subset of $T^*X$, and in
particular does not intersect the fiber infinity $S^*X$, then we call
$A$ \emph{compactly microlocalized}. Denote by $\Psic(X)$ the class of
all compactly microlocalized pseudodifferential operators; these operators lie in
$\Psi^k(X)$ for all $k$. Note that for a noncompact $X$, compactly
microlocalized operators need not have compactly supported Schwartz
kernels.

The wavefront set of $A\in\Psi^k$ is empty if and only if $A \in
h^N\Psi^{-N}$ for all $N$. In this case, we write $A=\Resh$.  This
property can also be stated as follows: $A$ is properly supported,
smoothing, and each $C^\infty$ seminorm of its Schwartz kernel is
$\mathcal O(h^\infty)$. With  this restatement, 
  $B=\Resh$ makes sense even for $B:\mathcal E'(X_2)\to
C^\infty(X_1)$ with $X_1 \ne X_2$.

We now explain our use of the $\mathcal O(\cdot)$
notation. For an operator $A$, we write
\begin{equation}
  \label{e:O-notation}
A=\mathcal O_Z(f)_{\mathcal X\to \mathcal Y},
\end{equation}
where $\mathcal X,\mathcal Y$ are Hilbert spaces, $Z$ is some list of
parameters, and $f$ is an expression depending on $h$ and perhaps some
other parameters, if the $\mathcal X\to \mathcal Y$ operator norm of
$A$ is bounded by $Cf$, where $C$  depends on $Z$.
Instead of $\mathcal X\to \mathcal Y$, we may put some class of
operators; for example, $A=\mathcal O_Z(f)_{\Psi^k}$ means that for
any fixed value of $Z$, the operator $f^{-1}A$ lies in $\Psi^k$. This
is stronger than estimating some norms of the full symbol of $A$,
 as the classes $\Psi^k$ are not preserved under
multiplication by functions of $h$ that are not polynomials.

One can also define wavefront sets for
operators that are not pseudodifferential. Let $X_1$ and $X_2$ be two
manifolds. A properly supported operator $B:C^\infty(X_2)\to \mathcal
D'(X_1)$ is called \emph{polynomially bounded}, if for each $\chi_j\in
C_0^\infty(X_j)$ and each $s$, there exists $N$ such that
$\chi_1B\chi_2$ is $\mathcal O(h^{-N})$ as an operator $\Hh^s\to
\Hh^{-N}$ and $\Hh^N\to\Hh^s$.  A product of such $B$ with an operator
that is $\Resh$ will also be $\Resh$.  For a polynomially bounded
$B:\mathcal D'(X_2)\to C_0^\infty(X_1)$, its wavefront set
$\WFh(B)\subset \overline T^*X_1\times
\overline T^*X_2$ is defined as follows: a point $(x_1,\xi_1,x_2,\xi_2)$ does
not lie in $\WFh(B)$, if there exist neighborhoods $U_j$ of
$(x_j,\xi_j)$ in $\overline T^*X_j$ such that for each
$A_j\in\Psi^{k_j}(X_j)$ with $\WFh(A_j)\subset U_j$, we have
$A_1BA_2=\Resh$.

If $X_1=X_2=X$, then we call a polynomially bounded operator $B$
\emph{pseudolocal} if its wavefront set is a subset of the diagonal;
in this case we consider $\WFh(B)$ as a subset of
$\overline T^*X$. Pseudolocality is equivalent to saying that for any
$A_j\in\Psi^{k_j}(X)$ with $\WFh(A_1)\cap \WFh(A_2)=\emptyset$, we have
$A_1BA_2=\Resh$.  The operators in $\Psi^k$ are pseudolocal and the
definition of their wavefront set given here agrees with the one given
earlier; however, the definition in this paragraph can also be applied
to operators with exotic symbols of \S\ref{s:1/2}.

A polynomially bounded operator $B$ is called \emph{compactly
microlocalized}, if its wavefront set is a compact subset of
$T^*X_1\times T^*X_2$, and in particular does not intersect the fiber
infinities $S^*X_1\times \overline T^*X_2$ and $\overline T^*X_1\times
S^*X_2$.  In this case, if $A_j\in\Psic(X_j)$ are equal to the
identity microlocally near the projections of $\WFh(B)$ onto
$\overline T^*X_j$, then $B=A_1BA_2+\Resh$. Every compactly
microlocalized operator $B$ is smoothing; we say that $B=\mathcal
O(h^r)$ for some $r$, if for each $\chi_j\in C_0^\infty(X_j)$, there
exist $s,s'$ such that
\begin{equation}
  \label{e:b-bound}
\|\chi_1B\chi_2\|_{\Hh^s\to \Hh^{s'}}=\mathcal O(h^r).
\end{equation}
In fact, if $B$ is compactly microlocalized and $B=\mathcal O(h^r)$,
then~\eqref{e:b-bound} holds for all $s,s'$.

Finally, we say that $A\in\Psi^k(X)$ is \emph{elliptic} (as an element
of $\Psi^k$) on some set $V\subset\overline T^*X$, if
$\lxir^{-k}\sigma(A)$ does not vanish on $V$. If $A$ is elliptic on
the wavefront set of some $B\in\Psi^{k'}(X)$, then we can find an
operator $W\in\Psi^{k'-k}(X)$ such that $B=WA+\Resh$. This implies the
following \emph{elliptic estimate}, formulated here for the case of a
compact $X$:
\begin{equation}
  \label{e:elliptic}
\|Bu\|_{\Hh^{s+k-k'}(X)}\leq C\|Au\|_{\Hh^s(X)}+\mathcal O(h^\infty)\|u\|_{\Hh^{-N}(X)},
\end{equation}
for all $s,N$ and each $u\in \mathcal D'(X)$ such that $Au\in\Hh^s$.

Still assuming $X$ compact, the \emph{non-sharp G\r arding inequality} says that if $A \in \Psi^k(X)$,   $\Psi_1 \in \Psi^0(X)$, $\lxir^{-k}\Real\sigma(A) > 0$ on $\WFh(\Psi_1)$, then
\begin{equation}\label{e:nonsharpg}
\Real \langle A \Psi_1u, \Psi_1 u \rangle \ge C^{-1} \|\Psi_1u\|_{\Hh^{k/2}(X)} -\mathcal O(h^\infty)\|u\|_{\Hh^{-N}(X)}.
\end{equation}
See Lemma~\ref{l:garding-1/2} for a proof of the corresponding statement in the $\Psie$ calculus; the same proof works here. Of course if $A$ is symmetric, we may drop $\Real$ from \eqref{e:nonsharpg}.

With $X$ still compact, the \emph{sharp G\r arding inequality} (see e.g. \cite[Theorem 9.11]{e-z})says that if instead $A \in \Psi^k(X)$,   $\Psi_1 \in \Psi^0(X)$, $\lxir^{-k}\Real\sigma(A) \ge 0$ near $\WFh(\Psi_1)$, then
\begin{equation}\label{e:sharpg}
\Real \langle A \Psi_1u, \Psi_1 u \rangle \ge -C h \|\Psi_1u\|_{\Hh^{(k-1)/2}(X)} -\mathcal O(h^\infty)\|u\|_{\Hh^{-N}(X)}.
\end{equation}

%%%%%%%%%%%%%%%%%%%%%%%%%%%%%%%%%%%%%%%%%%%%%%%%%%%%%%%%%%%%%%%%%%%%%%%%%%%%%%%%
%                                  SECTION 3.2                                 %
%%%%%%%%%%%%%%%%%%%%%%%%%%%%%%%%%%%%%%%%%%%%%%%%%%%%%%%%%%%%%%%%%%%%%%%%%%%%%%%%
\subsection{Quantization of canonical transformations}
  \label{s:prelim.canonical}

In this section, we discuss local quantization of symplectic
transformations.  The resulting semiclassical Fourier integral
operators will be needed to approximate the operator~$A$ from
Lemma~\ref{l:main} by finite rank operators; see
Lemma~\ref{l:main}(3) and  \S\ref{s:approximation1}.

The theory described below can be found in~\cite{ivana},
\cite[Chapter~6]{gui-st1}, \cite[Chapter~8]{gui-st2}, \cite[\S2.3]{svn},
or~\cite[Chapters~10--11]{e-z}. For the closely related microlocal
setting, see~\cite[Chapter~25]{ho4} or~\cite[Chapters~10--11]{gr-s}. We
follow in part~\cite[\S2.3]{zeeman}. Note that we will only need 
the relatively simple, local  part of the theory of Fourier integral
operators, as we quantize canonical transformations locally and we do
not use geometric invariance of the principal symbol. For a more complete discussion, see
for example~\cite[\S3]{d-g}.

Let $X_1,X_2$ be two manifolds of same dimension,
$U_j\subset T^*X_j$ two bounded open sets, and $\varkappa:U_1\to U_2$
a symplectomorphism. First, assume that
%%%%%%%%%%%%%%%%%%%%%%%%%%%%%%
\begin{itemize}
  \item $x_j$ are systems of coordinates
on the projections $\pi_j(U_j)$ of $U_j$ onto $X_j$;
  \item $(x_j,\xi_j)$ are the corresponding coordinates
on $U_j$;
  \item there exists a generating function $S(x_1,\xi_2)\in C^\infty(U_S)$
for some open $U_S\subset \mathbb R^{2n}$
such that in coordinates $(x_1,\xi_1,x_2,\xi_2)$,
the graph of $\varkappa$ is given by
$$
x_2=\partial_{\xi_2}S(x_1,\xi_2),\
\xi_1=\partial_{x_1}S(x_1,\xi_2).
$$
\end{itemize}
%%%%%%%%%%%%%%%%%%%%%%%%%%%%%%
Such coordinate systems and generating functions exist locally near
every point in the graph of $\varkappa$, see for example the paragraph
before the final remark of~\cite[Chapter~9]{gr-s}. (The authors of~\cite{gr-s}
consider the homogeneous case, but this does not make a difference here
except possibly at the points of the zero section of $T^*X_2$.
At these points, a different parametrization is possible and all the results
below still hold, but we do not present this parametrization since the
canonical transformations we use can be chosen to avoid the zero section.) An
operator $B:C^\infty(X_2)\to C_0^\infty(X_1)$ of the form
\begin{equation}
  \label{e:quantized-canonical}
(B f)(x_1)=(2\pi h)^{-n}\int_{\mathbb R^{2n}}  e^{i(S(x_1,\xi_2)
-x_2\cdot\xi_2)/h} b(x_1,\xi_2,x_2;h) f(x_2)\,dx_2
d\xi_2,
\end{equation}
where $b\in C_0^\infty(U_S\times \pi_2(U_2))$ is a classical symbol in $h$ (namely, it
is a smooth function of $h\geq 0$ up to $h=0$), is called
a local Fourier integral operator associated to $\varkappa$. Such an
operator is polynomially bounded and compactly microlocalized, in the sense
of \S\ref{s:prelim.basics}. 

In general, we call $B:C^\infty(X_2)\to C^\infty(X_1)$ a (compactly
microlocalized) Fourier integral operator associated to $\varkappa$,
if it can be represented as a finite sum of expressions of the
form~\eqref{e:quantized-canonical} with various choices of local
coordinate systems (and thus various generating functions) plus an
$\Resh$ remainder. Note that we use the convention that $B$ acts
$C^\infty(X_2)\to C^\infty(X_1)$, which is opposite to the more standard
convention that $B$ acts $C^\infty(X_1)\to C^\infty(X_2)$; in
the latter convention, we would say that $B$ quantizes
$\varkappa^{-1}$.

Here are some properties of (compactly microlocalized) Fourier integral operators:
%%%%%%%%%%%%%%%%%%%%%%%%%%%%%%
\begin{enumerate}
  \item if $B$ is a Fourier integral operator associated to $\varkappa$, then $\WFh(B)$
is a compact subset of the graph of $\varkappa$; 
  \item if $B$ is associated to $\varkappa$, then the adjoint $B^*$ is associated
to $\varkappa^{-1}$;
  \item if $X_1=X_2=X$, and $\varkappa$ is the identity map on some
open bounded $U\subset T^*X$, then $B$ is associated to $\varkappa$
if and only if $B\in\Psic(X)$ and $\WFh(B)\subset U$;
  \item\label{i:composition} if $B$ is associated to $\varkappa$ and $B'$ is associated
to $\varkappa'$, then $BB'$ is associated to $\varkappa'\circ\varkappa$;
  \item\label{i:bdd} if $B$ is associated to $\varkappa$, then it has norm $\mathcal O(1)$
in the sense of \S\ref{s:prelim.basics}.
This property follows from the previous three, as $B^*B\in\Psic(X_2)$.
  \item\label{i:egorov} (Egorov's theorem) If $A_j\in\Psi^{k_j}(X_j)$ and
$\sigma(A_1)=\sigma(A_2)\circ\varkappa$ near the projection
of $\WFh(B)$ onto $T^*X_1$, then $A_1B=BA_2+\mathcal O(h)$. Here
$\mathcal O(h)$ is understood in the sense of~\eqref{e:b-bound}, as
both sides of the equation are compactly microlocalized.
\end{enumerate}
%%%%%%%%%%%%%%%%%%%%%%%%%%%%%%
If $K_j\subset U_j$ are compact sets
such that $\varkappa(K_1)=K_2$, then we say that a pair of operators
$$
B:C^\infty(X_2)\to C^\infty(X_1),\
B':C^\infty(X_1)\to C^\infty(X_2)
$$
quantizes $\varkappa$ near $K_1\times K_2$, if
$B,B'$ are compactly microlocalized Fourier integral operators
associated to $\varkappa,\varkappa^{-1}$, respectively,
and
$$
BB'=1,\
B'B=1
$$
microlocally near $K_1$ and $K_2$, respectively.

Such a pair $(B,B')$ can  be found for any given $\varkappa$,
if we shrink $U_j$ sufficiently, as given by the following
construction of~\cite[Chapter~11]{e-z}.  First of all, we pass to
local coordinates to assume that $X_j=\mathbb R^n$.  Next,
by~\cite[Theorem~11.4]{e-z} (putting $\kappa=\varkappa^{-1},
\tilde\kappa_t=\varkappa_t^{-1},q_t=-z_t$), we can construct a smooth family of
symplectomorphisms $\varkappa_t,t\in [0,1]$ on $T^*\mathbb R^n$
such that
%%%%%%%%%%%%%%%%%%%%%%%%%%%%%%
\begin{itemize}
  \item $\varkappa_t$ is equal to the identity outside of some fixed compact set;
  \item $\varkappa_0=1_{T^*\mathbb R^n}$ and $\varkappa_1$ extends $\varkappa$;
  \item there exists a family of real-valued functions $z_t\in C_0^\infty(T^*\mathbb R^n)$
such that for each $t\in [0,1]$ and $H_{z_t}$ the Hamiltonian vector field of $z_t$,
\begin{equation}
  \label{e:fio-deformation}
(\partial_t\varkappa_t)\circ\varkappa_t^{-1}=H_{z_t}.
\end{equation}
\end{itemize}
%%%%%%%%%%%%%%%%%%%%%%%%%%%%%%
In other words, $\varkappa$ is a deformation of the identity along the Hamiltonian flow
of the time-dependent function $z_t$.

Let $Z_t$ be the Weyl quantization of $z_t$; this is a self-adjoint
operator on $L^2(\mathbb R^n)$. Consider the family of unitary
operators $B_t$ on $L^2(\mathbb R^n)$ solving the
equations~\cite[Theorem~10.1]{e-z}
\begin{equation}
  \label{e:fio-evolution}
h D_t B_t = B_t Z_t,\
B_0 = 1.
\end{equation}
By~\cite[Theorem~10.3]{e-z} (and using the composition
property~\eqref{i:composition} above to pass from small $t$ to $t=1$)
we see that, if $\Psi_j\in\Psic(X_j)$ with $\WFh(\Psi_j)\subset U_j$
and $\Psi_j=1$ microlocally near $K_j$, then
$(B,B')=(\Psi_1B_1\Psi_2,\Psi_2B_1^{-1}\Psi_1)$ quantizes $\varkappa$
near $K_1\times K_2$.

The difference between~\eqref{e:fio-evolution}
and~\cite[(10.2.1)]{e-z} is explained by the fact that we quantize a
transformation $T^*X_1\to T^*X_2$ as an operator $C^\infty(X_2)\to
C^\infty(X_1)$, while~\cite{e-z} quantizes it as an operator
$C^\infty(X_1)\to C^\infty(X_2)$; in the latter convention, 
$U(t)=B_t^*$ and $P(t)=-Z_t$.

%%%%%%%%%%%%%%%%%%%%%%%%%%%%%%%%%%%%%%%%%%%%%%%%%%%%%%%%%%%%%%%%%%%%%%%%%%%%%%%%
%                                   SECTION 3                                  %
%%%%%%%%%%%%%%%%%%%%%%%%%%%%%%%%%%%%%%%%%%%%%%%%%%%%%%%%%%%%%%%%%%%%%%%%%%%%%%%%
\section{Asymptotically hyperbolic manifolds}
  \label{s:ah}

%%%%%%%%%%%%%%%%%%%%%%%%%%%%%%%%%%%%%%%%%%%%%%%%%%%%%%%%%%%%%%%%%%%%%%%%%%%%%%%%
%                                  SECTION 3.1                                 %
%%%%%%%%%%%%%%%%%%%%%%%%%%%%%%%%%%%%%%%%%%%%%%%%%%%%%%%%%%%%%%%%%%%%%%%%%%%%%%%%
\subsection{Review of the construction of~\texorpdfstring{\cite{v2}}{[Va2]}}
\label{s:ah.vasy} 

Throughout this section, $|\Real z| \le C_0 h$ and $\Imag z  \ge -C_0h$.
Our notation will differ in several places from the one used in~\cite{v2}.
In particular, we use $(x,\xi)$ to denote coordinates on the whole $T^*X$,
$(\tilde x,\tilde y)$ for the product coordinates near the conformal
boundary of $M$, $\tilde\xi$ for the momentum corresponding to $\mu=\tilde x^2$,
and $\tilde\eta$ for the momentum corresponding to $\tilde y$. 

Let $(M,g)$ be an even asymptotically hyperbolic manifold;
we consider $\delta_0 > 0$ and a boundary defining function $\tilde x$ on $\overline M$ such that
$\{\tilde x<\delta_0^2\}\simeq [0,\delta_0^2)\times \partial \overline M$
and the metric $g$ has the form~\eqref{e:as-hyp}
for some metric $g_1$ depending smoothly on $\tilde x^2$. Consider the space
$\overline M_{\even}$, which is topologically $\overline M$, but with smooth structure at
the boundary changed so that
$$
\mu=\tilde x^2
$$
is a boundary defining function.
Now, we consider the modified Laplacian
\begin{equation}\label{e:p1def}
\begin{gathered}
P_1(z)=\mu^{-{1\over 2}-{n+1\over 4}}e^{i(z+1)\phi\over h}(h^2(\Delta_{g}-(n-1)^2/4)-(z+1)^2)
e^{-i(z+1)\phi\over h}\mu^{-{1\over 2}+{n+1\over 4}}.
\end{gathered}
\end{equation}
Here $\phi$ is a smooth real-valued function on $M$ such that
$$
e^{\phi}=\mu^{1/2}(1+\mu)^{-1/4}\text{ on }\{0<\mu<\delta_0\}.
$$
The function $\phi$ satisfies additional assumptions given in the proof of Lemma~\ref{l:im-part}.
This lemma, needed for the proof of the improved estimate in the upper half-plane~\eqref{e:improved-estimate},
is the only place where the factor $(1+\mu)^{-1/4}$ is needed; the rest of the analysis would work
if simply $e^\phi=\mu^{1/2}$ on all of $M$.

As computed in~\cite[(3.5)]{v2}, $P_1(z)$ has coefficients smooth up to the boundary
of $\overline M_{\even}$. For $\delta_0>0$ small enough,
this operator continues smoothly to
$$
X_{-\delta_0}=\{\mu>-\delta_0\},
$$
by which we mean $\overline M_{\even} \cup \{|\mu|<\delta_0\}$, where $ \{|\mu|<\delta_0\}$  is the double space of $\{ 0 \le \mu<\delta_0\}$.
(See~\cite[\S3.1]{v2} for more details.)

%
%%%%%%%%%%%%%%%%%%%%%%%%%%%%%% BEGIN PROP %%%%%%%%%%%%%%%%%%%%%%%%%%%%%%
%
\begin{lemm}\label{l:vasy-1}
There exists a manifold $X$ without boundary and a family of operators
$P(z)$ so that:
%%%%%%%%%%%%%%%%%%%%%%%%%%%%%%
\begin{enumerate}
  \item $P(z)\in\Psi^2(X)$ depends holomorphically on $z$ and
$p=\sigma(P(0))$ is real valued;
  \item $X_{-\delta_0}$ embeds into $X$ and $\mu$ continues to $X$ so that $X_{-\delta_0}=\{\mu>-\delta_0\}$;
  \item the restriction of $P(z)$ to $X_{-\delta_0}$
is equal to $P_1(z)$;
  \item the characteristic set $\{\lxir^{-2}p=0\}$
is the disjoint union of two closed sets $\Sigma_+$ and $\Sigma_-$,
with $\Sigma_\pm=\{\lxir^{-2}p=0\}\cap \{\pm\tilde \xi>0\}$ near $\{\mu=0\}\cap S^*X$;
  \item $\Sigma_+\cap S^*X\subset\{\mu\leq 0\}$ and $\Sigma_-\subset\{\mu\leq 0\}$;
  \item $P(z) - P(0) = (hL + a_0+a_1h)z  + a_2 z^2$, with the vector field
  $L$ and the functions $a_0,a_1,a_2$ independent of $z$ and $h$.
\end{enumerate}
%%%%%%%%%%%%%%%%%%%%%%%%%%%%%%
\end{lemm}
%%%%%%%%%%%%%%%%%%%%%%%%%%%%%%%%%%%%%%%%%%%%%%%%%%%%%%%%%%%%%%%%%%%%%%%%%%%%%%%%
\begin{proof}
The fact that the characteristic set of $p$ on $X_{-\delta_0}$ is
the disjoint union of two sets $\Sigma'_+$ and $\Sigma'_-$, with
$\Sigma'_\pm$ satisfying (4) and (5), is proven in~\cite[\S3.4]{v2},
with $\Sigma_\pm$ denoted by $\Sigma_{\hbar,\pm}$ there.
The manifold $X$ is taken to be the double space of $X_{-\delta_0}$; 
the extension of $P_1(z)$ and $\Sigma'_\pm$ to $X_{-\delta_0}$
is constructed in~\cite[\S3.5]{v2}. The formula (6) follows from~\eqref{e:p1def}.
\end{proof}
%
%%%%%%%%%%%%%%%%%%%%%%%%%%%%%%% END PROP %%%%%%%%%%%%%%%%%%%%%%%%%%%%%%%
%
Consider the product coordinates $(\mu,\tilde y)$ near $\{\mu=0\}$,
the corresponding momenta $(\tilde\xi,\tilde\eta)$, and define
\begin{equation}
  \label{e:l-pm}
L_\pm=\{\mu=0,\
\tilde\xi=\pm\infty,\
\tilde\eta=0\}\subset \Sigma_\pm\cap S^*X;
\end{equation}
they can be viewed as images of the conical sets $\{\mu=0,\ \pm\tilde\xi>0,\ \tilde\eta=0\}$
in $S^*M$.
%
%%%%%%%%%%%%%%%%%%%%%%%%%%%%%% BEGIN PROP %%%%%%%%%%%%%%%%%%%%%%%%%%%%%%
%
\begin{lemm}\label{l:vasy-2}
The `event horizon' $\{\mu=0\}$ has the following properties:
%%%%%%%%%%%%%%%%%%%%%%%%%%%%%%
\begin{enumerate}
  \item $L_\pm$ consists of fixed points for $\lxir^{-1}H_p$
(also called radial points), with $L_+$ a source and $L_-$ a sink;
  \item $\Sigma_+\cap S^*X\cap \{\mu=0\}=L_+$ and $\Sigma_-\cap \{\mu=0\}=L_-$;
  \item for $\delta>0$ small enough,
$\pm\lxir^{-1}H_p\mu<0$ on $\Sigma_\pm\cap\{-\delta\leq\mu\leq 0\}\setminus L_\pm$;
\end{enumerate}
%%%%%%%%%%%%%%%%%%%%%%%%%%%%%%
\end{lemm}
%%%%%%%%%%%%%%%%%%%%%%%%%%%%%%%%%%%%%%%%%%%%%%%%%%%%%%%%%%%%%%%%%%%%%%%%%%%%%%%%
From this Lemma it follows that bicharacteristics of $H_p$ can cross $\{\mu=0\}$ only from $\{\mu >0\}$ to $\{\mu <0\}$ and never in the other direction, hence the name `event horizon' as in the theory of black holes \cite{v1}.
\begin{proof}
(1) This is proved in~\cite[\S3.4]{v2}; see~\eqref{e:rho1-1}
below for a quantification of the source/sink property.

(2) Near $\{\mu=0\}$, we have~\cite[(3.23)]{v2}
\begin{equation}\label{e:p-near-0}
p=4\mu\tilde\xi^2-4(1+\mathcal O(\mu))\tilde\xi-1+\mathcal O(\mu)+g_1^{-1}(\tilde\eta,\tilde\eta).
\end{equation}
Here $\mathcal O(\mu)$ denotes a smooth function of $\mu$ vanishing at $\mu=0$.
Take $\check\xi=\lxir^{-1}\tilde\xi$,
$\check\eta=\lxir^{-1}\tilde\eta$; then at $\{\mu=0\}$,
$$
\lxir^{-2}p=-4\lxir^{-1}\check\xi
-\lxir^{-2}+g_1^{-1}(\check\eta,\check\eta),\\
$$
This shows that $\Sigma_\pm\cap \{\mu=0\}\cap S^*X=L_\pm$;
to see that $\Sigma_-\cap\{\mu=0\}\subset S^*X$, we can use
that $\Sigma_\pm\cap\{\mu=0\}\subset\{\pm(\tilde\xi+1/2)>0\}$,
as shown in the discussion following~\cite[(3.27)]{v2},
together with~\eqref{e:p-near-0}.

(3) Take $\check\xi,\check\eta$ as in part~(2);
using~\eqref{e:p-near-0}, we get
\begin{equation}\label{e:hp-near-0}
\lxir^{-1}H_p\mu=-4\lxir^{-1}(1+\mathcal O(\mu))
+8\mu\check\xi.
\end{equation}
On $\Sigma_+$, as in part (2), $\tilde\xi>-1/2$ and thus $\check\xi\geq-\lxir^{-1}/2$;
therefore, for $\mu\leq 0$,
$$
\lxir^{-1}H_p\mu\leq - 4\lxir^{-1}(1+\mathcal O(\mu))
- 4\mu\lxir^{-1}
$$
This is negative for small $\mu$ unless $\lxir^{-1}=0$;
in the latter case, $\lxir^{-1}H_p\mu=8\mu\check\xi$
and $\check\xi>0$, $\mu < 0$ since we are on $\Sigma_+ \cap S^*X \setminus L_+$.
On $\Sigma_-\setminus S^*M$ we use \eqref{e:p-near-0} to eliminate the last term from \eqref{e:hp-near-0}, getting
\[
\lxir^{-1}H_p\mu = 2\lxir^{-1}(2 + \tilde\xi^{-1})(1+\mathcal O(\mu)) - 2g_1^{-1}(\check\eta,\check\eta)\check\xi^{-1}
\ge 2\lxir^{-1}(2 + \tilde\xi^{-1})(1+\mathcal O(\mu)) .
\]
Since $\tilde\xi<-1/2$ on $\Sigma_-$,
and since $\tilde\xi$ must attain a maximum there, this expression is positive for small $\mu$.
Finally, on $\Sigma_-\cap S^*M\setminus L_-$ we have $\lxir^{-1}H_p\mu=8\mu\check\xi$
and $\mu<0$, $\check\xi<0$.
\end{proof}
%
%%%%%%%%%%%%%%%%%%%%%%%%%%%%%%% END PROP %%%%%%%%%%%%%%%%%%%%%%%%%%%%%%%
%
We also need to compute the sign of the imaginary part of $P(z)$ when
$z$ moves away from the real line. We will use this to obtain
improved estimates~\eqref{e:improved-estimate} in the physical half-plane.
%
%%%%%%%%%%%%%%%%%%%%%%%%%%%%%% BEGIN PROP %%%%%%%%%%%%%%%%%%%%%%%%%%%%%%
%
\begin{lemm}\label{l:im-part}
The operator $\partial_z P(z)$ lies in $\Psi^1(X)$
and for an appropriate choice of the function $\phi$ from~\eqref{e:p1def}
and $\delta>0$ small enough, we have 
\begin{equation}\label{e:im-part}
\mp \Real(\lxir^{-1} \sigma(\partial_{z}P(0)))>0
\text{ near }\Sigma_\pm\cap\{\mu\geq-\delta\}.
\end{equation}
\end{lemm}
%%%%%%%%%%%%%%%%%%%%%%%%%%%%%%%%%%%%%%%%%%%%%%%%%%%%%%%%%%%%%%%%%%%%%%%%%%%%%%%%
\begin{proof}
It follows from part~(6) of Lemma~\ref{l:vasy-1} that $\partial_zP(z)\in\Psi^1(X)$.
Moreover, the principal symbol of $\partial_zP(z)$ is equal to the derivative in $z$
of the principal symbol of $P(z)$.

For $\delta$ small enough, we can use~\cite[(3.6)]{v2} to
write the principal symbol of $P(z)$ in $\{|\mu|\leq\delta\}$ as
$$
4\mu\tilde\xi^2-4(1+\mathcal O(\mu))(1+z)\tilde\xi-(1+z)^2(1+\mathcal O(\mu))+g_1^{-1}(\tilde\eta,\tilde\eta).
$$
Therefore, the principal symbol of $\partial_z P(z)$ at $z=0$ is
$$
-4(1+\mathcal O(\mu))\tilde\xi-2+\mathcal O(\mu).
$$
Since $\Sigma_\pm\cap\{|\mu|\leq \delta\}\subset \{\pm (\tilde\xi+1/2)>0\}$, we obtain~\eqref{e:im-part}
for $|\mu|\leq\delta$.

It remains to consider the region $\{\mu>\delta\}$. Choose the function $\phi$ so that
$|d\phi|_{g^{-1}}<1$, this is possible by the discussion following~\cite[(3.13)]{v2}.
By~\cite[(3.11)]{v2}, we have
$$
\sigma(\partial_z P(0))=-2\mu^{-1}((\xi,d\phi)_{g^{-1}}+1-|d\phi|^2_{g^{-1}}).
$$
Since $(\xi-d\phi,\xi-d\phi)_{g^{-1}}=1$ on $\Sigma_+$, this becomes
\[
-\mu^{-1}(|\xi|^2_{g^{-1}}+1-|d\phi|^2_{g^{-1}})<0.\hfil\qedhere
\]
\end{proof}
%
%%%%%%%%%%%%%%%%%%%%%%%%%%%%%%% END PROP %%%%%%%%%%%%%%%%%%%%%%%%%%%%%%%
%
To relate $p$ to the principal symbol $p_0$ of $h^2\Delta_{g}-1$,
consider the map $\iota: T^*M\to \overline T^*X$
given by
\begin{equation}\label{e:iotadef}
\iota(x,\xi)=(x,\xi
+d\phi(x)),\
x\in M,\
\xi\in T^*_x M;
\end{equation}
then
\begin{equation}\label{e:p-p0}
p(\iota(x,\xi))=\mu(x)^{-1}p_0(x,\xi).
\end{equation}
In particular, the images of flow lines of $H_{p_0}$ on $p_0^{-1}(0)$ under $\iota$
are reparametrized flow lines of $H_p$ on $p^{-1}(0)$; the reparametrization
factor is bounded as long as we are away from $\{\mu\leq 0\}$.

To state our next lemma, which collects some global properties of the flow of $H_p$, we need the following
notions of sets trapped in one time direction on $T^*M$:
\begin{equation}\label{e:gammadef}
\widetilde \Gamma_\pm = \big\{\rho \in T^*M \setminus 0\mid \{\exp(tH_{p_0})\rho \mid \mp t \ge 0\} \textrm{ is bounded}\big\},\
\Gamma_\pm = \widetilde\Gamma_\pm \cap p_0^{-1}(0).
\end{equation}
The sets $\widetilde\Gamma_+$ and $\widetilde\Gamma_-$ are respectively the \textit{forward} and \textit{backward trapped sets}, and $\widetilde K = \widetilde \Gamma_+ \cap \widetilde \Gamma_-$, $K = \Gamma_+ \cap \Gamma_-$.
%
%%%%%%%%%%%%%%%%%%%%%%%%%%%%%% BEGIN PROP %%%%%%%%%%%%%%%%%%%%%%%%%%%%%%%
%
\begin{lemm}\label{l:vasy-3}
If $\delta>0$ is small enough and
$\gamma(t)$ is a flow line of $\lxir^{-1}H_p$
on $\{\lxir^{-2}p=0\}$ with $\gamma(0)\in\{\mu>-\delta\}$, then:
%%%%%%%%%%%%%%%%%%%%%%%%%%%%%%
\begin{enumerate}
  \item if $\gamma(0)\in\Sigma_+\setminus (L_+\cup\iota(\Gamma_-))$, then
there exists $T>0$ such that $\gamma(T)\in\{\mu\leq -\delta\}$;
  \item if $\gamma(0)\in\Sigma_+$, then either $\gamma(-T)\in\iota(\Gamma_+)$ for $T>0$ large enough or
  $\gamma(t)$ converges to $L_+$ as $t\to -\infty$;
  \item if $\gamma(0)\in\Sigma_-\setminus L_-$, then there exists
$T>0$ such that $\gamma(-T)\in\{\mu\leq -\delta\}$;
  \item if $\gamma(0)\in\Sigma_-$, then $\gamma(t)$ converges to $L_-$
as $t\to +\infty$.
\end{enumerate}
%%%%%%%%%%%%%%%%%%%%%%%%%%%%%%
See Figure~\ref{f:vasy} in the introduction for a picture of the global dynamics of the flow.
\end{lemm}
%%%%%%%%%%%%%%%%%%%%%%%%%%%%%%%%%%%%%%%%%%%%%%%%%%%%%%%%%%%%%%%%%%%%%%%%%%%%%%%%
\begin{proof}
We demonstrate (1); the other statements are proved similarly.
 By~\cite[Lemma~3.2]{v2}, there exists $T_0\geq 0$ such that $\gamma(T_0)\in \{|\mu|\geq \delta\}$.
If $\mu(\gamma(T_0))\leq -\delta$, then we are done; assume
that $\mu(\gamma(T_0))\geq \delta$. By~\cite[(3.30)]{v2}
and Lemma~\ref{l:vasy-2}(3),
the set $V_\delta=\Sigma_+\cap\{\mu\geq\delta\}$ is convex
in the following sense: if $\tilde\gamma(t)$ is any flow line of $\lxir^{-1}H_p$
with $\tilde\gamma(t_1),\tilde\gamma(t_2)\in V_\delta$ for some $t_1<t_2$, then
the whole segment $\tilde\gamma([t_1,t_2])$ lies in $V_\delta$.
This leaves only two cases: either $\gamma([T_0,\infty))\subset V_\delta$
or there exists $T_1>T_0$ such that $\gamma([T_1,\infty))\cap V_\delta=\emptyset$.
By~\eqref{e:p-p0}, the first case would mean that $\gamma(0)\in\iota(\Gamma_-)$;
the second case implies by~\cite[Lemma~3.2]{v2} that $\gamma(T)\in\{\mu=-\delta\}$
for some $T>0$.
\end{proof}
%
%%%%%%%%%%%%%%%%%%%%%%%%%%%%%%% END PROP %%%%%%%%%%%%%%%%%%%%%%%%%%%%%%%
%
Now, we take $\delta$ small enough so that Lemmas~\ref{l:vasy-2}, \ref{l:im-part}, and~\ref{l:vasy-3} hold
and any operator $Q$ such that:
%%%%%%%%%%%%%%%%%%%%%%%%%%%%%%
\begin{itemize}
\item $Q\in\Psi^2(X)$ with Schwartz kernel supported in $\{\mu < -\epsilon\}$ for some $\epsilon>0$;
\item $q=\sigma(Q)$ is real-valued and $\pm q\geq 0$ near $\Sigma_\pm$;
\item $Q$ is elliptic on $\{\lxir^{-2}p=0\}\cap \{\mu\leq-\delta\}$.
\end{itemize}
%%%%%%%%%%%%%%%%%%%%%%%%%%%%%%
Such a $Q$ satisfies the conditions of~\cite[\S3.5]{v2}, except for the self-adjointness condition;
Lemma~\ref{l:vasy-outline} then follows from~\cite[Theorem~4.3]{v2} and~\cite[proof of Theorem~5.1]{v2}.
(In~\cite{v2}, $Q$ was required to be self-adjoint; however as remarked in \cite[\S2.2]{v1},
this condition can be relaxed. Strictly speaking we are citing \cite[Theorem~2.11, Theorem~4.3]{v1}).
We will impose more conditions on $Q$ in \S\ref{s:ultimate}.

%%%%%%%%%%%%%%%%%%%%%%%%%%%%%%%%%%%%%%%%%%%%%%%%%%%%%%%%%%%%%%%%%%%%%%%%%%%%%%%%
%                                  SECTION 3.2                                 %
%%%%%%%%%%%%%%%%%%%%%%%%%%%%%%%%%%%%%%%%%%%%%%%%%%%%%%%%%%%%%%%%%%%%%%%%%%%%%%%%
\subsection{Conjugation and escape function}
  \label{s:ah.escape}

We first study $P(z)$ near the radial points $L_\pm$ defined in~\eqref{e:l-pm}.
The functions $\mu,\tilde y,\tilde\rho=|\tilde\xi|^{-1},\hat\eta=\tilde\rho\tilde\eta$
form a coordinate system on $\overline T^*X$ near $L_\pm$;
in these coordinates, $L_\pm$ is given by $\{\mu=\tilde\rho=\hat\eta=0\}$, and~\cite[(3.23) and~(3.28)]{v2}
\begin{equation}\label{e:near-l-pm}
\begin{gathered}
|\tilde\xi|^{-2} p=4\mu\mp 4(1+\mathcal O(\mu))\tilde\rho-(1+\mathcal O(\mu))\tilde\rho^2
+g_1^{-1}(\hat\eta,\hat\eta);\\
|\tilde\xi|^{-1} H_p=\pm 4(2\mu \partial_\mu+\tilde\rho \partial_{\tilde\rho}+\hat\eta \partial_{\hat\eta})
-4\tilde\rho \partial_\mu+\mathcal O(\hat\eta)\partial_{\tilde y}+\mathcal O(\mu^2+\tilde\rho^2+|\hat\eta|^2).
\end{gathered}
\end{equation}
Define the function
\begin{equation}
  \label{e:rho-1}
\rho_1=\mu^2+\tilde\rho^2+g_1^{-1}(\hat\eta,\hat\eta)
\end{equation}
near $L_\pm$; extend it to the whole $\overline T^*X$
so that $\rho_1>0$ outside of $L_+\cup L_-$. The function
$\rho_1$ has properties similar to the function
$\tilde\rho^2+\rho_0$ used in the proof of~\cite[Proposition~4.6]{v2};
we will use it to define neighborhoods of $L_\pm$.
%
%%%%%%%%%%%%%%%%%%%%%%%%%%%%%% BEGIN PROP %%%%%%%%%%%%%%%%%%%%%%%%%%%%%%
%
\begin{lemm}\label{l:rho1}
For $\delta>0$ small enough,
\begin{gather}
\label{e:rho1-1}
\pm\lxir^{-1}H_p\rho_1\geq\delta\rho_1
\text{ on }\Sigma_\pm\cap\{\rho_1\leq 5\delta\};\\
\label{e:rho1-2}
\{\lxir^{-2}p=0\}\cap S^*X
\cap \{\mu\geq-\delta\}\subset \{\rho_1< 5\delta\}.
\end{gather}
\end{lemm}
\begin{proof} 
We compute $\pm|\tilde\xi|^{-1}H_p\rho_1\geq 4\rho_1-\mathcal O(\rho_1^{3/2})$; \eqref{e:rho1-1} follows.
To show~\eqref{e:rho1-2}, take a point in $\Sigma_\pm\cap S^*X\cap\{\mu\geq-\delta\}$;
by~\eqref{e:p-p0}, we have $\mu\leq 0$.
By Lemma~\ref{l:vasy-2}(2), for $\delta$ small enough our point lies
in the domain of the coordinate system $(\mu,\tilde y,\tilde\rho,\hat\eta)$.
By~\eqref{e:near-l-pm}, and using $S^*X = \{\tilde \rho = 0\}$,
we get
$$
g_1^{-1}(\hat\eta,\hat\eta)=-4\mu\leq 4\delta;
$$
it remains to recall the definition of $\rho_1$.
\end{proof}
%
%%%%%%%%%%%%%%%%%%%%%%%%%%%%%%% END PROP %%%%%%%%%%%%%%%%%%%%%%%%%%%%%%%
%
We now take the density on $X$ introduced in~\cite[\S3.1]{v2}
(in fact, any density would work). If $B$ is any continuous operator
$C^\infty(X)\to \mathcal D'(X)$ and $B^*$ is its adjoint,
then define
$$
\Imag B={1\over 2i}(B-B^*);
$$
for each $u\in C^\infty(X)$,
$$
\Imag(Bu, u)=((\Imag B) u,u).
$$
Instead of using the radial points estimate~\cite[Proposition~4.5]{v2},
we will conjugate $P(z)$ by an elliptic operator of order $s$
to make the imaginary part of the subprincipal symbol have the correct sign:
%
%%%%%%%%%%%%%%%%%%%%%%%%%%%%%% BEGIN PROP %%%%%%%%%%%%%%%%%%%%%%%%%%%%%%
%
\begin{lemm}\label{l:radial-conj}
Assume that $|\Real z |\leq C_0h, \ |\Imag z| \leq C_0h$, $s>C_0$ is fixed, and
$T_s\in\Psi^s(X)$ is any elliptic operator, as in Lemma~\ref{l:main}. Then for $\delta>0$ small enough,
$$
\mp\sigma(h^{-1}\Imag(T_sP(z)T_s^{-1}))\geq \delta\lxir
\text{ on }\Sigma_\pm\cap\{\rho_1\leq 5\delta\}.
$$
\end{lemm}
Note that, since $\sigma(T_sP(z)T_s^{-1})=p$ is real-valued,
$\Imag(T_sP(z)T_s^{-1})\in h\Psi^1$.
%%%%%%%%%%%%%%%%%%%%%%%%%%%%%%%%%%%%%%%%%%%%%%%%%%%%%%%%%%%%%%%%%%%%%%%%%%%%%%%%
\begin{proof}
We consider $P(z)$ as a function of $\tilde z=h^{-1}z$.
It suffices to show that
$$
\pm\lxir^{-1}\sigma(h^{-1}\Imag(T_s P(z)T_s^{-1}))|_{L_\pm}<0.
$$
Consider the coordinates $\mu,\tilde y,\tilde\rho=|\tilde\xi|^{-1},\hat\eta=\tilde\rho\tilde\eta$
near $L_\pm$. By~\cite[(3.10)]{v2},
$$
|\tilde\xi|^{-1}\sigma(h^{-1}\Imag P(z))|_{L_\pm}=\mp 4\Imag\tilde z.
$$
Furthermore, by part~(6) of Lemma~\ref{l:vasy-1} we have $P(z)-P(0) \in h \Psi^1$, and hence $T_s(P(z) - P(0))T_s^{-1} -
(P(z) - P(0)) \in h^2 \Psi^0$, from which we conclude that
it suffices to prove that
\begin{equation}
  \label{e:zimbabwe}
|\tilde\xi|^{-1}\sigma(h^{-1}(T_sP(0)T_s^{-1}-P(0)))|_{L_\pm}=\mp 4is.
\end{equation}
Near $L_\pm$, we can write $\sigma(T_s)=\tilde\rho^{-s}e^a$,
with $a$ smooth on $\overline T^*X$. Then
$$
\begin{gathered}
|\tilde\xi|^{-1}\sigma(h^{-1}(T_sP(0)T_s^{-1}-P(0)))
=\tilde\rho \sigma(h^{-1}[T_s,P(0)]T_s^{-1})\\
=i\tilde\rho^{s+1}e^{-a}H_p(\tilde\rho^{-s}e^a)
=-isH_p\tilde\rho+i\tilde\rho H_pa.
\end{gathered}
$$
The first term on the right-hand side gives~\eqref{e:zimbabwe} by~\eqref{e:near-l-pm},
while the second one vanishes at $L_\pm$ since $|\tilde\xi|^{-1}H_p=0$ there.
\end{proof}
%
%%%%%%%%%%%%%%%%%%%%%%%%%%%%%%% END PROP %%%%%%%%%%%%%%%%%%%%%%%%%%%%%%%
%
Now, we construct an escape function for the region $(\Sigma_+ \cup
\Sigma_-)\cap\{\mu\geq -\delta\}\setminus
(U_K\cup\{\rho_1<5\delta\})$, where $U_K$ is a neighborhood of $\iota(K)$ in $T^*X$. This  is based partly on
\cite[Lemma 4.3]{d-v2} (see also \cite[Appendix]{g-s}, \cite[\S
4]{v-z}).
%
%%%%%%%%%%%%%%%%%%%%%%%%%%%%%% BEGIN PROP %%%%%%%%%%%%%%%%%%%%%%%%%%%%%%
%
\begin{lemm}\label{l:escape}
For $\delta>0$ small enough and any sufficiently small neighborhood
$U_K$ of $\iota(K)$, there exists a smooth nonnegative
function $f_0$ on $\overline T^*X$ such that:
%%%%%%%%%%%%%%%%%%%%%%%%%%%%%%
\begin{enumerate}
\item $f_0$ is supported in $\{\mu>-2\delta\}$ and away from $S^*X$;
\item $H_p f_0 \ge 0$ near $(\Sigma_+\cap\{\mu\geq -\delta\})\cup \overline U_K$,  and $ H_pf_0 \le 0$ near $\Sigma_-\cap\{\mu\geq -\delta\}$;
\item $\pm H_pf_0>0$ on $V_\pm=\Sigma_\pm\cap\{\mu\geq -\delta\}\setminus (U_K\cup\{\rho_1<5\delta\})$;
\item $H_pf_0 = 0$ near $\overline U_K \cap \iota(\widetilde K)$.
\end{enumerate}
%%%%%%%%%%%%%%%%%%%%%%%%%%%%%%
\textrm{In fact, the function $f_0$ we construct in the proof is identically $1$ near $\overline U_K \cap \iota(\widetilde K)$.}
\end{lemm}
%%%%%%%%%%%%%%%%%%%%%%%%%%%%%%%%%%%%%%%%%%%%%%%%%%%%%%%%%%%%%%%%%%%%%%%%%%%%%%%%
\begin{proof}
Note that $V_+$ and $V_-$ are compact and disjoint, and by~\eqref{e:rho1-2} neither intersects $S^*X$.
We will first construct a function $\tilde f_0 \in C^\infty(\overline T^*X)$,  with the following properties:
%%%%%%%%%%%%%%%%%%%%%%%%%%%%%%
\begin{enumerate}
\item $\tilde f_0 \le -2$ near $(\Sigma_+ \cup \Sigma_-) \cap S^*X \cap \{\mu \ge -\delta\}$;
\item near $\Sigma_\pm\cap\{\mu\geq -\delta\}$, $\pm H_p \tilde f_0\geq 0$;
\item $\pm H_p \tilde f_0>0$, $ \tilde f_0 \ge -1/2$ on $V_\pm$;
\item $ H_p \tilde f_0 = 0$ near $ \iota(\widetilde K)$.
\end{enumerate}
%%%%%%%%%%%%%%%%%%%%%%%%%%%%%%
%
%%%%%%%%%%%%%%%%%%%%%%%%%%%%%% BEGIN FIGURE %%%%%%%%%%%%%%%%%%%%%%%%%%%%%%
%
\begin{figure}
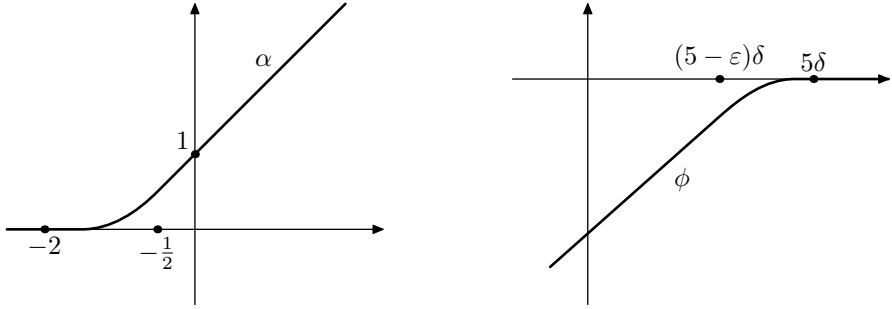

\includegraphics{fwl.2}\qquad\qquad
\includegraphics{fwl.3}
\caption{We precompose $\tilde f_0$ with $\alpha$ to obtain a
function which is $0$ near $L_+$ and $1$ near $\widetilde K$ and smooth in
between. We use a multiple of $\phi$ in the definition of $\tilde f_0$
to guarantee that $\tilde f_0 \le -2$ near $\Sigma_+ \cup \Sigma _-
\cap S^*X$ so that $\alpha \circ \tilde f_0$ vanishes there.}
\label{f:escape}
\end{figure}
%
%%%%%%%%%%%%%%%%%%%%%%%%%%%%%%% END FIGURE %%%%%%%%%%%%%%%%%%%%%%%%%%%%%%%
%

Before constructing this function we show how we use it to construct
$f_0$. In several places we will shrink $U_K$, keeping $U_K \cap \Sigma_+$ fixed:
note that this procedure does not affect $\tilde f_0$.

Take $\alpha \in C^\infty(\mathbb R)$ nondecreasing with
 $\alpha(t) = t+1$ near $t \ge -1/2$ and $\supp \alpha \subset (-2,\infty)$.
 (See Figure~\ref{f:escape}.)
Take $\chi \in C^\infty(\overline T^*X;[0,1])$ supported in
$\{\mu > -2\delta\}$ and away from $S^*X$, and with $\chi = 1$
near $(\Sigma_+ \cup \Sigma_-)\cap\{\mu\geq -\delta\} \cap
\overline{\{\tilde f_0 > -2\}}$. This is possible by property (1) of $\tilde
f_0$. Then put
\[
f_0(x, \xi) = \chi(x, \xi) \alpha (\tilde f_0(x,\xi)).
\]
Property (1) of $f_0$ follows from the support condition on $\chi$. Note that thanks to
the choice of the set where $\chi =1$ together with $\alpha(t) = 0$ near $t \le -2$, we have 
\begin{equation}\label{e:hpf0}
H_p f_0(x,\xi) = \alpha'(\tilde f_0(x, \xi)) H_p \tilde f_0 \textrm{ near } (\Sigma_+ \cup \Sigma_-)\cap\{\mu\geq -\delta\}.
\end{equation}
Hence (if necessary shrinking $U_K$ while keeping $U_K \cap \Sigma_+$ fixed) property (2) of $f_0$ follows from property (2) of $\tilde f_0$ together with the fact that $\alpha$ is nondecreasing. Properties (3) and (4)  of $f_0$ follow from properties (3) and (4) of $\tilde f_0$ together with  \eqref{e:hpf0} and with the fact that $\alpha(t) = t+1$ near $t \ge -1/2$, again if necessary shrinking $U_K$ while keeping $U_K \cap \Sigma_+$ fixed.

We will take $\tilde f_0$ of the form
\[
\tilde f_0 = \sum_{k=1}^N f_k + C \phi(\rho_1).
\]
Here $\phi \in C^\infty(\mathbb{R};(-\infty,0])$ is nondecreasing, supported in
$(-\infty,5\delta]$,  and  $\phi' = 1$ on
$(-\infty,(5-\varepsilon)\delta]$, where $\varepsilon>0$ is small
enough that $(\Sigma_+ \cup \Sigma_-) \cap S^*X \cap \{\mu \ge
-\delta\} \subset \{\rho_1 < (5-\varepsilon)\delta\}$ (see
\eqref{e:rho1-2}). Each $f_k$, specified below, is supported near the
bicharacteristic through $(x_k,
\xi_k)$, a suitably chosen point in $V_+ \cup V_-$. 
Now if $C$ is sufficiently large (depending on $\sum
f_k$) we will have property (1) of $\tilde f_0$. It suffices now to
construct the $f_k$ so that properties (2), (3) and (4) of $\tilde f_0$ hold, and indeed since
$\pm H_p\phi(\rho_1) \ge 0$ on $\Sigma_\pm$ by~\eqref{e:rho1-1} and $\supp
\phi(\rho_1) \cap (V_+ \cup V_- \cup U_K) = \emptyset$ it is enough to check
these properties for $\sum f_k$ (to prove property (2) we will also increase $C$ further).

To determine the $(x_k,\xi_k)$ we first fix an open neighborhood
$\widetilde U_K$ of $\iota(K)$ with $\overline{\widetilde U_K}\subset U_K$,
and associate to each $(x,\xi) \in V^+$ the following
\textit{escape times}:
\begin{align*}
T_{(x,\xi)}^0 &=  \inf\{t \in \mathbb R \colon \rho_1(\gamma(t)) \ge 4 \delta \textrm{ and } \gamma(t) \not\in \widetilde U_K\},\\
T_{(x,\xi)}^1 & = \sup\{t \in \mathbb R \colon \mu(\gamma(t)) \ge - 3\delta/2 \textrm{ and } \gamma(t) \not\in \widetilde U_K\}.
\end{align*}
Here $\gamma(t)$ is the bicharacteristic flowline through $(x,\xi)$. For $(x,\xi) \in V^-$ we put
\[
T_{(x,\xi)}^0 = \inf\{t \in \mathbb R \colon
\mu(\gamma(t)) \ge -3\delta/2\}, \qquad T_{(x,\xi)}^1 =
\inf\{t \in \mathbb R \colon \rho_1(\gamma(t)) \ge	 4 \delta\}.\]
Note that for every $(x,\xi) \in V_- \cup V_+$ we have
$- \infty <T_{(x,\xi)}^0 < 0 < T_{(x,\xi)}^1 <\infty$
thanks to the description of the large time behavior
of trajectories in $\Sigma_\pm$ given by Lemma \ref{l:vasy-3} (we use the fact
that all trajectories in $\iota(\Gamma_\pm)$ tend to $\iota (K)$ as $t
\to \mp \infty$, see for example \cite[Proposition~A.2]{g-s}).
%
%%%%%%%%%%%%%%%%%%%%%%%%%%%%%% BEGIN FIGURE %%%%%%%%%%%%%%%%%%%%%%%%%%%%%%
%
\begin{figure}
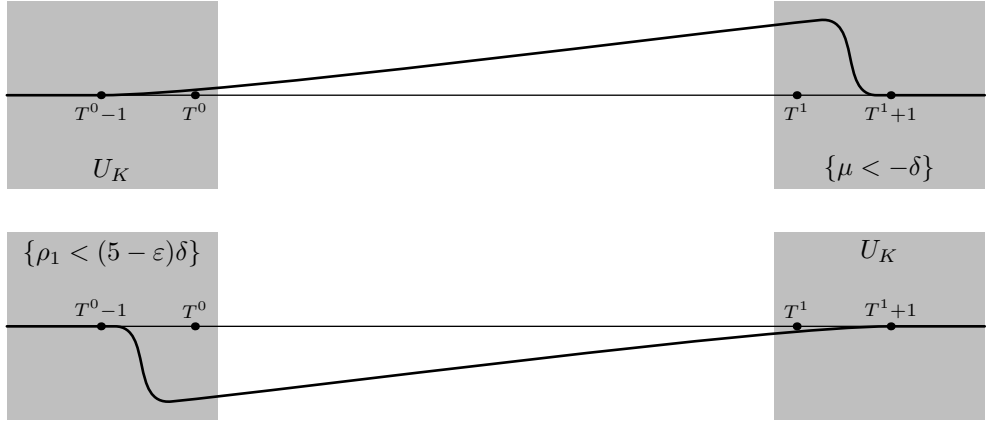

\includegraphics{fwl.4}

\vspace{.5cm}

\includegraphics{fwl.5}
\caption{A graph of $\chi_{(x, \xi)(t)}$ in the cases when the bicharacteristic
through $(x, \xi)$ tends to $\iota(\widetilde K)$ as $t \to
- \infty$ (top) and as $t \to + \infty$ (bottom) (other cases are
similar and simpler). Here $t$ is time along a
bicharacteristic flowline. In each case $\chi' \ge 0$ near the
trapped set. In the first case $\chi' > 0$ until the
bicharacteristic enters the elliptic set of $Q$, and in the second
case  $\chi' > 0$ starting when the bichracteristic leaves a
small neighborhood of the radial set and continuing until it enters a
small neighborhood of the trapped set.}
\label{f:escape2}
\end{figure}
%
%%%%%%%%%%%%%%%%%%%%%%%%%%%%%%% END FIGURE %%%%%%%%%%%%%%%%%%%%%%%%%%%%%%%
%
Next, let $\mathcal{S}_{(x,\xi)}$ be a hypersurface
through $(x,\xi)$ which is transversal to $H_p$ near
$(x,\xi)$. Then if $U_{(x,\xi)}$ is a
sufficiently small neighborhood of $(x,\xi)$, the set
\[V_{(x,\xi)} = \{\gamma_{(x',\xi')}(t) \colon (x',\xi') \in U_{(x,\xi)}
\cap \mathcal S_{(x,\xi)}, t \in (T_{(x,\xi)}^0 - 1, T_{(x,\xi)}^1 + 1)\},\]
where $\gamma_{(x',\xi')}(t)$ is the bicharacteristic
flowout of $(x',\xi')$, is diffeomorphic to
$(\mathcal S_{(x,\xi)} \cap U_{(x,\xi)})
\times (T_{(x,\xi)}^0 - 1, T_{(x,\xi)}^1
+ 1)$, and this diffeomorphism defines product coordinates on
$V_{(x,\xi)}$. If necessary, shrink $U_{(x,\xi)}$ so that
\begin{equation}\label{e:tesc}
\overline{V_{(x,\xi)}} \cap \{t \in [T_{(x,\xi)}^0-1,T_{(x,\xi)}^0]\cup [T_{(x,\xi)}^1,T_{(x,\xi)}^1+1]\}
\subset U_K \cup \{\rho_1 < (5 - \varepsilon) \delta\} \cup \{\mu < - \delta\}.
\end{equation}
Take $\varphi_{(x,\xi)} \in C_0^\infty(\mathcal
S_{(x,\xi)} \cap U_{(x,\xi)};[0,1])$
identically $1$ near $(x,\xi)$, also considered as a
function on $V_{(x,\xi)}$ via the product coordinates,
and let $V'_{(x,\xi)} \subset V_{(x,\xi)}$ be the product of $(T_0,T_1)$
with an open subset of $\mathcal S_{(x,\xi)}\cap U_{(x,\xi)}$ on which $\varphi_{(x,\xi)}=1$.
Using the compactness of $V_+ \cup V_-$, take $(x_1,\xi_1), \dots, (x_N,\xi_N)$ with
$$
V_+ \cup V_- \subset \bigcup_{k=1}^N V'_{(x_k,\xi_k)}.
$$
For each $k \in \{1,\dots, N\}$ put $f_k = f_{(x_k,\xi_k)}$, where
$$
f_{(x,\xi)} = \chi_{(x,\xi)}(t) \varphi_{(x,\xi)}, \qquad H_p f_{(x,\xi)}
= \chi'_{(x,\xi)}(t) \varphi_{(x,\xi)},
$$
and where $\chi_{(x,\xi)} \in C_0^\infty((T_{(x,\xi)}^0 - 1,
T_{(x,\xi)}^1 + 1))$.  Note that $V_{(x,\xi)} \cap
\iota(\widetilde K) = \emptyset$ for all $(x,\xi) \in V_+
\cup V_-$, and so each $f_k$ vanishes near $\iota(\widetilde K)$, and
in particular we have property (4) of $\tilde f_0$.

We further impose that $\pm \chi'_{(x,\xi)} >0$ (accordingly as
$(x,\xi) \in V_\pm$) and $\chi_{(x,\xi)} \ge -(2N)^{-1}$ on
$[T_{(x,\xi)}^0, T_{(x,\xi)}^1]$. This condition gives
property (3) of $\tilde f_0$ since $V_\pm \cap V_{(x,\xi)}
\subset \{t \in [T_{(x,\xi)}^0, T_{(x,\xi)}^1]\}$ and
$\pm H_pf_{(x,\xi)}>0$ on $V'_{(x,\xi)}$.
If $\gamma(t)$ (the bicharacteristic through
$(x,\xi)$) tends to $\iota(\widetilde K)$ as $t \to
\pm \infty$, then we further require that $\chi_{(x,\xi)}' (t) \ge 0$
for $\pm t \ge 0$. (Note that $\gamma(t)$ cannot tend to $\iota(\widetilde K)$ both as $t\to +\infty$
and $t\to -\infty$, as in this case $\gamma(t)\subset \iota(\widetilde K)$.)
 This is sufficient to imply
property (2) of $\tilde f_0$ since now \eqref{e:tesc} implies that
$\Sigma_\pm \cap \{\mu \ge - \delta\} \cap \{\rho_1 \ge
(5-\varepsilon)\delta\} \cap V_{(x,\xi)} \subset
\{\pm\chi'_{(x,\xi)} (t) \ge 0\}$, and since the $C
\phi(\rho_1)$ term takes care of the set $\{\rho_1 \le
(5-\varepsilon)\delta\} \cap \Sigma_\pm$.
\end{proof}
%
%%%%%%%%%%%%%%%%%%%%%%%%%%%%%%% END PROP %%%%%%%%%%%%%%%%%%%%%%%%%%%%%%%
%

%%%%%%%%%%%%%%%%%%%%%%%%%%%%%%%%%%%%%%%%%%%%%%%%%%%%%%%%%%%%%%%%%%%%%%%%%%%%%%%%
%                                   SECTION 4                                  %
%%%%%%%%%%%%%%%%%%%%%%%%%%%%%%%%%%%%%%%%%%%%%%%%%%%%%%%%%%%%%%%%%%%%%%%%%%%%%%%%
\section{Exotic classes of operators}\label{s:exotic}

%%%%%%%%%%%%%%%%%%%%%%%%%%%%%%%%%%%%%%%%%%%%%%%%%%%%%%%%%%%%%%%%%%%%%%%%%%%%%%%%
%                                  SECTION 4.1                                 %
%%%%%%%%%%%%%%%%%%%%%%%%%%%%%%%%%%%%%%%%%%%%%%%%%%%%%%%%%%%%%%%%%%%%%%%%%%%%%%%%
\subsection{\texorpdfstring{$\Psi_{1/2}$}{Psi-1/2} calculus}\label{s:1/2}

In this subsection we review the pseudodifferential calculus of
operators with symbols in the exotic class $S_{1/2}$
depending on two semiclassical parameters $h,\tilde h$,  studied in
\cite[\S3.3]{sj-z} and \cite[\S3]{w-z}. The escape function in \S\ref{s:ultimate0} will provide
positivity up to distance $(h/\tilde h)^{1/2}$  to the
trapped set, and the operator $A$ from Lemma~\ref{l:main} will be
supported $\mathcal O((h/\tilde h)^{1/2})$ close to the trapped set;
 we will study both using this exotic class.

We always assume
$\tilde h$ is small but independent of $h$, and $h$ is  small
depending on $\tilde h$. The reason for the second semiclassical
parameter $\tilde h$ and the corresponding symbol class $S_{1/2}$ is
the following: since our escape function is only regular on the scale $h^{1/2}$ it
belongs to a calculus with no asymptotic
decomposition in powers of $h$, and some of the remainder terms in the
positive commutator estimate in \S\ref{s:ultimate} will be of
order $h$, the same magnitude as the positive term coming from the
commutator. However, if we use symbols  in $S_{1/2}$
(which are regular on the larger scale $(h/\tilde
h)^{1/2}$ instead of just $h^{1/2}$), then we  have an asymptotic
decomposition in powers of $\tilde h$ for the corresponding calculus
and the remainder terms will be $\mathcal O(h\tilde h)$, and hence small in comparison with
the $h$ sized positive term for $\tilde h$ small enough.

\noindent\textbf{Remark.}
An alternative approach would use instead the mildly exotic class $S_\rho$, with
$\rho<1/2$; its elements are regular on the scale $h^\rho$. The
semiclassical calculus in this class has a decomposition in powers of
$h$, which would simplify the arguments below, eliminating the need
for $\tilde h$.  However, in this case the rank of the operator $A_R$
from Lemma~\ref{l:main} would grow as $h^{-2\nu\rho-(n-1)(1-2\rho)}$,
which is the number of cylinders on the energy surface of size $1$ in
the direction of the Hamiltonian flow and size $h^{1/2}$ in all other
directions that are needed to cover an $\mathcal O(h^\rho)$
neighborhood of the trapped set (see \S\ref{s:approximation}).
This is a weaker estimate than the $\mathcal O(h^{-\nu})$ that we 
 prove. By taking $\rho$ very close to $1/2$, one could get any
power of $h^{-1}$ bigger than $\nu$, but not $h^{-\nu}$; this makes a
difference if the trapped set is of pure Minkowski dimension, which is
the case in the most interesting examples (see the introduction).

\smallskip

We proceed to the construction of the $\Psi_{1/2}$ calculus.
We will only need compactly microlocalized operators, thus we restrict
ourselves to symbols that are $\mathcal O(h^\infty)$ outside of a compact set.
For a manifold $X$, we define the class $\Syme(X)$ as
follows: a function $a(x,\xi;h,\tilde h)$ smooth in $(x,\xi)\in T^*X$
lies in this class if and only if:
\begin{itemize}
\item there exists a compact set $V\subset T^*X$ such that each $(x,\xi)$-derivative
of $a$ is $\mathcal O(h^\infty\langle\xi\rangle^{-\infty})$ outside of $V$,
uniformly in $\xi$ and locally uniformly in $x$;
\item for each multiindex $\alpha$, there exists a constant $C_\alpha$ such that
near $V$,
\begin{equation}\label{e:1/2-definition}
|\partial^\alpha_{x,\xi} a|\leq C_{\alpha} (h/\tilde h)^{-|\alpha|/2}.
\end{equation}
\end{itemize}
As in~\S\ref{s:prelim.basics}, we require only local  uniformity in $x$.
This is in contrast with~\cite{sj-z} and~\cite{w-z}, but their results
 still hold if we only require our estimates to be locally uniform in $x$.

We begin with operators on $\mathbb R^n$.  For $a\in\Syme(\mathbb R^n)$,
let $\Op_h(a)$ be its Weyl quantization:
\begin{equation}
  \label{e:q-weyl}
\Op_h(a)u(x)=(2\pi h)^{-n}\int e^{{i\over h}(x-y)\cdot\xi}a\bigg({x+y\over 2},\xi\bigg)
\check\chi(x-y)u(y)\,d\xi dy,\ u\in C^\infty(\mathbb R^n).
\end{equation}
Here $\check\chi\in C_0^\infty(\mathbb R^n)$ is some fixed function
 equal to 1 near the origin. We use the $\check\chi(x-y)$
cutoff, which is absent in the standard definition of the Weyl
quantization, to make $\Op_h(a)$ properly supported. It is also needed
for the integral to converge, as $a$ can grow arbitrarily fast as
$x\to\infty$.  The factor $\check\chi(x-y)$  only changes the
operator $\Op_h(a)$ by a smoothing term of order $\mathcal
O(h^\infty)$ because of the pseudolocality of $\Op_h(a)$, see for
example~\cite[Lemma~3.4]{w-z}.  (We will need to use more standard
symbol classes~\eqref{e:1/2-rn-symbols} and the standard
definition~\eqref{e:1/2-weyl-q} of Weyl quantization in a limited way in \S\ref{s:approximation3}.)  Here are the basic properties of
quantization of exotic symbols on $\mathbb R^n$
(see~\cite[\S3.3]{sj-z} or~\cite[\S3.2]{w-z}):

%
%%%%%%%%%%%%%%%%%%%%%%%%%%%%%% BEGIN PROP %%%%%%%%%%%%%%%%%%%%%%%%%%%%%%
%
\begin{lemm}[Properties of the $S_{1/2}$ calculus on $\mathbb R^n$]
\label{l:1/2-properties-rn}\quad

1. For $a\in\Syme(\mathbb R^n)$, $\Op_h(a)$ is compactly
microlocalized, pseudolocal, and has norm $\mathcal O(1)$, in the
sense of \S\ref{s:prelim.basics}. If $\supp a\subset K$
for some compact set $K\subset T^*\mathbb R^n$
independent of $h,\tilde h$, then $\WFh(\Op_h(a))\subset K$.

2. For $a\in\Syme(\mathbb R^n)$, $\Op_h(a)^*=\Op_h(\bar a)$.

3. For $a,b\in\Syme(\mathbb R^n)$, there exists a symbol $a\#
b\in\Syme(\mathbb R^n)$ such that
$$
\Op_h(a)\Op_h(b)=\Op_h(a\#b)+\Resh.
$$
(The $\Resh$ error comes from the $\check\chi(x-y)$ cutoff.)  The same holds when one of $a,b$ lies in $\Syme(\mathbb R^n)$ and the other in
the class $S^k_{\cl}(\mathbb R^n)$ defined in \S\ref{s:prelim.basics},
with $a\#b$ still in $\Syme(\mathbb R^n)$.
  
4. If $a,b\in\Syme(\mathbb R^n)$, then
$$
a\#b=ab+\mathcal O(\tilde h)_{\Syme(\mathbb R^n)}.
$$

5. If one of $a,b$ lies in $\Syme(\mathbb R^n)$ and the other in $S^k_{\cl}(\mathbb R^n)$,
then
$$
a\#b=ab+\mathcal O(h^{1/2}\tilde h^{1/2})_{\Syme(\mathbb R^n)},\
a\#b-b\#a=-ih\{a,b\}+\mathcal O(h^{3/2}\tilde h^{3/2})_{\Syme(\mathbb R^n)}.
$$

6. Assume that $f:U_1\to U_2$ is a diffeomorphism, $U_1,U_2\subset
\mathbb R^n$, and take $\chi\in C^\infty_0(U_1)$.  Then for each
$a\in\Syme(\mathbb R^n)$,
$$
\begin{gathered}
(f^{-1})^*\chi\Op_h(a)\chi f^*=\Op_h(a_f)+\mathcal O(h^\infty)_{\Psi^{-\infty}},\
a_f\in \Syme(\mathbb R^n),\\
a_f(x,\xi)=\chi(f^{-1}(x))^2a(f^{-1}(x),{}^t
f'(x)\xi)+\mathcal O(h^{1/2}\tilde h^{1/2})_{\Syme}.
\end{gathered}
$$
This fact depends on using the Weyl quantization; the proof can be found
in~\cite[Lemma~3.3]{w-z}. (See also the proof of Lemma~\ref{l:1/2-egorov}
below.)
\end{lemm}
%
%%%%%%%%%%%%%%%%%%%%%%%%%%%%%%% END PROP %%%%%%%%%%%%%%%%%%%%%%%%%%%%%%%
%
Lemma~\ref{l:1/2-properties-rn}(4) and (5) follow from the following
asymptotic expansion~\cite[Lemma~3.6]{sj-z}:
\begin{equation}\label{e:1/2-asymptotic}
(a\#b)(x,\xi)\sim\sum_{j\geq 0} {h^j\over j!(2i)^j}
(\partial_\xi\cdot \partial_y-\partial_\eta\cdot \partial_x)^j
a(x,\xi)b(y,\eta)|_{y=x\atop\eta=\xi}.
\end{equation}
The expansion~\eqref{e:1/2-asymptotic} holds in the following
sense: for every $N$, each $S_{1/2}$ seminorm of the
difference of the left-hand side and the sum of the terms with $j<N$
on the right-hand side is bounded by a certain $S_{1/2}$
seminorm of the $j=N$ term of the sum, taken without restricting to
$y=x,\eta=\xi$ (see~\cite[(3.12)]{sj-z}).  If
both $a,b$ are in $S_{1/2}$, then the $j$th term of the
asymptotic sum is~$\mathcal O(\tilde h^j)$
and~\eqref{e:1/2-asymptotic} is an expansion in powers of $\tilde h$,
not $h$. However, if one of $a,b$ lies in the class $S^k_{\cl}$, then
the $j$th term of the sum is~$\mathcal O(h^{j/2}\tilde h^{j/2})$ and
we get an expansion in powers of $h$ and better remainders for the
product and commutator formulas.  The improved remainder estimate for
the commutator $a\#b-b\# a$ in part~5 of
Lemma~\ref{l:1/2-properties-rn} is due to the fact (specific to the
Weyl quantization) that the $j=2$ term in~\eqref{e:1/2-asymptotic} is
the same for $a\#b$ and $b\#a$; therefore, the remainder comes from
the $j=3$ term.

We can now construct the $\Psi_{1/2}$ calculus on a
manifold, similarly to~\cite[\S3.3]{w-z}. More specifically, we
will define the class~$\Psie(X)$ of compactly microlocalized operators
with symbols in~$\Syme(X)$. Let $X$ be a manifold and $A$ be a
properly supported operator on $X$ depending on $h,\tilde h$, which is
compactly microlocalized and pseudolocal in the sense of
\S\ref{s:prelim.basics}.  We say that $A$ lies in $\Psie(X)$, if
for each coordinate system $f:U_f\to V_f$, with $U_f\subset X$,
$V_f\subset \mathbb R^n$, and each $\chi\in C_0^\infty(U_f)$, there
exists $a_{f,\chi}\in \Syme(U_f)\cap C_0^\infty(T^*U_f)$ such that
$$
(f^{-1})^*\chi A\chi
f^*=\Op_h((f^{-1})^*a_{f,\chi})+\mathcal O(h^\infty)_{\Psi^{-\infty}}.
$$
Here $(f^{-1})^*a_{f,\chi}\in\Syme(\mathbb R^n)\cap
C_0^\infty(T^*V_f)$ denotes the pullback of $a_{f,\chi}$ under the map
$T^*V_f\to T^*U_f$ induced by $f^{-1}$. It can be seen from
Lemma~\ref{l:1/2-properties-rn} that there exists a symbol $a_f\in
C^\infty(T^*U_f)$ such that for each $\chi$, $a_{f,\chi}=\chi^2
a_f+\mathcal O(h^{1/2}\tilde h^{1/2})_{\Syme(X)}$ and moreover, the
symbols $a_f$ given by different coordinate charts agree modulo
$\mathcal O(h^{1/2}\tilde h^{1/2})_{\Syme}$. This makes it possible to
define the principal symbol map
\begin{equation}\label{e:1/2-symbol}
\tilde\sigma:\Psie(X)\to \Syme(X)/(h^{1/2}\tilde h^{1/2}
\Syme(X)).
\end{equation}
We will sometimes consider operators of the form $f(h,\tilde h)A$,
where $f$ is some function and $A\in\Psie(X)$. We put
$\tilde\sigma (fA)=f\tilde\sigma(A)$; it is defined modulo
$\mathcal O(f(h,\tilde h)h^{1/2}\tilde h^{1/2})_{\Syme(X)}$. For instance,
the symbol of an element of $h^{1/2}\tilde h^{1/2}\Psie(X)$
is defined modulo $\mathcal O(h\tilde h)_{\Syme(X)}$.

The symbol map has a non-canonical right inverse $\Op_h:\Syme(X)\to
\Psie(X)$, defined as follows: consider a locally finite covering of
$X$ by the domains $U_j$ of some coordinate charts $f_j:U_j\to
V_j\subset \mathbb R^n$, a partition of unity $\chi_j\in
C_0^\infty(U_j)$ on $X$, and some functions $\chi'_j\in
C_0^\infty(V_j)$ equal to 1 near $f_j(\supp\chi_j)$.  Then for
$a\in\Syme(X)\cap C_0^\infty(T^*X)$, we put
\begin{equation}\label{e:1/2-quantization}
\Op_h(a)=\sum_j f_j^*\chi'_j \Op_h((f_j^{-1})^*(\chi_j a))\chi'_j (f_j^{-1})^*;
\end{equation}
here $(f_j^{-1})^*(\chi_j a)\in\Syme(\mathbb R^n)\cap
C_0^\infty(T^*V_j)$ is quantized by~\eqref{e:q-weyl}. We have
$\Op_h(a)\in\Psie(X)$ and $\tilde\sigma(\Op_h(a))=a+\mathcal
O(h^{1/2}\tilde h^{1/2})_{\Syme(X)}$.

Using the properties of the $S_{1/2}$ calculus on $\mathbb
R^n$ listed above, we get
%
%%%%%%%%%%%%%%%%%%%%%%%%%%%%%% BEGIN PROP %%%%%%%%%%%%%%%%%%%%%%%%%%%%%%
%
\begin{lemm}[Properties of the $\Psi_{1/2}$ calculus on manifolds]
\label{l:1/2-properties}
Let $X$ be a manifold. Then:

1. Each operator in $\Psie(X)$ is compactly microlocalized,
pseudolocal, and has norm $\mathcal O(1)$ in the sense of
\S\ref{s:prelim.basics}. If $\supp a\subset K$ for some compact
set $K\subset T^*X$ independent of $h,\tilde h$, then
$\WFh(\Op_h(a))\subset K$.

2. The class $\Psic(X)$ of compactly microlocalized classical operators
from \S\ref{s:prelim.basics} is contained in $\Psie(X)$,
with a correspondence between the symbol maps.

3. For $A\in\Psie(X)$, $\tilde\sigma(A)=\mathcal O(h^{1/2}\tilde h^{1/2})_{\Syme}$
if and only if $A\in (h^{1/2}\tilde h^{1/2})\Psie(X)$.

4. If $A\in\Psie(X)$, then its adjoint $A^*$ (with respect to some given density)
also lies in $\Psie(X)$ and
$$
\tilde\sigma(A^*)=\overline{\tilde\sigma(A)}+\mathcal O(h^{1/2}\tilde h^{1/2})_{\Syme}.
$$

5. If $A,B\in\Psie(X)$, then $AB\in\Psie(X)$ and
$$
\tilde\sigma(AB)=\tilde\sigma(A)\tilde\sigma(B)+\mathcal O(\tilde h)_{\Syme}.
$$

6. If $A\in\Psi^k(X)$ and $B\in\Psie(X)$, then $AB,BA\in\Psie(X)$,
and
\begin{equation}
  \label{e:1/2-mull}
\tilde\sigma(AB)=\sigma(A)\tilde\sigma(B)+\mathcal O(h^{1/2}\tilde h^{1/2})_{\Syme}=
\tilde\sigma(BA).
\end{equation}
Moreover $[A,B]\in h^{1/2}\tilde h^{1/2}\Psie(X)$, and we have
\begin{equation}
  \label{e:1/2-comm}
\tilde\sigma([A,B])=-ih\{\sigma(A),
\tilde\sigma(B)\}+\mathcal O(h\tilde h)_{\Syme}.
\end{equation}

7. If $A\in\Psi^k(X)$, $B\in\Psie(X)$, and
$\{\sigma(A),\tilde\sigma(B)\}=\mathcal O(1)_{\Syme}$ (instead of
$\mathcal O((\tilde h/h)^{1/2})$ known a priori), then $[A,B]\in
h\Psie(X)$.
\end{lemm}
%
%%%%%%%%%%%%%%%%%%%%%%%%%%%%%%% END PROP %%%%%%%%%%%%%%%%%%%%%%%%%%%%%%%
%
Note that in Lemma~\ref{l:1/2-properties}(1), we
use the notion of  wavefront set of a pseudolocal operator from
\S\ref{s:prelim.basics}, which gives a set independent of
$h$. In particular, if $a\in\Syme(X)$ is supported in an $\mathcal
O((h/\tilde h)^{1/2})$ neighborhood of some compact set $K$, then
$\WFh(\Op_h(a))\subset K$.

We also have the following version of the non-sharp G\r arding inequality:
%
%%%%%%%%%%%%%%%%%%%%%%%%%%%%%% BEGIN LEMMA %%%%%%%%%%%%%%%%%%%%%%%%%%%%%%
%
\begin{lemm}\label{l:garding-1/2}
Assume that $X$ is a compact manifold, $A\in\Psie(X)$,
$\Psi_1\in\Psic(X)$, and $\Real\tilde\sigma(A)>0$ near $\WFh(\Psi_1)$.
Then there exists a constant $C$ such that for each $u\in L^2(X)$,
$$
\Real \langle A\Psi_1 u,\Psi_1 u\rangle\geq C^{-1}\|\Psi_1 u\|_{L^2(X)}^2-\mathcal O(h^\infty)\|u\|_{L^2(X)}^2.
$$
\end{lemm}
\begin{proof}
For $C>0$ large enough, we can write
$$
\Real\tilde\sigma(A)=C^{-1}+|b|^2+\mathcal O(h^{1/2}\tilde h^{1/2})_{\Syme}
\text{ near }\WFh(\Psi_1)
$$
for some $b\in\Syme(X)$. Take $B=\Op_h(b)\in\Psie$; then
$$
\Real A=C^{-1}+B^*B+\mathcal O(\tilde h)_{\Psie}\text{ microlocally near }
\WFh(\Psi_1).
$$
Therefore,
$$
\Real\langle A\Psi_1 u,\Psi_1 u\rangle=C^{-1}\|\Psi_1 u\|_{L^2}^2+\|B\Psi_1 u\|_{L^2}^2+\mathcal O(\tilde h)
\|\Psi_1 u\|_{L^2}^2+\mathcal O(h^\infty)\|u\|_{L^2}^2.
$$
The second term on the right-hand side is nonnegative; it remains to
take $\tilde h$ small enough so that the third term is absorbed by
the first one.
\end{proof}
%
%%%%%%%%%%%%%%%%%%%%%%%%%%%%%% END LEMMA %%%%%%%%%%%%%%%%%%%%%%%%%%%%%%
%
We will additionally need to work with symbols all of whose derivatives
grow according to~\eqref{e:1/2-definition}, but the symbols themselves
grow like $\log(1/h)$; those are the growth conditions satisfied by
the logarithmically flattened escape function from
Lemma~\ref{l:f-hat}. We need to know that certain properties
of the $\Psi_{1/2}$ calculus still hold in this setting:
%
%%%%%%%%%%%%%%%%%%%%%%%%%%%%%% BEGIN PROP %%%%%%%%%%%%%%%%%%%%%%%%%%%%%%
%
\begin{lemm}\label{l:kangaroo}
Assume that $a\in C_0^\infty(T^*X)$ satisfies
\begin{equation}
  \label{e:kangaroo}
a=\mathcal O(\log(1/h));\
\partial^\alpha_{x,\xi} a=\mathcal O((h/\tilde h)^{-|\alpha|/2}),\
|\alpha|>0.
\end{equation}
In particular, $a\in\log(1/h)\Syme(X)$.  Let $A=\Op_h(a)\in\log(1/h)\Psie(X)$
be defined by~\eqref{e:1/2-quantization}. Then:

1. The symbol~$\tilde\sigma(A)$ from~\eqref{e:1/2-symbol} is defined
modulo $\mathcal O(h^{1/2}\tilde h^{1/2})_{\Syme}$ (without the $\log(1/h)$ factor), and $\tilde\sigma(A)=a+\mathcal
O(h^{1/2}\tilde h^{1/2})_{\Syme}$.

2. If $B=\Op_h(b)$ for some $b$ satisfying~\eqref{e:kangaroo}, then
$[A,B]=\mathcal O(\tilde h)_{\Psie}$.

3. If $B\in\Psi^k(X)$, then $[A,B]\in h^{1/2}\tilde h^{1/2}\Psie(X)$
and~\eqref{e:1/2-comm} holds.
\end{lemm}
\begin{proof}
The remainder estimates of the asymptotic decompositions for the
$S_{1/2}(\mathbb R^n)$ calculus in~\eqref{e:1/2-asymptotic}
and~\cite[Lemma~3.3]{w-z} depend on bounds on derivatives of the
symbols involved, but not on the size of the symbols themselves.
Therefore, the remainders in parts~4--6 of
Lemma~\ref{l:1/2-properties-rn} are the same for symbols
satisfying~\eqref{e:kangaroo} as for $\Syme$. Note also
that the class of symbols satisfying~\eqref{e:kangaroo} is invariant
under changes of variables and multiplication by symbols in
$S^k_{\cl}(X)$ (though not by $\Syme(X)$). All the statements above
can now be verified in a straightforward fashion.
\end{proof}
%
%%%%%%%%%%%%%%%%%%%%%%%%%%%%%%% END PROP %%%%%%%%%%%%%%%%%%%%%%%%%%%%%%%
%
Finally, we establish the following analog of Egorov's Theorem,
needed in  \S\ref{s:approximation1}.  A direct argument
involving~\eqref{e:quantized-canonical} and the method of stationary
phase would give an $\mathcal O(\tilde h)$ error;
the slightly more delicate
argument  below yields an $\mathcal O(h^{1/2}\tilde h^{1/2})$ error for
the Weyl quantization, giving a natural generalization (and a different
 proof) of~\cite[Lemma~3.3]{w-z}.
%
%%%%%%%%%%%%%%%%%%%%%%%%%%%%%% BEGIN LEMMA %%%%%%%%%%%%%%%%%%%%%%%%%%%%%%
%
\begin{lemm}\label{l:1/2-egorov}
Let $X_1,X_2$ be two manifolds of the same dimension, $\varkappa$ a
symplectomorphism mapping an open subset of $T^*X_1$ onto an
open subset of $T^*X_2$, and $B:C^\infty(X_2)\to C^\infty(X_1)$ a
compactly microlocalized semiclassical Fourier integral operator
associated to $\varkappa$, in the sense of
\S\ref{s:prelim.canonical}. Assume that $A_j\in\Psie(X_j)$ are
such that
$$
\tilde\sigma(A_1)=\tilde\sigma(A_2)\circ\varkappa+\mathcal O(h^{1/2}\tilde h^{1/2})_{\Syme}
$$
near the projection of $\WFh(B)$ onto $T^*X_1$. Then
$$
A_1B=BA_2+\mathcal O(h^{1/2}\tilde h^{1/2});
$$
here $\mathcal O(h^{1/2}\tilde h^{1/2})$ is understood in the sense of~\eqref{e:b-bound},
as both sides of the equation are compactly microlocalized.
\end{lemm}
%%%%%%%%%%%%%%%%%%%%%%%%%%%%%%%%%%%%%%%%%%%%%%%%%%%%%%%%%%%%%%%%%%%%%%%%%%%%%%%%
\begin{proof}
Using a microlocal partition of unity, we may assume that $A_1$ is
microlocalized in a small neighborhood of some $(x_1,\xi_1)\in T^*X_1$
and $A_2$ is microlocalized in a small neighborhood of
$\varkappa(x_1,\xi_1)$.  We can then let $X_1=X_2=\mathbb R^n$
and quantize $\varkappa$ near $\WFh(A_1)\times\WFh(A_2)$ by the
unitary operators $(B_1,B_1^{-1})$ constructed at the end of
\S\ref{s:prelim.canonical}, multiplied by certain
pseudodifferential cutoffs. We will use the families of symplectomorphisms
$\varkappa_t$, Fourier integral operators $B_t$, and pseudodifferential
operators $Z_t$ from this construction.

Using the composition property~\eqref{i:composition}
of Fourier integral operators in \S\ref{s:prelim.canonical}, we see that 
$C_1=BB_1^{-1}$ and $C_2=B_1^{-1}B$ lie in $\Psic$
and by the standard
Egorov property~\eqref{i:egorov} in \S\ref{s:prelim.canonical},
$\sigma(C_1)=\sigma(C_2)\circ\varkappa$.
We then need to prove that
\begin{equation}
  \label{e:egorov-1}
A_1C_1=B_1(C_2A_2)B_1^{-1}+\mathcal O(h^{1/2}\tilde h^{1/2}),
\end{equation}
where we know that
\begin{equation}
  \label{e:egorov-sym}
\tilde\sigma(A_1C_1)=\tilde\sigma(C_2A_2)\circ\varkappa+\mathcal O(h^{1/2}\tilde h^{1/2})_{\Syme}.
\end{equation}
Let $A(t)=\Op_h(a(t))$, where
\begin{equation}
  \label{e:a-t}
a(t)=\tilde\sigma(A_1C_1)\circ\varkappa_t^{-1}
\end{equation}
and $\Op_h$ is defined by~\eqref{e:q-weyl}.
Then $A(0)=A_1C_1+\Resh$; by~\eqref{e:egorov-sym},
$A(1)=C_2A_2+\mathcal O(h^{1/2}\tilde h^{1/2})_{\Psie}$,
and~\eqref{e:egorov-1} reduces to
\begin{equation}
  \label{e:egorov-2}
B_1A(1)B_1^{-1}=A(0)+\mathcal O(h^{1/2}\tilde h^{1/2}).
\end{equation}
Using~\eqref{e:fio-evolution}, we get 
\begin{equation}
  \label{e:egorov-3}
hD_t (B_tA(t)B_t^{-1})=B_t([Z_t,A(t)]+hD_t A(t))B_t^{-1}.
\end{equation}
Now, by part~5 of Lemma~\ref{l:1/2-properties-rn} (which is where we need
the Weyl quantization),
$$
[Z_t,A(t)]+hD_t A(t)={h\over i}\Op_h(\{z_t,a(t)\}+\partial_t a(t))+\mathcal O(h^{3/2}\tilde h^{3/2});
$$
by~\eqref{e:a-t}, we get $\{z_t,a(t)\}+\partial_t a(t)=\partial_t(a(t)\circ\varkappa_t)\circ\varkappa_t^{-1}=0$
and thus the right-hand side of~\eqref{e:egorov-3} is $\mathcal O(h^{3/2}\tilde
h^{3/2})$.  Integrating in $t$ from 0 to 1, we get~\eqref{e:egorov-2}.
\end{proof}
%
%%%%%%%%%%%%%%%%%%%%%%%%%%%%%% END LEMMA %%%%%%%%%%%%%%%%%%%%%%%%%%%%%%
%

%%%%%%%%%%%%%%%%%%%%%%%%%%%%%%%%%%%%%%%%%%%%%%%%%%%%%%%%%%%%%%%%%%%%%%%%%%%%%%%%
%                                  SECTION 4.2                                 %
%%%%%%%%%%%%%%%%%%%%%%%%%%%%%%%%%%%%%%%%%%%%%%%%%%%%%%%%%%%%%%%%%%%%%%%%%%%%%%%%
\subsection{Second microlocalization}\label{s:2nd}

In this subsection, we study certain operators microlocalized
$\mathcal O(h/\tilde h)$ close to the energy surface $p^{-1}(0)$. (See
\S\ref{s:1/2} for why one needs the second semiclassical
parameter $\tilde h$; as always in this paper, we assume that $\tilde h$ is
small and $h$ is small enough depending on $\tilde h$.)  We need the
operator~$A$ from Lemma~\ref{l:main} to be localized $\mathcal
O(h/\tilde h)$ close to the energy surface to be able to approximate
it by an operator of rank $\mathcal O(h^{-\nu})$. Without this
additional localization, the rank of $A_R$ from Lemma~\ref{l:main},
and thus the number of resonances, would be estimated by $\mathcal
O(h^{-\nu-1})$; however, this estimate would be valid in an $o(1)$
spectral window (with the imaginary part of the resonances still
bounded by $C_0h$) instead of the $\mathcal O(h)$ one that we study.

Since $h/\tilde h\ll h^{1/2}$, operators microlocalized $\mathcal
O(h/\tilde h)$ close to the energy surface will not be
pseudodifferential even in the exotic classes studied in
\S\ref{s:1/2}. In fact, they will not even be pseudolocal; their
wavefront set will include transport along the Hamiltonian flow of $p$
on the energy surface. This presents a difficulty with constructing a
calculus of such operators; however, as shown in~\cite[\S5]{sj-z}, one
can still quantize symbols which are regular on the scale $h/\tilde h$
in the direction transversal to the energy surface, on the scale $1$
in the direction of the Hamiltonian flow of $p$, and on the scale
$(h/\tilde h)^{1/2}$ in all other directions.

The second microlocal calculus of~\cite{sj-z} is rather involved and
we do not use it here, proceeding instead as follows. Let $\hat p$ be a real-valued symbol with
\[
\hat p = p \quad \textrm{near the trapped set } K,
\]
and which is elliptic near fiber infinity.
We quantize $\hat p$ to a self-adjoint operator
$\widehat P$ (on the compact manifold $X$ introduced in
\S\ref{s:ah.vasy}), and use an operator of the form
$\chi((\tilde h/h)\widehat P)$, where $\chi\in C_0^\infty$; this
operator is microlocalized $\mathcal O(\tilde h/h)$ close to the
surface $\hat p^{-1}(0)$. In this subsection, we use spectral theory to
get  estimates on the resulting operator in an abstract
setting; we will apply these to our problem in \S\ref{s:approximation1}
and Lemma~\ref{l:estimate-k-aux}.

Throughout this subsection, $X$ is a compact manifold without boundary
and with a prescribed volume form, and $\widehat P\in \Psi^k(X)$, $k>0$, is
a  symmetric pseudodifferential operator.
Let $\hat p$ be the principal symbol of $\widehat P$ (not to be
confused with the notation for $\lxir^{-2}p$ used in~\cite{v2}). 
Assume that $\hat p$ is elliptic near the fiber infinity; namely, the
characteristic set $\{\lxir^{-k}\hat p=0\}$ does not intersect $S^*X$
(and thus can be written $\hat p^{-1}(0)$).  Then
$\widehat P$ is self-adjoint on $L^2(X)$ with domain $\Hh^k(X)$ and
compact resolvent (see for example~\cite[\S7.10]{tay}).

For any bounded Borel measurable function $\chi$ on $\mathbb R$,
define $\chi((\tilde h/h)\widehat P)$ by 
spectral theory (see for example~\cite[Chapter~8]{tay}). This operator
is bounded on $\Hh^s(X)$ for each $s$, uniformly in $h,\tilde
h$. Indeed, this is true for $s=0$ by spectral theory, for $s \in k \mathbb Z$ 
by commuting with $i+\widehat P$, which is an
isomorphism $\Hh^{s+k}\to\Hh^s$ for all $s$, and for general $s$ by
interpolation. Note also that the $\Hh^s(X)$ operator norm of
$\chi((\tilde h/h)\widehat P)$ depends only on $s,\widehat P$, and
$\sup|\chi|$. In particular, the unitary operator $e^{it(\tilde
h/h)\widehat P}$ is bounded on each $\Hh^s$ uniformly in $t$.

We first show that for $\chi$ Schwartz, the operator $\chi((\tilde
h/h)\widehat P)$ is microlocalized on $\hat p^{-1}(0)$:
%
%%%%%%%%%%%%%%%%%%%%%%%%%%%%%% BEGIN LEMMA %%%%%%%%%%%%%%%%%%%%%%%%%%%%%%
%
\begin{lemm}\label{l:2nd-compact}
Let $\chi\in\mathscr S(\mathbb R)$. Then $\chi((\tilde h/h)\widehat P)$
is a compactly microlocalized operator of norm $\mathcal O(1)$,
in the sense of \S\ref{s:prelim.basics}, and
$$
\WFh(\chi((\tilde h/h)\widehat P))\subset \hat p^{-1}(0)\times\hat p^{-1}(0).
$$
\end{lemm}
%%%%%%%%%%%%%%%%%%%%%%%%%%%%%%%%%%%%%%%%%%%%%%%%%%%%%%%%%%%%%%%%%%%%%%%%%%%%%%%%
\begin{proof}
It suffices to show that if $\Psi_1\in \Psi^l(X)$ satisfies $\WFh(\Psi_1)\cap \hat p^{-1}(0)=\emptyset$, then
$$
\chi((\tilde h/h)\widehat P)\Psi_1=\mathcal O(h^\infty)_{\Psi^{-\infty}},\
\Psi_1\chi((\tilde h/h)\widehat P)=\mathcal O(h^\infty)_{\Psi^{-\infty}}.
$$
We prove the first statement. Take a large positive integer $N$.
Since $\widehat P$ is elliptic near $\WFh(\Psi_1)$,
there exists
$\Psi_2\in\Psi^{l-kN}(X)$ such that $\Psi_1=\widehat P^N\Psi_2+\mathcal O(h^\infty)_{\Psi^{-\infty}}$.
We then have
$$
\chi((\tilde h/h)\widehat P)\Psi_1=h^{N}\tilde h^{-N}\chi_N((\tilde h/h)\widehat P)\Psi_2
+\mathcal O(h^\infty)_{\Psi^{-\infty}}.
$$
Here $\chi_N(\lambda)=\lambda^N\chi(\lambda)$ is Schwartz.
The first term on the right-hand side is $\mathcal O(h^{N-1})_{\Hh^s\to \Hh^{s+kN-l}}$ for all $s$,
for $h$ small enough depending on $\tilde h$ (for example, for $h<e^{-1/\tilde h}$);
it remains to let $N$ go to infinity. 
\end{proof}
%
%%%%%%%%%%%%%%%%%%%%%%%%%%%%%% END LEMMA %%%%%%%%%%%%%%%%%%%%%%%%%%%%%%
%
To establish further properties of $\chi((\tilde h/h)\widehat P)$,
we use the following 
%
%%%%%%%%%%%%%%%%%%%%%%%%%%%%%% BEGIN LEMMA %%%%%%%%%%%%%%%%%%%%%%%%%%%%%%
%
\begin{lemm}\label{l:crazy}
Let $\chi\in \mathscr S(\mathbb R)$, and $B:C^\infty(X)\to
C^\infty(X)$ be a polynomially bounded operator in the sense of
\S\ref{s:prelim.basics}. Then for each $s,s'$ and each integer
$N\geq 0$,
\begin{equation}
  \label{e:crazy-formula}
\begin{gathered}
\chi((\tilde h/h) \widehat P) B
=\sum_{0\leq j<N} {(\tilde h/ h)^j\over j!}(\ad_{\widehat P}^j B)\chi^{(j)}((\tilde h/h)\widehat P)\\
+\mathcal O((\tilde h/h)^N\|\ad_{\widehat P}^N B\|_{\Hh^s\to\Hh^{s'}})_{\Hh^s\to\Hh^{s'}}.
\end{gathered}
\end{equation}
Here $\chi^{(j)}$ denotes $j$-th derivative of $\chi$ and
$\ad_{\widehat P} A=[\widehat P,A]$ for any $A$.  The constant in
$\mathcal O(\cdot)$ depends on $\chi,N,s,s',\widehat P$, but not on
$B,h,\tilde h$.
\end{lemm}
%%%%%%%%%%%%%%%%%%%%%%%%%%%%%%%%%%%%%%%%%%%%%%%%%%%%%%%%%%%%%%%%%%%%%%%%%%%%%%%%
\begin{proof}
By the Fourier inversion formula,
\begin{equation}\label{e:crazy-fourier}
\chi((\tilde h/h)\widehat P)={1\over 2\pi}\int \hat \chi(t)e^{it(\tilde h/h)\widehat P}\,dt.
\end{equation}
Here $\hat \chi\in \mathscr S(\mathbb R)$ is the Fourier transform of $\chi$.
Now, for each $j$,
$$
\partial_t^j (e^{it(\tilde h/h)\widehat P} Be^{-it(\tilde h/h)\widehat P})
=(i\tilde h/h)^je^{it(\tilde h/h) \widehat P}(\ad_{\widehat P}^j  B)e^{-it(\tilde h/h) \widehat P};
$$
since $e^{\pm it (\tilde h/h)\widehat P}$ is bounded uniformly in $t$ on each $\Hh^s(X)$, we have
by Taylor's formula
$$
\bigg\|e^{it(\tilde h/h) \widehat P} Be^{-it(\tilde h/h)\widehat P}-\sum_{0\leq j<N} {(it\tilde h/h)^j\over j!}
\ad_{\widehat P}^j B\bigg\|_{\Hh^s\to\Hh^{s'}}\leq C|t\tilde h/h|^N\|\ad_{\widehat P}^N B\|_{\Hh^s\to\Hh^{s'}}.
$$
It remains to multiply the operator in the left-hand side by $e^{it(\tilde h/h)\widehat P}$
on the right and substitute into~\eqref{e:crazy-fourier}.
\end{proof}
%
%%%%%%%%%%%%%%%%%%%%%%%%%%%%%% END LEMMA %%%%%%%%%%%%%%%%%%%%%%%%%%%%%%
%
The operator $\chi((\tilde h/h)\widehat P)$ is not
$h$-pseudodifferential.  As mentioned in the beginning of the
subsection, we expect it to have nonlocal contributions
corresponding to transport along the Hamiltonian flow for all
times. To see this, recall~\eqref{e:crazy-fourier} and the fact that
$e^{it(\tilde h/h)\widehat P}$ is a Fourier integral operator
associated to the Hamiltonian flow of $\hat p$ at time $t\tilde h$ (see
also~\cite[\S5.4]{sj-z}).  However, the nonlocal part of 
$\chi((\tilde h/h)\widehat P)$ decays rapidly with
respect to $\tilde h$, and commuting with certain pseudodifferential
operators produces a power of $\tilde h$:
%
%%%%%%%%%%%%%%%%%%%%%%%%%%%%%% BEGIN LEMMA %%%%%%%%%%%%%%%%%%%%%%%%%%%%%%
%
\begin{lemm}\label{l:second-mic}
Let $\chi\in \mathscr S(\mathbb{R})$. Then:
%%%%%%%%%%%%%%%%%%%%%%%%%%%%%%
\begin{enumerate}
  \item if $\Psi_1,\Psi_2\in\Psi^0(X)$ have $\WFh(\Psi_1)\cap\WFh(\Psi_2)=\emptyset$,
then
$$
\Psi_1\chi((\tilde h/h)\widehat P)\Psi_2=\mathcal O(\tilde h^\infty);
$$
  \item if $\Psi_1\in\Psi^0(X)$, or both $\Psi_1\in\Psie(X)$ and
$H_{\hat p}\tilde\sigma(\Psi_1)=\mathcal O(1)_{\Syme}$, then
$$
[\chi((\tilde h/h)\widehat P),\Psi_1]=\mathcal O(\tilde h).
$$
\end{enumerate}
%%%%%%%%%%%%%%%%%%%%%%%%%%%%%%
In both cases, the $\mathcal O(\cdot)$ is understood in the sense of~\eqref{e:b-bound},
as the left-hand sides are compactly microlocalized by Lemma~\ref{l:2nd-compact}.
\end{lemm}
%%%%%%%%%%%%%%%%%%%%%%%%%%%%%%%%%%%%%%%%%%%%%%%%%%%%%%%%%%%%%%%%%%%%%%%%%%%%%%%%
\begin{proof} 
(1) We apply Lemma~\ref{l:crazy} with $B=\Psi_2$.
Since $\ad_{\widehat P}^N \Psi_2=\mathcal O(h^N)_{\Psi^{N(k-1)}}$, we have for each $N$ and each $s$,
$$
\Psi_1\chi((\tilde h/h)\widehat P)\Psi_2=\sum_{0\leq j<N}{(\tilde h/h)^j\over j!}\Psi_1
(\ad_{\widehat P}^j\Psi_2)\chi^{(j)}((\tilde h/h)\widehat P)+\mathcal O(\tilde h^N)_{\Hh^s\to \Hh^{s+N(1-k)}};
$$
each term in the sum is $\Resh$ as $\Psi_1(\ad_{\widehat P}^j\Psi_2)=\Resh$.
It remains to let $N\to\infty$.

(2) We have $[\widehat P,\Psi_1]=\mathcal O(h)_{L^2\to L^2}$ (see part~7 of
Lemma~\ref{l:1/2-properties} for the second case); it remains to apply
Lemma~\ref{l:crazy} with $B=\Psi_1$ and $N=1$.
\end{proof}
%
%%%%%%%%%%%%%%%%%%%%%%%%%%%%%% END LEMMA %%%%%%%%%%%%%%%%%%%%%%%%%%%%%%
%
Finally, we establish a version of Egorov's theorem, needed in  \S\ref{s:approximation1}.
%
%%%%%%%%%%%%%%%%%%%%%%%%%%%%%% BEGIN LEMMA %%%%%%%%%%%%%%%%%%%%%%%%%%%%%%
%
\begin{lemm}\label{l:2nd-egorov}
Let $X_1,X_2$ be two compact manifolds of the same dimension,
$\varkappa$ be a symplectomorphism mapping an open subset of $T^*X_1$
onto an open subset of $T^*X_2$, and $B:C^\infty(X_2)\to
C^\infty(X_1)$ be a compactly microlocalized semiclassical Fourier
integral operator associated to $\varkappa$, in the sense of
\S\ref{s:prelim.canonical}.  Assume that $\widehat
P_j\in\Psi^{k_j}(X_j)$, $k_j>0$, are symmetric operators
elliptic near $S^* X_j$ and 
$$
\sigma(\widehat P_1)=\sigma(\widehat P_2)\circ\varkappa
$$
near the projection of $\WFh(B)$ onto $T^*X_1$.
Then for each $\chi\in \mathscr S(\mathbb R)$,
$$
\chi((\tilde h/h)\widehat P_1)B=B\chi((\tilde h/h)\widehat P_2)+\mathcal O(\tilde h),
$$
with $\mathcal O(\tilde h)$ understood in the sense of~\eqref{e:b-bound}.
\end{lemm}
%%%%%%%%%%%%%%%%%%%%%%%%%%%%%%%%%%%%%%%%%%%%%%%%%%%%%%%%%%%%%%%%%%%%%%%%%%%%%%%%
\begin{proof}
As in the proof of Lemma~\ref{l:crazy}, we use the Fourier inversion formula:
$$
\chi((\tilde h/h)\widehat P_1)B-B\chi((\tilde h/h)\widehat P_2)
={1\over 2\pi}\int \hat\chi(t)(e^{it(\tilde h/h)\widehat P_1}B-Be^{it(\tilde h/h)\widehat P_2})\,dt.
$$
It is then enough to prove that
$$
e^{it(\tilde h/h)\widehat P_1}B-Be^{it(\tilde h/h)\widehat P_2}=\mathcal O(t\tilde h)_{L^2\to L^2}.
$$
Multiply the left-hand side by $e^{-it(\tilde h/h)\widehat P_2}$ on the right and differentiate in $t$:
we get
$$
i(\tilde h/h)e^{it(\tilde h/h)\widehat P_1}(\widehat P_1 B -B\widehat P_2)e^{-it(\tilde h/h)\widehat P_2};
$$
this expression is $\mathcal O(\tilde h)_{L^2\to L^2}$ uniformly in $t$
by the standard Egorov property~\eqref{i:egorov} in \S\ref{s:prelim.canonical},
as $\widehat P_1B-B\widehat P_2=\mathcal O(h)$. It remains to integrate in $t$.
\end{proof}
%
%%%%%%%%%%%%%%%%%%%%%%%%%%%%%% END LEMMA %%%%%%%%%%%%%%%%%%%%%%%%%%%%%%
%

%%%%%%%%%%%%%%%%%%%%%%%%%%%%%%%%%%%%%%%%%%%%%%%%%%%%%%%%%%%%%%%%%%%%%%%%%%%%%%%%
%                                  SECTION 4.3                                 %
%%%%%%%%%%%%%%%%%%%%%%%%%%%%%%%%%%%%%%%%%%%%%%%%%%%%%%%%%%%%%%%%%%%%%%%%%%%%%%%%
\section{Approximation by finite rank operators}
  \label{s:approximation}
  
In this section, we prove the following analog of \cite[Proposition~5.10]{sj-z}:
%
%%%%%%%%%%%%%%%%%%%%%%%%%%%%%% BEGIN LEMMA %%%%%%%%%%%%%%%%%%%%%%%%%%%%%%
%
\begin{lemm}\label{l:approximation}
Let $X$ be a compact manifold, $\widehat P\in\Psi^k(X)$, $k>0$, a
symmetric operator with principal symbol $\hat
p=\sigma(\widehat P)$ elliptic outside of a compact set,
and $\widetilde A\in\Psie(X)$ is such that $\hat p$ has no critical points
on $\hat p^{-1}(0)\cap\WFh(\widetilde A)$.
Assume moreover that $\tilde\sigma(\widetilde A)=\tilde a+\mathcal
O(h^{1/2}\tilde h^{1/2})_{\Syme}$, where $\tilde a\in
C_0^\infty(T^*X)\cap \Syme(X)$ and there exists a constant $\nu \ge 0$ such that for each
$R>0$, there exists $C>0$ such that 
\begin{equation}
  \label{e:vol-est}
\begin{gathered}
\Vol_{\hat p^{-1}(0)}\{\exp(tH_{\hat p})(x,\xi)\mid
|t|\leq R,\\
(x,\xi)\in
(\supp\tilde a\cap \hat p^{-1}(0))+B_{\hat p^{-1}(0)}(R(h/\tilde h)^{1/2})\}
\leq C(h/\tilde h)^{n-1-\nu}.
\end{gathered}
\end{equation}
Here $\Vol_{\hat p^{-1}(0)}$ denotes the volume
with respect to the Liouville measure on $\hat p^{-1}(0)$ and
$V+B_{\hat p^{-1}(0)}(r)$ denotes the set of points in $\hat
p^{-1}(0)$ lying distance at most $r$ away from $V\subset \hat
p^{-1}(0)$ (with respect to some fixed smooth metric).  Finally, let
$\chi\in C_0^\infty(\mathbb R)$ and define
$$
A=\chi((\tilde h/h)\widehat P)\widetilde A.
$$
Then we can write $A=A_R+A_E$, where $A_R,A_E:C^\infty(X)\to C^\infty(X)$
are compactly microlocalized and for some constant $C(\tilde h)$ independent
of $h$,
$$
A_R=\mathcal O(1),\
A_E=\mathcal O(\tilde h),\
\rank A_R\leq C(\tilde h)h^{-\nu}.
$$
Here $\mathcal O(\cdot)$ is understood in the sense of~\eqref{e:b-bound}.
\end{lemm}
%
%%%%%%%%%%%%%%%%%%%%%%%%%%%%%% END LEMMA %%%%%%%%%%%%%%%%%%%%%%%%%%%%%%
%
Lemma~\ref{l:approximation} is the main component needed to
approximate the operator $A$ from Lemma~\ref{l:main} by a finite rank
operator. In our case, the volume estimate~\eqref{e:vol-est} is a
direct consequence of the definition of the upper Minkowski dimension
of the trapped set  \eqref{e:minkdef} and the fact that the symbol $\tilde a$ will be
supported $\mathcal O((h/\tilde h)^{1/2})$ close to the trapped set.
See \S\ref{s:ultimate3} for details.

To prove Lemma~\ref{l:approximation}, we will first, in \S\ref{s:approximation1}, conjugate the
operator $A$ locally by a  semiclassical Fourier integral
operator to make $\hat p=\xi_1$  thus trivializing the `second
microlocalized' factor $\chi((\tilde h/h)\widehat P)$ of $A$. To
simplify the discussion, we drop the second
microlocalized factor in this paragraph and explain how to approximate
$\widetilde A$ rather than  $A$. For that,  cover $\supp \tilde a$
 by balls of size $\sim(h/\tilde h)^{1/2}$ (the analog of this step is carried out in \S\ref{s:approximation2}); then to each such
ball,  associate an operator of rank $\mathcal O(1)$ which is a
function of a quantum harmonic oscillator, shifted to be centered in that ball (see \S\ref{s:approximation3}).  The sum
$\widetilde\Psi_\Pi$ of these operators will be elliptic near
$\WFh(\widetilde A)$, thus we can approximate $\widetilde A$ by a
multiple of $\widetilde\Psi_\Pi$ (see \S\S\ref{s:approximation3}--\ref{s:approximation4}). However, the rank of
$\widetilde\Psi_\Pi$ is bounded by a constant times the number of
balls of size $(h/\tilde h)^{1/2}$ that are needed to cover
$\supp\tilde a$; this number can be estimated by the volume
in~\eqref{e:vol-est}.

We generally follow the proof of~\cite[Proposition~5.10]{sj-z} (see
also~\cite[\S6.4]{e-z} for an application of some of the ideas
used in a simpler setting), with the following two differences. First
of all, we treat the operators microlocalized $\mathcal O(h/\tilde h)$
near the energy surface as in \S\ref{s:2nd}, rather than
using~\cite[\S5]{sj-z}.  Secondly, we prove in detail (see
Lemma~\ref{l:harmonic}) that the operator $\widetilde\Psi_\Pi$ lies
in~$\Psie$. This is not trivial because, even though the operator
associated to each ball lies in $\Psie$, we sum $\sim(\tilde
h/h)^{n-1-\nu}$ many such operators; this step is skipped in~\cite{sj-z}.

One of the anonymous referees of this paper has suggested an alternative approach to this lemma. If one arranges that $A$ is a positive operator, and shows that the trace of $A$ is bounded by $C h^{-n} (h/\tilde h)^{n-\nu}$, or more generally by $C(\tilde h)h^{-\nu}$, then $A$ can have no more than $C(\tilde h)\tilde h^{-1}h^{-\nu}$ many eigenvalues greater than $\tilde h$, and the decomposition follows. If $A$ were pseudodifferential, its trace could be computed by integrating its full symbol as in \cite[Theorem 9.4]{d-s}. The factor $\chi((\tilde h/h)\widehat P)$ prevents one from applying this approach directly, but this difficulty could likely be overcome using the methods of \S\ref{s:2nd}.

We prove Lemma~\ref{l:approximation} over the course of the following four subsections:

%%%%%%%%%%%%%%%%%%%%%%%%%%%%%%%%%%%%%%%%%%%%%%%%%%%%%%%%%%%%%%%%%%%%%%%%%%%%%%%%
\subsection{Reduction to a model problem}\label{s:approximation1} We reduce the general case to the following model case:
\[
X=X_M=\mathbb S^1_{x_1}\times \mathbb R^{n-1}_{x'}, \qquad \widehat P=hD_{x_1}.
\]
The manifold $X_M$ is not compact, so, strictly speaking, the
statement of Lemma~\ref{l:approximation} does not apply. However, the
only place where we use the compactness of $X$ in the proof is the
application of Lemma~\ref{l:2nd-egorov} in this subsection; since
$hD_{x_1}$ is self-adjoint on $L^2(X_M)$ and the associated unitary
operator $e^{it(\tilde h/h)hD_{x_1}}$ is a shift in the $x_1$ variable
and thus bounded on each $\Hh^s(X_M)$ uniformly in $t$, the proof of
this lemma still goes through. The operator
$\chi(\tilde hD_{x_1})$ is a Fourier series multiplier in the
$x_1$ variable; therefore, it is properly supported and
polynomially bounded in the sense of \S\ref{s:prelim.basics};
moreover, the product of this operator with a compactly
microlocalized operator will also be compactly microlocalized.

We now construct finitely many operators $\Psi_j\in \Psic(X)$ such that
%%%%%%%%%%%%%%%%%%%%%%%%%%%%%%
\begin{enumerate}
  \item $\sum_j\Psi_j=1$ microlocally near $\hat p^{-1}(0)\cap\WFh(\widetilde A)$;
  \item for each $j$, there exists a symplectomorphism
$\varkappa_j$ from a neighborhood of $\WFh(\Psi_j)$ onto
some open subset of $T^*X_M$ such that $\hat p=\xi_1\circ\varkappa_j$
on the domain of $\varkappa_j$;
  \item for each $j$, there exist compactly microlocalized
semiclassical Fourier integral operators
$$
B_j:C^\infty(X_M)\to C^\infty(X),\
B'_j:C^\infty(X)\to C^\infty(X_M),
$$
associated to $\varkappa_j$ and $\varkappa_j^{-1}$, respectively,
such that $B_jB'_j=1$ microlocally near $\WFh(\Psi_j)$.
\end{enumerate}
%%%%%%%%%%%%%%%%%%%%%%%%%%%%%%
Indeed, by the Darboux theorem (see for
example~\cite[Theorem~12.1]{e-z}) each $(x,\xi)\in \hat p^{-1}(0)\cap\WFh(\widetilde A)$ has
a neighborhood $U_{(x,\xi)}\subset T^*X$ with a symplectomorphism
$\varkappa:U_{(x,\xi)}\to T^* X_M$ such that $\hat
p=\xi_1\circ\varkappa_{(x,\xi)}$ on $U_{(x,\xi)}$. Using the method
described at the end of \S\ref{s:prelim.canonical}, we can find
compactly microlocalized semiclassical Fourier integral operators
$B_{(x,\xi)}:C^\infty(X_M)\to C^\infty(X)$,
$B'_{(x,\xi)}:C^\infty(X)\to C^\infty(X_M)$ quantizing $\varkappa$
near the closure of
$V_{(x,\xi)}\times\varkappa_{(x,\xi)}(V_{(x,\xi)})$, where
$V_{(x,\xi)}\subset U_{(x,\xi)}$ is some neighborhood of $(x,\xi)$.
It remains to choose $\Psi_j$ as a microlocal partition of unity
subordinate to an open cover of $\hat p^{-1}(0)$ by finitely many of
the sets $V_{(x,\xi)}$.

By Lemma~\ref{l:2nd-egorov} (see the remark about non-compactness of $X_M$ in
the beginning of this subsection), we have
$$
B'_j\chi((\tilde h/h)\widehat P)=\chi(\tilde hD_{x_1})B'_j+\mathcal O(\tilde h);
$$
here $\mathcal O(\tilde h)$ is understood in the sense of~\eqref{e:b-bound},
as both sides of the equation are compactly microlocalized.
Next, let $\chi_j\in C_0^\infty(T^*X)$ be supported inside the domain of $\varkappa_j$,
but $\chi_j=1$ near the projection of $\WFh(B'_j)$ onto $T^*X$, and let
$\widetilde A_j\in\Psi^{\comp}_{1/2}(X_M)$ satisfy
$$
\tilde\sigma(\widetilde A_j)=\tilde a_j+\mathcal O(h^{1/2}\tilde h^{1/2}),\
\tilde a_j=(\chi_j\tilde a)\circ\varkappa_j^{-1};
$$
here $\tilde a_j$ is extended by zero outside of the image of $\varkappa_j$.
By Lemma~\ref{l:1/2-egorov},
$$
B'_j\widetilde A=\widetilde A_jB'_j+\mathcal O(h^{1/2}\tilde h^{1/2});
$$
moreover, $\tilde a_j$ satisfies the volume bound~\eqref{e:vol-est},
with $\xi_1$ taking the place of $\hat p$.
By Lemmas~\ref{l:2nd-compact} and~\ref{l:second-mic}(1),
$$
\begin{gathered}
A=\sum_j \Psi_j\chi((\tilde h/h)\widehat P)\widetilde A+\mathcal O(\tilde h^\infty)
=\sum_j \Psi_jB_j B'_j\chi((\tilde h/h)\widehat P)\widetilde A+\mathcal O(\tilde h^\infty)\\
=\sum_j \Psi_jB_j \chi(\tilde hD_{x_1})\widetilde A_jB'_j+\mathcal O(\tilde h).
\end{gathered}
$$
It now suffices to establish the decomposition for each of the operators
$\chi(\tilde hD_{x_1})\widetilde A_j$ (bearing in mind that $B_j,B'_j$ have
norm $\mathcal O(1)$~--- see property~\ref{i:bdd} in \S\ref{s:prelim.canonical}).
Therefore, we henceforth assume that $X=X_M$ and $\widehat P=hD_{x_1}$.

%%%%%%%%%%%%%%%%%%%%%%%%%%%%%%%%%%%%%%%%%%%%%%%%%%%%%%%%%%%%%%%%%%%%%%%%%%%%%%%%
\subsection{Covering by cylinders}\label{s:approximation2} We now cover $\supp \tilde a\cap\hat p^{-1}(0)$ by cylinders. 

We write $(x,\xi)\in T^* X_M$ as $(x_1,x',\xi_1,\xi')$, where
$x_1\in \mathbb S^1$, $\xi_1\in \mathbb R$, and $x',\xi'\in \mathbb R^{n-1}$.
The energy surface is $\hat p^{-1}(0)=\{\xi_1=0\}$. The Hamiltonian flow
of $\hat p$ is $2\pi$-periodic; thus, for given $(x,\xi)$ and $R\geq \pi$, the set
$\{\exp(tH_{\hat p}(x,\xi))\mid |t|\leq R\}$ is equal to the
circle in the direction of the $x_1$ variable passing through $(x,\xi)$.
If $V\subset T^* X_M$, define
\begin{equation}\label{e:def-pi}
\Pi(V)=\{(x',\xi')\mid \exists x_1: (x_1,x',0,\xi')\in V\}\subset \mathbb R^{2(n-1)}.
\end{equation}
The definition~\eqref{e:def-pi} is useful for handling the second
microlocalization in \S\ref{s:approximation4}.  In fact, covering
$\Pi(V)$ by $(h/\tilde h)^{1/2}$ sized balls is morally the same as
covering $V\cap \hat p^{-1}([-h/\tilde h,h/\tilde h])$ (which is where
the operator $A$ is microlocalized) by cylinders of size $h/\tilde h$
in the direction transversal to $\hat p^{-1}(0)$, of size $1$ in the
direction of the Hamiltonian flow of $\hat p$, and of size $(h/\tilde
h)^{1/2}$ in all other directions; to each such cylinder will
correspond an operator of rank $\mathcal O(1)$, which is essentially
the product of a $C_0^\infty$ function of $\tilde h D_{x_1}$ and
spectral projector associated to a shifted harmonic oscillator.  This
subsection and \S\ref{s:approximation3} will handle the $(x',\xi')$
directions, constructing the covering by balls and the corresponding
finite rank operators.

The volume estimate~\eqref{e:vol-est} with $R = \pi$ implies (the choice of $1/2$
is convenient later)
$$
\Vol_{x',\xi'}\Big(\Pi(\supp \tilde a)
+B\Big(0,{1\over 2}(h/\tilde h)^{1/2}\Big)\Big)\leq C(h/\tilde h)^{n-1-\nu};
$$
here $B(\rho,r)$ denotes the closed Euclidean ball of radius $r$
centered at $\rho$. Following~\cite[Lemma~3.3]{sj}, take a maximal set
of points
$$
(x'_l,\xi'_l)\in\Pi(\supp \tilde a),\
1\leq l\leq M(h,\tilde h),
$$
such that the distance between any two distinct points in this set is
$>(h/\tilde h)^{1/2}$. Then
$$
\bigsqcup_{l=1}^{M(h,\tilde h)}B\Big((x'_l,\xi'_l),{1\over 2}(h/\tilde h)^{1/2}\Big)
\subset \Pi(\supp\tilde a)+B\Big(0,{1\over 2}(h/\tilde h)^{1/2}\Big);
$$
therefore, comparing the volumes of the two sides, we get
\begin{equation}\label{e:m-h-bound}
M(h,\tilde h)\leq C(h/\tilde h)^{-\nu}.
\end{equation}
However, we also know by maximality that
\begin{equation}\label{e:covering}
\Pi(\supp\tilde a)\subset \bigcup_{l=1}^{M(h,\tilde h)} B((x'_l,\xi'_l),(h/\tilde h)^{1/2}).
\end{equation}

%%%%%%%%%%%%%%%%%%%%%%%%%%%%%%%%%%%%%%%%%%%%%%%%%%%%%%%%%%%%%%%%%%%%%%%%%%%%%%%%
\subsection{The finite rank operator}\label{s:approximation3} To each ball in the covering, we associate 
a finite rank operator constructed using a function of a shifted
quantum harmonic oscillator.  The sum of these operators will generate
the finite rank term in the decomposition. As noted in the beginning
of this section, since the number of balls in the covering grows
polynomially in $h$, we need to take care when summing up the
corresponding operators.

Consider the following shifted harmonic oscillators on $\mathbb R^{n-1}$
(here $x'_{l,i},\xi'_{l,i}$ are the $i$th coordinates of $x'_l$, $\xi'_l$ respectively)
$$
\widetilde P_l=\sum_{i=1}^{n-1}(hD_{x'_i}-\xi'_{l,i})^2+(x'_i-x'_{l,i})^2.
$$
Let $\tilde\chi\in C_0^\infty(-2,2)$ be nonnegative and equal to $1$ on $(-1,1)$
and put
$$
\widetilde\Psi_l=\tilde\chi((\tilde h/h)\widetilde P_l).
$$
%
%%%%%%%%%%%%%%%%%%%%%%%%%%%%%% BEGIN LEMMA %%%%%%%%%%%%%%%%%%%%%%%%%%%%%%
%
\begin{lemm}\label{l:harmonic}
Take $\chi_b\in C_0^\infty(\mathbb R^{n-1})$ equal to $1$ near the projection
of $\Pi(\supp\tilde a)$ onto the base space $\mathbb R^{n-1}_{x'}$. Then the sum
$$
\widetilde\Psi_\Pi=\sum_{l=1}^{M(h,\tilde h)}\chi_b(x')\widetilde\Psi_l\chi_b(x')
$$
lies in $\Psie(\mathbb R^{n-1})$ (the cutoff $\chi_b$ is needed
because the operators $\widetilde\Psi_l$ are not properly supported), and its
symbol $\tilde s_\Pi=\tilde\sigma(\widetilde\Psi_\Pi)$ satisfies
$$
|\tilde s_\Pi|\geq 1-\mathcal O(\tilde h)
$$
on $\Pi(\supp\tilde a)$. Moreover,
if we consider $\widetilde\Psi_\Pi$ as an operator on $L^2(\mathbb R^{n-1})$,
then
\begin{equation}\label{e:harmonic-rank}
\rank \widetilde\Psi_\Pi\leq C\tilde h^{\nu+1-n}h^{-\nu}.
\end{equation}
\end{lemm}
%
%%%%%%%%%%%%%%%%%%%%%%%%%%%%%% END LEMMA %%%%%%%%%%%%%%%%%%%%%%%%%%%%%%
%
For \eqref{e:harmonic-rank}, we use~\eqref{e:m-h-bound}
and the fact that the rank of each $\widetilde\Psi_l$ is bounded by $C\tilde h^{1-n}$.
The latter fact follows from Weyl's law for the eigenvalues for the harmonic oscillator
(see for example~\cite[Theorem~6.3]{e-z})
and the fact that $\tilde\chi$ is compactly supported. 

Since functions of the quantum harmonic oscillator are not properly
supported operators, and to facilitate the rescaling argument in
the proof of Lemma~\ref{l:harmonic-internal}, we 
use a variation of the $\Psie(\mathbb R^{n-1})$ calculus.
Define the operator classes $\widehat S$ and $\widehat S_{1/2}$
on $T^* \mathbb R^{n-1}$ as follows:
\begin{equation}
  \label{e:1/2-rn-symbols}
\begin{gathered}
a(x',\xi';h,\tilde h)\in \widehat S\iff \forall\alpha\forall N
\sup_{(x',\xi')\in T^* \mathbb R^{n-1}}
|\partial^\alpha_{x',\xi'} a|\leq C_{\alpha N}\langle(x',\xi')\rangle^{-N};\\
a(x',\xi';h,\tilde h)\in \widehat S_{1/2}\iff \forall\alpha\forall N
\sup_{(x',\xi')\in T^* \mathbb R^{n-1}}
|\partial^\alpha_{x',\xi'} a|\leq C_{\alpha N} (h/\tilde h)^{-|\alpha|/2}\langle (x',\xi')\rangle^{-N}.
\end{gathered}
\end{equation}
Clearly $\widehat S\subset\widehat S_{1/2}$. The difference
between $\widehat S_{1/2}$ and the class $\Syme$ from \S\ref{s:1/2} is that we do not
require compact essential support, imposing instead uniform bounds on
the derivatives of the symbol as $(x',\xi')\to\infty$. However,
$\Syme(\mathbb R^{n-1}) \not\subset \widehat
S_{1/2}$, as the former does not require uniform bounds
as $x'\to\infty$.  For $a\in \widehat S_{1/2}$, we define its Weyl
quantization $\widehat\Op_h(a)$ as
\begin{equation}
  \label{e:1/2-weyl-q}
\widehat\Op_h(a)u(x')=(2\pi h)^{1-n}\int e^{{i\over h}(x'-y')\cdot\xi'}a\Big({x'+y'\over 2},\xi'\Big)
u(y')\,d\xi'dy'.
\end{equation}
The difference from~\eqref{e:q-weyl} is the lack of the cutoff
$\check\chi(x-y)$; because of this, the operator $\widehat\Op_h(a)$
need not be properly supported. However, this operator acts on the
space of Schwartz functions on $\mathbb R^{n-1}$, as well as on
$L^2(\mathbb R^{n-1})$, therefore one can still compose two such
operators.  The resulting calculus has the properties listed in 
Lemma~\ref{l:1/2-properties-rn}, parts~2--5, 
with the improvement that $\widehat\Op_h(a)\widehat\Op_h(b)=\widehat\Op_h(a\#b)$ 
without the
$\mathcal O(h^\infty)$ remainder; in fact, $\widehat S_{1/2}$ lies in
the class $\widetilde S_{1/2}(T^*\mathbb R^{n-1})$
from~\cite[(3.5)]{w-z}.

The proof of Lemma~\ref{l:harmonic} is based on the following precise
estimates on the full symbol of a function of the harmonic oscillator:
%
%%%%%%%%%%%%%%%%%%%%%%%%%%%%%% BEGIN PROP %%%%%%%%%%%%%%%%%%%%%%%%%%%%%%
%
\begin{lemm}\label{l:harmonic-internal}
Put
$$
P_0(h)=-h^2\Delta_{x'}^2+|x'|^2,
$$
an unbounded operator on $L^2(\mathbb R^{n-1})$, and let $\chi\in
C_0^\infty(\mathbb R)$. Then:

1. $\chi(P_0(h))=\widehat\Op_h(a_\chi)$, where $a_\chi(x',\xi';h)\in \widehat S$
and
\begin{equation}\label{e:chi-p0-symb}
\begin{gathered}
a_\chi(x',\xi';h)=\chi(|x'|^2+|\xi'|^2)+\mathcal O(h)_{\widehat S},\\
a_\chi=\mathcal O(h^\infty)_{\widehat S}\text{ outside of any neighborhood of }
\{(x',\xi')\mid |x'|^2+|\xi'|^2\in\supp\chi\}.
\end{gathered}
\end{equation}

2. $\chi((\tilde h/h)P_0(h))=\widehat\Op_h(a_{\chi((\tilde h/h)\cdot)})$,
where $a_{\chi((\tilde h/h)\cdot)}(x',\xi';h)\in\widehat S_{1/2}$ and
$$
a_{\chi((\tilde h/h)\cdot)}=\chi((\tilde h/h)(|x'|^2+|\xi'|^2))+\mathcal O(\tilde h)_{\widehat S_{1/2}}.
$$
If $T>0$ satisfies $\supp\chi\subset (-T,T)$
and $r=\sqrt{|x'|^2+|\xi'|^2}$,
then for each $\alpha,N$,
$$
|\partial^\alpha_{x',\xi'} a_{\chi((\tilde h/h)\cdot)}(x',\xi')|
\leq C_{\alpha N} (h/\tilde h)^{-|\alpha|/2}(h/r^2)^N
\text{ if }r\geq (Th/\tilde h)^{1/2}.
$$
\end{lemm}
%
%%%%%%%%%%%%%%%%%%%%%%%%%%%%%%% END PROP %%%%%%%%%%%%%%%%%%%%%%%%%%%%%%%
%
The proof of Lemma~\ref{l:harmonic-internal} is given at the end of
this subsection. Using part~2 of it  together with conjugating by an exponential
and performing a shift to reduce to the case $(x'_l,\xi'_l)=0$, we see
that $\widetilde\Psi_l=\widehat\Op_h(\check s_l)$, where
$$
\check s_l(x',\xi';h,\tilde h)=a_{\tilde\chi((\tilde h/h)\cdot)}(x'-x'_l,\xi'-\xi'_l;h)
$$
lies in $\widehat S_{1/2}$ uniformly in $l$.  Moreover, putting $T=2$ in
part~2 of Lemma~\ref{l:harmonic-internal} and recalling that
$\supp\tilde\chi\subset (-2,2)$, we get
\begin{gather}
\label{e:harmonic-1}
\check s_l(x',\xi')=\tilde\chi((\tilde h/h)(|x'-x'_l|^2+|\xi'-\xi'_l|^2))+\mathcal O(\tilde h)_{\widehat S_{1/2}},\\
\label{e:harmonic-2}
\check s_l=\mathcal O\big((h/(|x'-x'_l|^2+|\xi'-\xi'_l|^2))^{\infty}\big)_{\widehat S_{1/2}}
\text{ outside of }B((x'_l,\xi'_l),2(h/\tilde h)^{1/2}).
\end{gather}
%
%%%%%%%%%%%%%%%%%%%%%%%%%%%%%% BEGIN PROOF %%%%%%%%%%%%%%%%%%%%%%%%%%%%%%
%
\begin{proof}[Proof of Lemma~\ref{l:harmonic}]
The idea is to use the improved bound on the full symbol~\eqref{e:harmonic-2}
together with information on how many of the points $(x'_l,\xi'_l)$ can lie close to some fixed point.
Fix $(x'_0,\xi'_0)\in T^*\mathbb R^{n-1}$ and consider the dyadic partition
$$
\begin{gathered}
\{1,\dots,M(h,\tilde h)\}=\bigsqcup_{j\geq 0} L_j,\\
L_0=\{l\mid (x'_l,\xi'_l)\in B((x'_0,\xi'_0),2(h/\tilde h)^{1/2})\},\\
L_j=\{l\mid (x'_l,\xi'_l)\in B((x'_0,\xi'_0),2^{j+1}(h/\tilde h)^{1/2})
\setminus B((x'_0,\xi'_0),2^j (h/\tilde h)^{1/2})\},\ j\geq 1.
\end{gathered}
$$
By the triangle inequality,
$$
\bigcup_{l\in L_j} B\Big((x'_l,\xi'_l),{1\over 2}(h/\tilde h)^{1/2}\Big)
\subset B((x'_0,\xi'_0),2^{j+2}(h/\tilde h)^{1/2}).
$$
Moreover, since the pairwise distance between the points
$(x'_l,\xi'_l)$ is $>(h/\tilde h)^{1/2}$ by the construction in
\S\ref{s:approximation2}, the union on the left-hand side is disjoint.  By comparing the
volumes of the two sets, we arrive at the bound
\begin{equation}\label{e:l-j-bound}
|L_j|\leq C\cdot 2^{j(2n-2)}.
\end{equation}
Now, we write $\widetilde\Psi_\Pi=\chi_b\widehat\Op_h(\check s_\Pi)\chi_b$,
where
$$
\check s_\Pi=\sum_{j\geq 0}\check s_\Pi^{(j)},\
\check s_\Pi^{(j)}=\sum_{l\in L_j}\check s_l.
$$
The symbol $\check s_\Pi^{(0)}$ will be the principal part of $\check
s_\Pi$ at $(x_0',\xi_0')$.  The symbols $\check s_\Pi^{(j)}$ for $j\geq 1$ at
$(x'_0,\xi'_0)$ are nonzero because of tunneling; however, we can
estimate their contribution as $\mathcal O(\tilde h^\infty)$. More
precisely, we combine~\eqref{e:harmonic-2} with the
bound~\eqref{e:l-j-bound} on $|L_j|$ to get
$$
\check s^{(j)}=\mathcal O((2^{-j}\tilde h)^\infty)_{\widehat S_{1/2}}\text{ at }(x'_0,\xi'_0)\text{ for }j\geq 1.
$$
Hence $\sum_{j=1}^\infty \check s^{(j)}$ is $\mathcal
O(\tilde h^\infty)_{\widehat S_{1/2}}$ at $(x'_0,\xi'_0)$. Now,
by~\eqref{e:l-j-bound}, $\check s_\Pi^{(0)}$ is the sum of a bounded
number of $\check s_l$'s and is therefore in $\widehat S_{1/2}$ at
$(x'_0,\xi'_0)$. Moreover, if $(x'_0,\xi'_0)\in\Pi(\supp\tilde a)$,
then by~\eqref{e:covering} and~\eqref{e:harmonic-1}, at least one term
in the sum for $\check s^{(0)}(x'_0,\xi'_0)$ is equal to
$1+\mathcal O(\tilde h)$, and the other terms are nonnegative modulo
$\mathcal O(\tilde h)$.  Therefore, $\check s_\Pi\in \widehat S_{1/2}$
and $|\check s_\Pi|\geq 1-\mathcal O(\tilde h)$ near $\Pi(\supp\tilde
a)$. Also, $\check s_\Pi=\mathcal O(h^\infty)_{\widehat S_{1/2}}$
outside a fixed compact set.  Then
$\widetilde\Psi_\Pi=\chi_b\widehat\Op_h(\check s_\Pi)\chi_b$ lies in
$\Psie(\mathbb R^{n-1})$ and
$\tilde\sigma(\widetilde\Psi_\Pi)=\chi_b(x')^2\check s_\Pi+\mathcal
O(h^{1/2}\tilde h^{1/2})_{\Syme(\mathbb R^{n-1})}$; this finishes the
proof.
\end{proof}
%
%%%%%%%%%%%%%%%%%%%%%%%%%%%%%%% END PROOF %%%%%%%%%%%%%%%%%%%%%%%%%%%%%%%
%
%
%%%%%%%%%%%%%%%%%%%%%%%%%%%%%% BEGIN PROOF %%%%%%%%%%%%%%%%%%%%%%%%%%%%%%
%
\begin{proof}[Proof of Lemma~\ref{l:harmonic-internal}]
1. This follows from standard results on functional calculus of
pseudodifferential operators, see for example~\cite[Chapter~8]{d-s}.
In particular, the fact that $a\in\widehat S$ follows
from~\cite[Theorem~8.7]{d-s}, with the order function
$m=1+|x'|^2+|\xi'|^2$, while~\eqref{e:chi-p0-symb} follows from the
expansion for $a_\chi$ preceding~\cite[(8.15)]{d-s}.  The validity
of~\eqref{e:chi-p0-symb} in $\widehat S$ is checked as in the proof
of~\cite[Theorem~8.7]{d-s}, by using the composition formula and the
fact that $\chi(P_0(h))=\chi_k(P_0(h))(P_0(h)+i)^{-k}$, where
$\chi_k(\lambda)=(\lambda+i)^k\chi(\lambda)$ lies in
$C_0^\infty(\mathbb R)$ and $(P_0(h)+i)^{-k}$ has symbol in
$S(m^{-k})$, in the notation of~\cite{d-s}.

2. We use the unitary rescaling operator (see~\cite[\S6.1.2]{e-z})
$$
T_\beta:L^2(\mathbb R^{n-1})\to L^2(\mathbb R^{n-1}),\
\beta>0,\
(T_\beta u)(\tilde x)=\beta^{(n-1)/4}u(\beta^{1/2} \tilde x);
$$
then 
$$
\chi((h/\tilde h)P_0(h))=T_\beta^{-1}\chi((\tilde h\beta/ h)P_0(h/\beta))T_\beta.
$$
Moreover, the operator $T_\beta$ changes Weyl quantized symbols as follows:
$$
\widehat\Op_h(a)=T_\beta^{-1}\widehat\Op_{h/\beta}(a_\beta)T_\beta,\
a_\beta(\tilde x,\tilde \xi)=a(\beta^{1/2}\tilde x,\beta^{1/2}\tilde\xi).
$$
Take $\beta=h/\tilde h$; then
$$
\begin{gathered}
\chi((\tilde h/h)P_0(h))=T_\beta^{-1}\chi(P_0(\tilde h))T_\beta
=T_\beta^{-1}\widehat\Op_{\tilde h}(a_\chi(\cdot,\cdot;\tilde h))T_\beta
=\widehat\Op_h(a_{\chi((\tilde h/h)\cdot)}),\\
a_{\chi((\tilde h/h)\cdot)}(x',\xi';h)=a_\chi((\tilde h/h)^{1/2}x',
(\tilde h/h)^{1/2}\xi';\tilde h).
\end{gathered}
$$
It remains to use the estimates on $a_\chi$ from part~1, with $\tilde h$
taking the place of $h$.
\end{proof}
%
%%%%%%%%%%%%%%%%%%%%%%%%%%%%%%% END PROOF %%%%%%%%%%%%%%%%%%%%%%%%%%%%%%%
%

%%%%%%%%%%%%%%%%%%%%%%%%%%%%%%%%%%%%%%%%%%%%%%%%%%%%%%%%%%%%%%%%%%%%%%%%%%%%%%%%
\subsection{Approximation}
  \label{s:approximation4}

We finally use parametrices to obtain the approximation.

We write $\widetilde A=\widetilde A'+\widetilde A''$, where
$$
\widetilde A'=\Op_h(\tilde a'),\
\tilde a'(x_1,x',\xi_1,\xi')=\widetilde\chi(\xi_1)\widetilde a(x_1,x',0,\xi'),
$$
with $\widetilde \chi \in C_0^\infty(\mathbb R)$ as in \S\ref{s:approximation3}.
Consider the operator $\widetilde\Psi_\Pi$ from Lemma~\ref{l:harmonic},
take $\widetilde\chi'\in C_0^\infty(\mathbb R)$ equal to $1$ near $\supp\widetilde\chi$,
and put
$$
\Psi_\Pi=\widetilde\chi'(hD_{x_1})\otimes \widetilde\Psi_\Pi\in\Psie(X_M).
$$
The symbol $s_\Pi=\tilde\sigma(\Psi_\Pi)$ satisfies
$|s_\Pi|\geq 1-\mathcal O(\tilde h)$ on $\supp(\tilde a')$.
(Here we use that $|\tilde s_\Pi|\geq 1-\mathcal O(\tilde h)$ on $\Pi(\supp\tilde a)$ from Lemma~\ref{l:harmonic}, the definition~\eqref{e:def-pi} of~$\Pi$, and $\widetilde \chi' \widetilde \chi = \widetilde \chi$.)
Put $B'=\Op_h(\tilde a'/s_\Pi)\in\Psi^{\comp}_{1/2}(X_M)$; then
$$
\widetilde A'=\Psi_\Pi B'+\mathcal O(\tilde h)_{\Psie(X_M)},
$$
and we have written $\widetilde A'$ as the sum of a finite rank term and an $\mathcal O(\tilde h)$ remainder, as needed in the statement of Lemma~\ref{l:approximation}. To treat the $\widetilde A''$ term,
put $B''=\Op_h((\tilde a-\tilde
a')/\xi_1)\in(\tilde h/h)^{1/2}\Psie(X_M)$ (with the prefactor coming
from the fact that $\partial_{\xi_1}\tilde a$ can grow like $(\tilde
h/h)^{1/2}$). Then by part~6 of Lemma~\ref{l:1/2-properties},
$$
\widetilde A''=(hD_{x_1})B''+\mathcal O(\tilde h)_{\Psie(X_M)}.
$$
Therefore (with $\mathcal O(\tilde h)$ below in the sense of~\eqref{e:b-bound})
$$
A=\chi(\tilde hD_{x_1})\widetilde A=
\chi(\tilde hD_{x_1})\Psi_\Pi B'
+\chi(\tilde hD_{x_1})(hD_{x_1})B''+\mathcal O(\tilde h).
$$
However, $\chi(\tilde hD_{x_1})(hD_{x_1})= \chi(\tilde hD_{x_1})(\tilde hD_{x_1}) (h/\tilde h)= \mathcal O(h/\tilde h)$; therefore,
$$
A=\chi(\tilde hD_{x_1})\Psi_\Pi B'+\mathcal O(\tilde h).
$$
We now put $A_R=\chi(\tilde hD_{x_1})\Psi_\Pi B'$.  To estimate its
rank, we write for $\tilde h$ small enough, $\chi(\tilde
hD_{x_1})\Psi_\Pi=\chi(\tilde hD_{x_1})\otimes\widetilde\Psi_\Pi$,
where $\chi(\tilde hD_{x_1})$ on the right-hand side acts on
$L^2(\mathbb S^1)$ and has rank $\mathcal O(\tilde h^{-1})$;
therefore, by~\eqref{e:harmonic-rank},
$$
\rank A_R\leq C\tilde h^{\nu-n}h^{-\nu}
$$
and the proof of Lemma~\ref{l:approximation} is finished.

%%%%%%%%%%%%%%%%%%%%%%%%%%%%%%%%%%%%%%%%%%%%%%%%%%%%%%%%%%%%%%%%%%%%%%%%%%%%%%%%
%                                  SECTION 4.4                                 %
%%%%%%%%%%%%%%%%%%%%%%%%%%%%%%%%%%%%%%%%%%%%%%%%%%%%%%%%%%%%%%%%%%%%%%%%%%%%%%%%
\section{Proof of the main lemma}
  \label{s:ultimate}

Fix $\delta>0$ small enough so that all the results of \S\ref{s:ah} apply.
We will now impose an additional condition on $Q$: namely, that there
exists $Q_0\in \Psi^1(X)$ such that $T_sQT_s^{-1}=\pm Q_0^*Q_0$
microlocally near $\Sigma_\pm$. Such  a $Q$ can be obtained by first
choosing $Q_0$ and $T_s$.  Note that $Q_0$ will be elliptic on
$\{\lxir^{-2}p=0\}\cap \{\mu\leq-\delta\}$.

\subsection{Positive commutator estimates}\label{s:ultimate0}
In \S\ref{s:ultimate0} and \S\ref{s:ultimate1} we assume $|\Real z| \le C_0 h$, $|\Imag z| \le C_0 h$. We relax this assumption in \S\ref{s:ultimate2} where we prove an improved estimate for $\Imag z > 0$.

We start with the construction of an escape function near the trapped
set. We recall~\cite[Lemma~7.6]{sj-z} (where one puts $\epsilon=h/\tilde h$):
%
%%%%%%%%%%%%%%%%%%%%%%%%%%%%%% BEGIN LEMMA %%%%%%%%%%%%%%%%%%%%%%%%%%%%%%
%
\begin{lemm}\label{l:f-hat}
Suppose the geodesic flow on $M$ is hyperbolic on $K$ in the sense
of the assumption of Theorem~\ref{l:theorem-strong}. Then
 there exists a neighborhood $V_K$
of $\iota(K)$ and a family of smooth real-valued functions
$\hat f(x,\xi;h,\tilde h)$ on $V_K$ depending
on two parameters $0<h<\tilde h$, such that
for some constant $C_{\hat f}>1$,
$$
\begin{gathered}
\hat f=\mathcal O(\log(1/h)),\
\partial^\alpha H_p^k\hat f=\mathcal O((\tilde h/h)^{-|\alpha|/2}),\
|\alpha|+k\geq 1;\\
H_p\hat f(x,\xi)\geq C_{\hat f}^{-1}>0\text{ for }
d((x,\xi),\iota(\widetilde K))>C_{\hat f}(h/\tilde h)^{1/2}.
\end{gathered}
$$ 
\end{lemm}
Note that we have written $\mathcal O(\log(1/h))$ where \cite{sj-z} has $\mathcal O(\log(\tilde h/h))$. This is equivalent because when $h<\tilde h^2<1$, $\log(1/h)/2\leq \log(\tilde h/h)\leq \log(1/h)$. For later convenience, we assume that $V_K$ is small enough that
\begin{equation}\label{e:vkdelta}
\overline V_K \subset \{\mu > \sqrt {5\delta}\},
\end{equation}
and that the estimate \eqref{e:im-part} holds on $\overline V_K$.
%
%%%%%%%%%%%%%%%%%%%%%%%%%%%%%% END LEMMA %%%%%%%%%%%%%%%%%%%%%%%%%%%%%%
%

Take a neighborhood $U_K$ of $K$ such that $\overline U_K\subset V_K$, and which is
sufficiently small that Lemma~\ref{l:escape} applies, and
let $f_0$ be the function defined by Lemma~\ref{l:escape}. Then define
\[
F_0 \in\Psi^{\comp}(X),\
F_0^*=F_0,\
\sigma(F_0)=f_0+\mathcal O(h).
\]
Here $F_0$ is a standard compactly microlocalized semiclassical pseudodifferential operator in the sense of \S\ref{s:prelim.basics}. Next take $\hat\chi\in C_0^\infty(V_K)$ such that
$\hat\chi=1$ near $\overline U_K$, and define
\[
\widehat F  \in\log(1/h)\Psie(X),\
\WFh(\widehat F)\subset V_K,\
\widehat F^*=\widehat F,\
\tilde\sigma(\widehat F)=\hat\chi\hat f+\mathcal O(h^{1/2}\tilde h^{1/2}).
\]
Here $\widehat F$ is a pseudodifferential operator in the exotic class described in Lemma~\ref{l:kangaroo}.
Let $M$ be a large constant and define the full quantized escape function
\[
F=\widehat F+M\log(1/h)F_0.
\]
Then $F=\mathcal O(\log(1/h))_{\Psie}$
and its principal symbol is
$$
f=\hat\chi\hat f+M\log(1/h) f_0.
$$
Note that
$$
f=\mathcal O(\log(1/h)),\ \partial^\alpha f=\mathcal O((h/\tilde h)^{-|\alpha|/2}),\ |\alpha|>0.
$$
We then calculate 
\begin{equation}\label{e:h-p-f}
H_p f=\hat\chi\cdot H_p\hat f+\hat f\cdot H_p\hat\chi+M\log(1/h)H_pf_0;
\end{equation}
we see by Lemma~\ref{l:f-hat}
that $H_p f=\mathcal O(\log(1/h))_{\Syme}$ and therefore
(see Lemma~\ref{l:1/2-properties}(7) and~Lemma~\ref{l:kangaroo}(3)) we gain a full power of $h$ when commuting $F$ with $P(0)$:
\begin{equation}\label{e:conj-f}
[P(0),F]=\mathcal O(h\log(1/h))_{\Psie},\
\tilde\sigma([P(0),F])=-ihH_pf+\mathcal O(h\tilde h)_{\Syme}.
\end{equation}
Since $\pm H_pf_0>0$ near the set $V_\pm$ defined in Lemma~\ref{l:escape}(3),
we can fix $M$ large enough so that
for some constant $C_f>0$,
\begin{equation}
  \label{e:f-positivity}
\pm H_p f> C_f^{-1}\log(1/h)\text{ near }V_\pm.
\end{equation}
We now perform the conjugation.
Since $F$ is compactly microlocalized and $\|F\|=\mathcal O(\log(1/h))$, we see
that for any fixed $t$, the operator $e^{tF}-1$ is compactly microlocalized
and polynomially bounded in $h$. Given the operator $T_s$
from Lemma~\ref{l:radial-conj} (and in particular
$s$ is large enough depending on $\Imag z$), define the conjugated operator
$$
P_t(z)=e^{-tF}T_s(P(z)-iQ)T_s^{-1}e^{tF}.
$$
Here $t>0$ will be chosen in Lemma~\ref{l:estimate-k} so that the contribution from the conjugation by $e^{tF}$ is
greater than that of $\Imag T_sP(z)T_s^{-1}$; how large $t$ needs to be depends
on $\Imag z$.
%
%%%%%%%%%%%%%%%%%%%%%%%%%%%%%% BEGIN PROP %%%%%%%%%%%%%%%%%%%%%%%%%%%%%%
%
\begin{lemm}\label{l:conjugation} We have
\begin{equation}\label{e:conj-dec}
P_t(z)=T_s(P(z)-iQ)T_s^{-1}+t[T_s(P(z)-iQ)T_s^{-1},F]+\mathcal O_t(h\tilde h)_{\Psie}.
\end{equation}
\end{lemm}
%%%%%%%%%%%%%%%%%%%%%%%%%%%%%%%%%%%%%%%%%%%%%%%%%%%%%%%%%%%%%%%%%%%%%%%%%%%%%%%%
\begin{proof}
We follow the proof of~\cite[Proposition~8.2]{sj-z}. We write
\begin{equation}\label{e:ptexpansion}
\begin{gathered}
P_t(z)=e^{-t\ad_F} (T_s(P(z)-iQ)T_s^{-1})
=T_s(P(z)-iQ)T_s^{-1}-t\ad_F(T_s(P(z)-iQ)T_s^{-1})\\
+\int_0^t (t-t')e^{-t'F}\ad^2_F(T_s(P(z)-iQ)T_s^{-1})e^{t'F}\,dt'.
\end{gathered}
\end{equation}
By the Bony--Chemin Theorem (\cite{bonychemin}; see~\cite[Lemma~8.1]{sj-z} for this version), the
conjugation $A\mapsto e^{-t'F}Ae^{t'F}$ is  continuous
$\Psie \to \Psie$. Hence, it remains to show
that
$$
\ad^2_F(T_s(P(z)-iQ)T_s^{-1})=\mathcal O(h\tilde h)_{\Psie}.
$$
Using Lemma~\ref{l:vasy-1}(6) and the fact that $|z| \le 2C_0h$ (see the remark at the beginning of the subsection), $\WFh(\widehat F)\cap\WFh(Q)=\emptyset$,
and Lemma~\ref{l:kangaroo}(3),
\begin{equation}\label{e:conj-comm}
[T_s(P(z)-iQ)T_s^{-1},F]=[P(0),F]-iM\log(1/h)[Q,F_0]+\mathcal O(h^{3/2}\tilde h^{1/2})_{\Psie}.
\end{equation}
The second term on the right-hand side is $\mathcal O(h\log(1/h))_{\Psi^{\comp}}$;
commuting it with $F$, we get $\mathcal O(h^{3/2}\tilde h^{1/2}\log(1/h))$. The third term is handled similarly.
As for the first term, by~\eqref{e:conj-f} the symbol of $h^{-1}[P(0),F]$
satisfies~\eqref{e:kangaroo}; therefore, by Lemma~\ref{l:kangaroo}(2),
$[[P(0),F],F]\in h\tilde h\Psie$.
\end{proof}
%
%%%%%%%%%%%%%%%%%%%%%%%%%%%%%%% END PROP %%%%%%%%%%%%%%%%%%%%%%%%%%%%%%%
%
Combining Lemma~\ref{l:vasy-1}(6), $|z| \le 2C_0h$, \eqref{e:conj-f}, \eqref{e:conj-dec}, and~\eqref{e:conj-comm}, we get
\begin{equation}\label{e:conj-dec-approx}
P_t(z)=P(0)-iQ+\mathcal O_t(h)_{\Psi^1}+\mathcal O_t(h\log(1/h))_{\Psie}.
\end{equation}
Since $q$ and $f_0$ are real-valued, $\Real[Q,F_0]=\mathcal O(h^2)_{\Psi^{\comp}}$; therefore, using \eqref{e:conj-dec} and \eqref{e:conj-comm},
\begin{equation}\label{e:conj-dec-im}
\Imag P_t(z)=\Imag(T_sP(z)T_s^{-1})-\Real(T_sQT_s^{-1})+t\Imag [P(0),F]+\mathcal O_t(h\tilde h)_{\Psie}.
\end{equation}
We will use a positive commutator argument for $\Imag \widetilde P_t(z)$,
with $\widetilde P_t(z)=P_t(z)-ithA$ as in~\eqref{e:p-t-tilde-intro},
to control the norm of $u$ in terms of $\widetilde P_t(z) u$.
We first analyze the terms involving $P_t(z)$, then  define $A$
and analyze the $-ithA$ term, and
finally put them all together in \S\S\ref{s:ultimate1}--\ref{s:ultimate2}.
The right-hand side of~\eqref{e:conj-dec-im}
consists of several components in different symbol classes, with positivity
of the sum provided by different components in different regions of $\overline T^*X$.
In Lemma~\ref{l:microlocal-partition} we construct a  microlocal partition of unity corresponding
to these regions, and treat each member of the partition in a separate
lemma. The following lemma is useful in dealing with this partition;
the left-hand side of~\eqref{e:localization} is easily summed over different
operators $\Psi_1$, and the right-hand side is adapted to a positive commutator
argument.
%
%%%%%%%%%%%%%%%%%%%%%%%%%%%%%% BEGIN LEMMA %%%%%%%%%%%%%%%%%%%%%%%%%%%%%%
%
\begin{lemm}\label{l:localization}
Let $\Psi_1\in\Psi^0$ have real-valued principal
symbol and assume that $T_sQT_s^{-1}=\pm Q_0^*Q_0$ near
$\WFh(\Psi_1)$. Then for $u \in C^\infty(X)$,
\begin{equation}\label{e:localization}
\begin{gathered}
\pm\Real\langle\Imag P_t(z)u,\Psi_1^2 u\rangle\leq\pm\langle(\Imag(T_sP(z)T_s^{-1}))\Psi_1u,\Psi_1u\rangle\\
\mp t\Real\langle i[P(0),F]\Psi_1 u,\Psi_1 u\rangle+\mathcal O_t(h\tilde h)\|u\|_{L^2}^2.
\end{gathered}
\end{equation}
\end{lemm}
%%%%%%%%%%%%%%%%%%%%%%%%%%%%%%%%%%%%%%%%%%%%%%%%%%%%%%%%%%%%%%%%%%%%%%%%%%%%%%%%
\begin{proof}
We first claim that the left-hand side can be replaced
by $\pm\langle(\Imag P_t(z))\Psi_1u,\Psi_1u\rangle$:
\begin{equation}\label{e:localization-internal}
\Real\langle\Imag P_t(z)u,\Psi_1^2u\rangle-\langle\Imag P_t(z)\Psi_1 u,\Psi_1u\rangle=\mathcal O_t(h^{3/2}\tilde h^{1/2}\log(1/h))\|u\|_{L^2}.
\end{equation}
Indeed, write the left-hand side as $\Real(Bu,u)$, where
$$
B=\Psi_1^*((\Psi_1^*-\Psi_1)\Imag P_t(z)+[\Psi_1,\Imag P_t(z)]).
$$
We now use~\eqref{e:conj-dec-approx}. The part of $B$ corresponding
to $P(0)-iQ+\mathcal O_t(h)_{\Psi^1}$ lies in $h\Psi^1$ for each $t$, and, when multiplied by $h^{-1}$,
has imaginary-valued
principal symbol; therefore, the corresponding part of $\Real(Bu,u)$
is $\mathcal O_t(h^2)\|u\|_{L^2}^2$. The part of $B$ corresponding to
$\mathcal O_t(h\log(1/h))_{\Psie}$ is
$\mathcal O_t(h^{3/2}\tilde h^{1/2}\log(1/h))_{\Psie}$ by Lemma~\ref{l:1/2-properties}(6).

Having established~\eqref{e:localization-internal}, we use~\eqref{e:conj-dec-im}:
$$
\begin{gathered}
\pm\langle\Imag P_t(z)\Psi_1u,\Psi_1u\rangle=\pm\langle\Imag(T_sP(z)T_s^{-1})\Psi_1u,\Psi_1u\rangle\\
-\|Q_0\Psi_1u\|_{L^2}^2\pm t\Imag\langle[P(0),F]\Psi_1u,\Psi_1u\rangle+\mathcal O_t(h\tilde h)\|u\|_{L^2}^2;
\end{gathered}
$$
it remains to note that the second term on the right-hand side is $\leq 0$.
\end{proof}
%
%%%%%%%%%%%%%%%%%%%%%%%%%%%%%% END LEMMA %%%%%%%%%%%%%%%%%%%%%%%%%%%%%%
%
We now introduce the microlocal partition of unity mentioned before
Lemma~\ref{l:localization}. The operator $\Psi_E$ corresponds to the
elliptic set of $P(z)-iQ$, $\Psi_{L_\pm}$ correspond
to the neighborhoods of the radial sets $L_\pm$ where $\Imag(T_s P(z)T_s^{-1})$
has a favorable sign by Lemma~\ref{l:radial-conj},
$\Psi_{0\pm}$ correspond to the transition region,
where the escape function $f_0$ from Lemma~\ref{l:escape}
provides positivity, and $\Psi_K$ handles a neighborhood of
the trapped set.
%
%%%%%%%%%%%%%%%%%%%%%%%%%%%%%% BEGIN LEMMA %%%%%%%%%%%%%%%%%%%%%%%%%%%%%%
%
\begin{lemm}\label{l:microlocal-partition}
There exist operators $\Psi_E,\Psi_{L_\pm}\in\Psi^0(X)$ and $\Psi_{0\pm},\Psi_K\in\Psic(X)$
with real-valued principal symbols, and neighborhoods $\widetilde\Sigma_\pm$ of $\Sigma_\pm\cap\{\mu\geq-\delta\}$, such that
%%%%%%%%%%%%%%%%%%%%%%%%%%%%%%
\begin{enumerate}
  \item $\WFh(\Psi_E)$ is contained in the elliptic set of $p-iq$;
  \item  $\Psi_{L_+}^2+\Psi_{0+}^2+\Psi_K^2=1$ and $\Psi_{L_-}=\Psi_{0-}=0$ microlocally
on $\widetilde\Sigma_+$ and $\Psi_{L_-}^2+\Psi_{0-}^2=1$ and
$\Psi_{L_+}=\Psi_{0+}=\Psi_K=0$ microlocally on $\widetilde\Sigma_-$;
  \item $\Psi_E$ is elliptic on the complement of $\widetilde\Sigma_+\cup\widetilde\Sigma_-$;
  \item $T_sQT_s^{-1}=\pm Q_0^*Q_0$ microlocally near $\WFh(\Psi_{L_\pm}) \cup \WFh(\Psi_{0\pm})$, and $\WFh (T_sQT_s^{-1}) \cap \WFh(\Psi_K) = \emptyset$;
  \item $\pm\lxir^{-1}\sigma(h^{-1}\Imag(T_sP(z)T_s^{-1}))<0$ and $\pm H_pf_0\geq 0$ near $\WFh(\Psi_{L_\pm})$;
moreover, $\WFh(\Psi_{L_\pm})\cap \overline U_K = \WFh(\Psi_{L_\pm})\cap\WFh(\widehat F)=\emptyset$;
  \item $\pm H_pf\geq C_f^{-1}\log(1/h)>0$ on $\WFh(\Psi_{0\pm})$;
  \item $\WFh(\Psi_{L_\pm})$ and $\WFh(\Psi_{0\pm})$ do not intersect $\iota(\widetilde K)$;
  \item $\WFh(\Psi_K)\subset U_K$ and $H_pf_0\geq 0$ near $\WFh(\Psi_K)$;
  \item $\mp \Real(\lxir^{-1} \sigma(\partial_zP(0)))>0$
on $\widetilde\Sigma_\pm$.
\end{enumerate}
%%%%%%%%%%%%%%%%%%%%%%%%%%%%%%
\end{lemm}
%%%%%%%%%%%%%%%%%%%%%%%%%%%%%%%%%%%%%%%%%%%%%%%%%%%%%%%%%%%%%%%%%%%%%%%%%%%%%%%%
\begin{proof}
We will define open coverings
\begin{equation}\label{e:ucover}
\begin{gathered}
\Sigma_+\cap\{\mu\geq -\delta\}\subset U_{L_+}\cup U_{0+}\cup U_K,\
\Sigma_-\cap\{\mu\geq -\delta\}\subset U_{L_-}\cup U_{0-},
\end{gathered}
\end{equation}
take partitions of unity subordinate to these open coverings such that
\begin{equation}
  \label{e:nbhd-1}
\psi_{L_+}^2+\psi_{0+}^2+\psi_K^2=1\text{ near }\Sigma_+\cap\{\mu\geq -\delta\},\
\psi_{L_-}^2+\psi_{0-}^2=1\text{ near }\Sigma_-\cap\{\mu\geq-\delta\},
\end{equation}
satisfying the additional support conditions (possible since $\Sigma_+\cap\Sigma_-=\emptyset$)
\begin{equation}
  \label{e:nbhd-2}
\begin{gathered}
(\supp\psi_{L_+}\cup\supp\psi_{0+}\cup\supp\psi_K)\cap(\Sigma_-\cap\{\mu\geq -\delta\})=\emptyset,\\
(\supp\psi_{L_-}\cup\supp\psi_{0-})\cap(\Sigma_+\cap\{\mu\geq -\delta\})=\emptyset,
\end{gathered}
\end{equation}
and take open $\widetilde\Sigma_\pm\supset\Sigma_\pm\cap\{\mu\geq -\delta\}$ such that~\eqref{e:nbhd-1} and~\eqref{e:nbhd-2}
hold with $\Sigma_\pm\cap\{\mu\geq -\delta\}$ replaced by $\widetilde\Sigma_\pm$. 
The $\Psi_j$ will be obtained at the end of this proof
by quantizing the $\psi_j$ and adding correction terms (without changing the semiclassical wavefront sets) to obtain the equations in item~(2) without remainders. Since $p-iq$ is elliptic
on the complement of $\widetilde\Sigma_+\cup\widetilde\Sigma_-$ (as follows from the properties of $q$
listed at the end of Section~\ref{s:ah.vasy}), we can find $\Psi_E\in\Psi^0(X)$ such that items~(1) and~(3)
hold.

The open set $U_K$ is the same as the one defined immediately following Lemma~\ref{l:f-hat}. Since $\WFh(\Psi_K) \subset U_K$, the properties of $\WFh (\Psi_K)$ asserted in items (4) and (8) follow from \eqref{e:vkdelta}, which implies $\overline U_K \subset \{\mu >0\} \subset \overline T^*X \setminus \WFh (T_s Q T_s^{-1})$, and from Lemma~\ref{l:escape}(2).

 Let $U_{L\pm}$ be open sets such that, with $\rho_1$ defined in~\eqref{e:rho-1}, 
\begin{equation}\label{e:ul}
\Sigma_\pm\cap\{\mu\geq -\delta\}\cap\{\rho_1\le5\delta\}\subset U_{L\pm},
\end{equation}
and such that
$\overline U_{L\pm}$ is disjoint from 
\[
\begin{gathered}
\WFh(T_sQT_s^{-1}\mp Q_0^* Q_0) \cup \{\pm\lxir^{-1}\sigma(h^{-1}\Imag(T_sP(z)T_s^{-1}))\ge0\} \,\cup \overline U_K \\\cup \WFh(\widehat F)  \cup \overline{\{\pm H_pf_0 < 0\}}  \cup \iota(\widetilde K)  \cup \{\mp \Real(\lxir^{-1} \sigma(\partial_zP(0)))\le0\}.
\end{gathered}
\]
To see that such sets exist, note that $\Sigma_\pm\cap\{\mu\geq-\delta\}\cap\{\rho_1\le5\delta\}$ is disjoint from $\WFh(T_sQT_s^{-1}\mp Q_0^*Q_0)$ by the condition imposed on $Q_0$ when it was introduced at the beginning of \S\ref{s:ultimate},  from $\{\pm\lxir^{-1}\sigma(h^{-1}\Imag(T_sP(z)T_s^{-1}))\ge0\}$ by Lemma~\ref{l:radial-conj}, from $\overline U_K \cup \WFh(\widehat F)$ by  \eqref{e:vkdelta}, from $\overline{\{\pm H_pf_0 < 0\}}$ by Lemma~\ref{l:escape}(2), from $\iota(\widetilde K)$ by the fact that $L_\pm$ is a source/sink (see \eqref{e:rho1-1}), and from $\{\mp \Real(\lxir^{-1}\sigma(\partial_zP(0)))\le0\}$ by \eqref{e:im-part}. This disjointness condition implies the properties of $\WFh(\Psi_{L\pm})$ asserted in items (4), (5), and (7).

Let $U_{0\pm}$ be open sets such that
\begin{equation}\label{e:u0}
 V_\pm\subset U_{0\pm},
\end{equation}
(with notation as in Lemma~\ref{l:escape}(3)) and such that $\overline U_{0\pm}$ is disjoint from
$$
\begin{gathered}
\WFh(T_sQT_s^{-1} \mp Q_0^* Q_0) \cup \{\pm H_p f\leq C_f^{-1}\log(1/h)\}\\\cup \iota(\widetilde K) \cup S^*X  \cup \{\mp \Real(\lxir^{-1}\sigma(\partial_zP(0)))\le0\}.
\end{gathered}
$$
That this is possible is checked as in the construction of $U_{L\pm}$ above, and by~\eqref{e:f-positivity}.
This disjointness condition implies the properties of $\WFh(\Psi_{0\pm})$ asserted in items (4), (6), and (7).
The condition of disjointess from $S^*X$  ensures that $\Psi_{0\pm}$ is compactly microlocalized.

The covering property~\eqref{e:ucover} follows from~\eqref{e:ul} and~\eqref{e:u0}. 
Now item (9) follows from 
\[
\widetilde\Sigma_+\subset U_{L_+}\cup U_{0+}\cup U_K,\
\widetilde\Sigma_-\subset U_{L_-}\cup U_{0-},
\]
together with the fact that the closures of the right hand sides of this formula are disjoint from $ \{\mp \Real(\lxir^{-1}\sigma(\partial_zP(0)))\le0\}$ by construction.

We now explain in  detail the construction of the $\Psi_j$, giving item (2). We have
\[
\Op_h(\psi_{L_+})^2 + \Op_h(\psi_{0+})^2 +\Op_h(\psi_K)^2 =1 + R_+
\]
microlocally on $\widetilde \Sigma_+$, where $R_+ \in h \Psi^{-1}$.
There exists an operator $S_+=1+\mathcal O(h)_{\Psi^{-1}}$ such that
$S_+^2(1+R_+)=1$ microlocally on $\widetilde\Sigma_+$. Then
\[
(S_+\Op_h(\psi_{L_+}))^2+(S_+\Op_h(\psi_{0+}))^2+(S_+\Op_h(\psi_K))^2=1+\mathcal O(h^2)_{\Psi^{-2}},
\]
microlocally on $\widetilde \Sigma_+$. Iterating the process of dividing the right hand side over, and concluding with a Borel summation, we improve the remainder to $\mathcal O(h^\infty)_{\Psi^{-\infty}}$, while
preserving the property $\WFh(\Psi_j)\subset\supp\psi_j$. The operators $\Psi_{L_-}$ and $\Psi_{0-}$
are constructed similarly.
\end{proof}
%
%%%%%%%%%%%%%%%%%%%%%%%%%%%%%% END LEMMA %%%%%%%%%%%%%%%%%%%%%%%%%%%%%%
%
We now prove portions of the estimate corresponding to each
of the pseudodifferential~operators of Lemma~\ref{l:microlocal-partition}.
We start with the radial points, where we use the conjugation~by~$T_s$:
%
%%%%%%%%%%%%%%%%%%%%%%%%%%%%%% BEGIN LEMMA %%%%%%%%%%%%%%%%%%%%%%%%%%%%%%
%
\begin{lemm}\label{l:estimate-l-pm}
For some constant $C_t$ and  $u\in C^\infty(X)$,
$$
\pm \Real\langle\Imag P_t(z)u,\Psi^2_{L_\pm}u\rangle\leq -C_t^{-1}h\|\Psi_{L_\pm}u\|_{\Hh^{1/2}}^2
+\mathcal O_t(h\tilde h)\|u\|_{L^2}^2.
$$
\end{lemm}
%%%%%%%%%%%%%%%%%%%%%%%%%%%%%%%%%%%%%%%%%%%%%%%%%%%%%%%%%%%%%%%%%%%%%%%%%%%%%%%%
\begin{proof}
Note that, by Lemma~\ref{l:microlocal-partition}(5),
$\WFh(\widehat F)\cap\WFh(\Psi_{L_\pm})=\emptyset$.
By Lemma~\ref{l:localization}, it is then enough to estimate
$$
\pm\langle(\Imag(T_sP(z)T_s^{-1}))\Psi_{L_\pm}u,\Psi_{L_\pm}u\rangle
\mp tM\log(1/h)\Real\langle i[P(0),F_0]\Psi_{L_\pm}u,\Psi_{L_\pm}u\rangle.
$$
We now use Lemma~\ref{l:microlocal-partition}(5) again. By the non-sharp G\r arding inequality \eqref{e:nonsharpg},
the first term is $\leq -C_t^{-1}h\|\Psi_{L_\pm}u\|_{\Hh^{1/2}}^2+\mathcal O(h^\infty)\|u\|_{L^2}^2$.
Also, the principal symbol of $\mp h^{-1}i[P(0),F_0]$ is equal to $\mp H_pf_0\leq 0$
near $\WFh(\Psi_{L_\pm})$; then
the second term is $\leq \mathcal O(h^2\log(1/h))\|u\|_{L^2}^2$ by the sharp G\r arding inequality \eqref{e:sharpg}.
\end{proof}
%
%%%%%%%%%%%%%%%%%%%%%%%%%%%%%% END LEMMA %%%%%%%%%%%%%%%%%%%%%%%%%%%%%%
%
Next, we deal with the transition region:
%
%%%%%%%%%%%%%%%%%%%%%%%%%%%%%% BEGIN LEMMA %%%%%%%%%%%%%%%%%%%%%%%%%%%%%%
%
\begin{lemm}\label{l:estimate-u0}
For some constant $C_t$ and $u\in C^\infty(X)$,
$$
\pm\Real\langle\Imag P_t(z)u,\Psi^2_{0\pm}u\rangle\leq -C_t^{-1}h\log(1/h)\|\Psi_{0\pm}u\|_{L^2}^2
+\mathcal O_t(h\tilde h)\|u\|_{L^2}^2.
$$
\end{lemm}
%%%%%%%%%%%%%%%%%%%%%%%%%%%%%%%%%%%%%%%%%%%%%%%%%%%%%%%%%%%%%%%%%%%%%%%%%%%%%%%%
\begin{proof}
By Lemma~\ref{l:localization}, it is enough to estimate
$$
\pm\langle(\Imag(T_sP(z)T_s^{-1}))\Psi_{0\pm}u,\Psi_{0\pm}u\rangle
\mp t\Real\langle i[P(0),F]\Psi_{0\pm}u,\Psi_{0\pm}u\rangle.
$$
Since $\Imag(T_sP(z)T_s^{-1})\in h\Psi^1$ and $\Psi_{0\pm}$ is compactly microlocalized,
the first term is $\mathcal O(h)\|\Psi_{0\pm}u\|_{L^2}^2$. Therefore, it is enough to show that
$$
\mp \Real\langle i[P(0),F]\Psi_{0\pm}u,\Psi_{0\pm}u\rangle\leq -C^{-1}h\log(1/h)\|\Psi_{0\pm}u\|_{L^2}^2
+\mathcal O(h^\infty)\|u\|_{L^2};
$$
by~\eqref{e:conj-f} and Lemma~\ref{l:microlocal-partition}(6),
this follows by Lemma~\ref{l:garding-1/2} applied to $(h\log(1/h))^{-1}i[P(0),F]$.
\end{proof}
%
%%%%%%%%%%%%%%%%%%%%%%%%%%%%%% END LEMMA %%%%%%%%%%%%%%%%%%%%%%%%%%%%%%
%
To deal with the $\Psi_K$ term, we have to modify our operator,
adding a term $-ith\widetilde A$, which provides positivity
in an $\mathcal O((h/\tilde h)^{1/2})$ size neighborhood of the trapped set.
Note that the resulting operator is not yet $\widetilde P_t(z)$;
the final operator $A$ that we use will also have a second microlocalization
factor, introduced below.
Let $C_{\hat f}$ be the constant from Lemma~\ref{l:f-hat}.
Let $\chi_1\in C^\infty(\mathbb R)$ be a nonnegative function
such that $\chi_1(\lambda)+\lambda=1$ for $\lambda\leq C_{\hat f}^{-1}/2$
and $\supp\chi_1\subset (-\infty,C_{\hat f}^{-1})$. %Since $H_p\hat f$ is bounded,
Then the function $\chi_1(H_p\hat f)$, defined on $V_K$, %lies in $\Syme(X)$ and
is supported $\mathcal O((h/\tilde h)^{1/2})$ close to $\iota(\widetilde K)$. Moreover,
$\chi_1(H_p\hat f)>0$ in the region where $H_p\hat f$ is not positive and
\begin{equation}
  \label{e:chi-better}
\chi_1(H_p\hat f)+H_p\hat f\geq C_{\hat f}^{-1}/2>0\text{ on }V_K.
\end{equation}
Take a real-valued $\tilde\chi\in C_0^\infty(U_K)$ equal to 1 near
$\WFh(\Psi_K)$. Then the function
\begin{equation}
  \label{e:a-tilde}
\tilde a=\chi_1(H_p\hat f)\tilde\chi
\end{equation}
is in $C_0^\infty(U_K)$ and thus can be extended to  $\overline T^*X$.
It follows from Lemma~\ref{l:f-hat} that $\tilde a$ lies in
the exotic class $\Syme$ from \S\ref{s:1/2}:
\begin{equation}
  \label{e:a-derivatives}
\partial^\alpha_{x,\xi}H_p^k \tilde a=\mathcal O((h/\tilde h)^{-|\alpha|/2}).
\end{equation}
Let $\widetilde A\in\Psie$ be any self-adjoint quantization of $\tilde a$;
note that $\WFh(\widetilde A)\subset U_K$. In fact,
$\WFh(\widetilde A)\subset \supp\tilde\chi\cap\iota(\widetilde K)$,
because by Lemma~\ref{l:f-hat}, for a fixed $(x,\xi)\not\in\iota(\widetilde K)$ and
$h/\tilde h$ small enough, we have $H_p \hat f(x,\xi)\geq C_{\hat f}^{-1}$
and thus $\chi_1(H_p\hat f(x,\xi))=0$.
In particular, recalling Lemma~\ref{l:escape}(4), we have
\begin{equation}
  \label{e:a-f0}
\WFh(\widetilde A)\cap\supp (H_pf_0)=\emptyset.
\end{equation}
%
%%%%%%%%%%%%%%%%%%%%%%%%%%%%%% BEGIN LEMMA %%%%%%%%%%%%%%%%%%%%%%%%%%%%%%
%
\begin{lemm}\label{l:estimate-k-1}
For $t$ large enough, some constant $C_t$ and  $u\in C^\infty(X)$,
$$
\Real\langle\Imag(P_t(z)-ith\widetilde A) u,\Psi_K^2 u\rangle\leq
-C_t^{-1}h\|\Psi_Ku\|_{L^2}^2+\mathcal O_{t}(h\tilde h)\|u\|_{L^2}^2.
$$
\end{lemm}
%%%%%%%%%%%%%%%%%%%%%%%%%%%%%%%%%%%%%%%%%%%%%%%%%%%%%%%%%%%%%%%%%%%%%%%%%%%%%%%%
\begin{proof}
By Lemma~\ref{l:localization}, and using that
$[\widetilde A,\Psi_K]=\mathcal O(h^{1/2}\tilde h^{1/2})_{\Psie}$ by Lemma~\ref{l:1/2-properties}(6),
we see that it suffices
to estimate
$$
\begin{gathered}
\langle (\Imag(T_sP(z)T_s^{-1}))\Psi_K u,\Psi_K u\rangle
-Mt\log(1/h)\Real\langle i[P(0),F_0]\Psi_K u,\Psi_K u\rangle\\
-t\Real\langle i[P(0),\widehat F]\Psi_K u,\Psi_K u\rangle
-th\langle\widetilde A\Psi_Ku,\Psi_Ku\rangle.
\end{gathered}
$$
Since $\Imag(T_sP(z)T_s^{-1}) \in h \Psi^1$ and $\Psi_K$ is compactly microlocalized, the first term can be estimated by $C_1h\|\Psi_K u\|_{L^2}^2$, where $C_1$ is independent of $t$.
Since $\sigma(h^{-1}i[P(0),F_0])=H_pf_0\geq 0$ near $\WFh(\Psi_K)$ by Lemma~\ref{l:microlocal-partition}(8),
the second term is $\leq
\mathcal O_t(h^2\log(1/h))\|u\|_{L^2}^2$ by sharp G\r arding inequality~\eqref{e:sharpg}. Therefore,
it suffices to pick $t$ large enough
and prove that for some constant $C_2$ independent of $t$, we have
\begin{equation}\label{e:psi-k-internal}
-\Real((i[P(0),\widehat F]+h\widetilde A)\Psi_K u,\Psi_K u)\leq -C_2^{-1}h\|\Psi_K u\|_{L^2}^2+\mathcal O(h^\infty)\|u\|_{L^2}^2.
\end{equation}
Near $\WFh(\Psi_K) \subset U_K$, $\hat \chi = 1$ and $H_p(\hat\chi \hat f)=H_p\hat f=\mathcal O(1)_{\Syme}$; therefore,
by Lemma~\ref{l:kangaroo}(3),
$i[P(0),\widehat F]+h\widetilde A\in h\Psie$
and, when multiplied by $h^{-1}$, its principal symbol is $H_p\hat f+\chi_1(H_p\hat f)+\mathcal O(\tilde h)$; by~\eqref{e:chi-better}
this symbol is positive and it remains to apply Lemma~\ref{l:garding-1/2}.
\end{proof}
%
%%%%%%%%%%%%%%%%%%%%%%%%%%%%%% END LEMMA %%%%%%%%%%%%%%%%%%%%%%%%%%%%%%
%
We will now fix $t$ and forget the dependence of the remainders on it.

It is finally time to use second microlocalization
and construct the operator $A$.
Let $\widehat P\in\Psi^2$ be any self-adjoint operator
elliptic near the fiber infinity $S^*X$ and 
whose principal symbol $\hat p$
is equal to $p$ in $U_K$. Take a function $\chi\in C_0^\infty(\mathbb R)$
equal to 1 near 0 and put
\begin{equation}
  \label{e:finally}
A=\chi((\tilde h/h)\widehat P)\widetilde A,\
\widetilde P_t(z)=P_t(z)-ithA.
\end{equation}
We use the ellipticity of $\widetilde P_t(z)$ away from the energy surface
to estimate the difference $A-\widetilde A$:
%
%%%%%%%%%%%%%%%%%%%%%%%%%%%%%% BEGIN LEMMA %%%%%%%%%%%%%%%%%%%%%%%%%%%%%%
%
\begin{lemm}\label{l:estimate-k-aux}
For  $u\in C^\infty(X)$, and any $N$
$$
\|\Real (A-\widetilde A) u\|_{L^2}\leq \mathcal O(\tilde h/h)\|\widetilde P_t(z)u\|_{\Hh^{-N}}
+\mathcal O(\tilde h)\|u\|_{L^2}.
$$
\end{lemm}
%%%%%%%%%%%%%%%%%%%%%%%%%%%%%%%%%%%%%%%%%%%%%%%%%%%%%%%%%%%%%%%%%%%%%%%%%%%%%%%%
\begin{proof}
Since both $\widetilde A$ and $\chi((\tilde h/h)\widehat P)$ are self-adjoint, we get
$$
\Real(A-\widetilde A)={1\over 2}[\widetilde A,\chi((\tilde h/h)\widehat P)]
-(1-\chi((\tilde h/h)\widehat P))\widetilde A.
$$
Now, by~\eqref{e:a-derivatives} and Lemma~\ref{l:second-mic}(2),
$[\widetilde A,\chi((\tilde h/h)\widehat P)]=\mathcal O(\tilde h)$,
so we can drop the commutator term.
Write $\chi(\lambda)-1=\lambda\psi(\lambda)$, where
$\psi$ is a bounded function. 
By the functional calculus,
$$
-(1-\chi((\tilde h/h)\widehat P))\widetilde A=\psi((\tilde h/h)\widehat P)(\tilde h/h)\widehat P\widetilde A;
$$
since $\psi((\tilde h/h)\widehat P)$ is bounded
on $L^2$ uniformly in $h,\tilde h$, it is enough to prove that
$$
\|\widehat P\widetilde A u\|_{L^2}\leq \mathcal O(1)\|\widetilde P_t(z)u\|_{\Hh^{-N}}
+\mathcal O(h)\|u\|_{L^2}.
$$
Since $[\widehat P,\widetilde A]=\mathcal O(h)$ by~\eqref{e:a-derivatives} and Lemma~\ref{l:1/2-properties}(7), this reduces to
$$
\|\widetilde A\widehat P u\|_{L^2}\leq \mathcal O(1)\|\widetilde P_t(z)u\|_{\Hh^{-N}}
+\mathcal O(h)\|u\|_{L^2}.
$$
Now, $\widehat P=T_s(P(0)-iQ)T_s^{-1}+\mathcal O(h)$ microlocally in $U_K$; therefore, it suffices to show that
$$
\|\widetilde AT_s(P(0)-iQ)T_s^{-1}u\|_{L^2}\leq \mathcal O(1)\|\widetilde P_t(z)u\|_{\Hh^{-N}}
+\mathcal O(h)\|u\|_{L^2}.
$$
The latter estimate follows from~\eqref{e:conj-dec}, $P_t(z) = \widetilde P_t(z) + \mathcal O(h)$, \eqref{e:conj-comm},
and the fact that
$[P(0)-iQ,F]=\mathcal O(h)$ microlocally near $\WFh(\widetilde A)$; indeed, $Q=0$
and $[P(0),\widehat F]=\mathcal O(h)$ microlocally on $U_K \supset \WFh(\widetilde A)$, and $[P(0),F_0]=\mathcal O(h^2)$
microlocally near $\WFh(\widetilde A)$ by~\eqref{e:a-f0}.
\end{proof}
%
%%%%%%%%%%%%%%%%%%%%%%%%%%%%%% END LEMMA %%%%%%%%%%%%%%%%%%%%%%%%%%%%%%
%
We can now combine the previous two lemmas to prove
%
%%%%%%%%%%%%%%%%%%%%%%%%%%%%%% BEGIN LEMMA %%%%%%%%%%%%%%%%%%%%%%%%%%%%%%
%
\begin{lemm}\label{l:estimate-k}
For $t$ large enough, some constant $C_t$, and  $u\in C^\infty(X)$,
$$
\Real\langle\Imag \widetilde P_t(z) u,\Psi_K^2 u\rangle\leq
-C_{t}^{-1}h\|\Psi_Ku\|_{L^2}^2+\mathcal O_{t}(h\tilde h)\|u\|_{L^2}^2
+\mathcal O_{t}(\tilde h)\|\widetilde P_t(z)u\|_{\Hh^{-N}}\|u\|_{L^2}.
$$
\end{lemm}
%%%%%%%%%%%%%%%%%%%%%%%%%%%%%%%%%%%%%%%%%%%%%%%%%%%%%%%%%%%%%%%%%%%%%%%%%%%%%%%%
\begin{proof}
The left-hand side is
$$
\Real\langle\Imag(P_t(z)-ith\widetilde A)u,\Psi_K^2u\rangle
-th\Real\langle\Real(A-\widetilde A)u,\Psi_K^2u\rangle;
$$
the first term is estimated by Lemma~\ref{l:estimate-k-1} and
the second term is estimated by Lemma~\ref{l:estimate-k-aux}.
\end{proof}
%
%%%%%%%%%%%%%%%%%%%%%%%%%%%%%% END LEMMA %%%%%%%%%%%%%%%%%%%%%%%%%%%%%%
%

%
%%%%%%%%%%%%%%%%%%%%%%%%%%%%%% BEGIN PROOF %%%%%%%%%%%%%%%%%%%%%%%%%%%%%%
%
\subsection{Proof of Lemma~\ref{l:main}(1)}
  \label{s:ultimate1}

First, let $\Psi_1\in\{\Psi_{L_\pm},\Psi_{0\pm}\}$. Then
$$
2\Psi_1^*(\Real A)=[\Psi_1^*,\chi((\tilde h/h)\widehat P)]\widetilde A
+\chi((\tilde h/h)\widehat P)\Psi_1^*\widetilde A
+\Psi_1^*\widetilde A\chi((\tilde h/h)\widehat P).
$$
This operator is $\mathcal O(\tilde h)$, as the commutator above is $\mathcal O(\tilde h)$
by Lemma~\ref{l:second-mic}(2) and $\Psi_1^*\widetilde A=\mathcal O(h^\infty)$
by Lemma~\ref{l:microlocal-partition}(7).
Hence Lemmas~\ref{l:estimate-l-pm} and~\ref{l:estimate-u0}
are valid for $\widetilde P_t(z)$ in place of $P_t(z)$.
Now, put
\begin{equation}\label{e:Z}
Z=\Psi_{L_+}^2+\Psi_{0+}^2+\Psi_K^2-\Psi_{L_-}^2-\Psi_{0-}^2;
\end{equation}
then $Z=\pm 1$ microlocally on $\widetilde\Sigma_\pm$
by Lemma~\ref{l:microlocal-partition}(2).
For $u\in C^\infty(X)$, we have
$$
\Imag\langle\widetilde P_t(z)u,Zu\rangle=\langle\Imag\widetilde P_t(z)u, Zu\rangle
+{1\over 2i}\langle u,([\widetilde P_t(z),Z]+(Z-Z^*)\widetilde P_t(z))u\rangle.
$$
However, by~\eqref{e:conj-dec-approx}, Lemma~\ref{l:1/2-properties}(6), and
the fact that the principal symbol of $Z$ is real,
\[
[P_t(z),Z]+(Z-Z^*)P_t(z)=\mathcal O(h)_{\Psi^1}+\mathcal O(h^{3/2}\tilde h^{1/2}\log(1/h))_{\Psi^{\comp}_{1/2}},
\]
and it is microlocalized outside of $\widetilde\Sigma_+\cup\widetilde\Sigma_-$, that is,
on the elliptic set of $\Psi_E$. Also, $[A,Z]+(Z-Z^*)A=\mathcal O(\tilde h)_{L^2\to L^2}$
by Lemma~\ref{l:second-mic}(2). Therefore, by
Lemmas~\ref{l:estimate-l-pm}, \ref{l:estimate-u0}, and~\ref{l:estimate-k}, for $t$
large enough, we have 
\begin{equation}\label{e:big}
\begin{gathered}
\Imag\langle\widetilde P_t(z)u,Zu\rangle=\langle\Imag\widetilde P_t(z)u,Zu\rangle
+\mathcal O(h)\|\Psi_E u\|_{\Hh^{1/2}}^2
+\mathcal O(h\tilde h)\|u\|_{L^2}^2\\\leq
-C^{-1}h(\|\Psi_{L_+}u\|_{\Hh^{1/2}}^2+\|\Psi_{0+}u\|_{\Hh^{1/2}}^2+\|\Psi_Ku\|_{\Hh^{1/2}}^2
+\|\Psi_{L_-}u\|_{\Hh^{1/2}}^2+\|\Psi_{0-}u\|_{\Hh^{1/2}}^2)\\
+\mathcal O(h\tilde h)\|u\|_{L^2}^2+\mathcal O(h)\|\Psi_E u\|_{\Hh^{1/2}}^2
+\mathcal O(\tilde h)\|\widetilde P_t(z) u\|_{\Hh^{-N}}\|u\|_{L^2}.
\end{gathered}
\end{equation}
Combining this with Lemma~\ref{l:microlocal-partition}(3), we get
$$
\|u\|_{\Hh^{1/2}}^2\leq \mathcal O(1)\|\Psi_Eu\|_{\Hh^{1/2}}^2+\mathcal O(\tilde h)\|u\|_{L^2}^2+\mathcal O(h^{-1})\|\widetilde P_t(z)u\|_{\Hh^{-1/2}}
\|u\|_{\Hh^{1/2}}
$$
and therefore for $\tilde h$ small enough,
\begin{equation}\label{l:ultimate-internal}
\|u\|_{\Hh^{1/2}}\leq \mathcal O(1)\|\Psi_Eu\|_{\Hh^{1/2}}+\mathcal O(h^{-1})\|\widetilde P_t(z)u\|_{\Hh^{-1/2}}.
\end{equation}
Finally, by Lemma~\ref{l:microlocal-partition}(1) and the elliptic estimate~\eqref{e:elliptic},
$$
\|\Psi_Eu\|_{\Hh^{1/2}}\leq \mathcal O(1)\|(P(0)-iQ)u\|_{\Hh^{-1/2}}+\mathcal O(h^\infty)\|u\|_{\Hh^{1/2}};
$$
combining this with~\eqref{e:conj-dec-approx} and using $\widetilde P_t(z) = P_t(z) + \mathcal O(h)$, we get
\begin{equation}\label{e:elll}
\|\Psi_E u\|_{\Hh^{1/2}}\leq \mathcal O(1)\|\widetilde P_t(z)u\|_{\Hh^{-1/2}}+\mathcal O(h\log(1/h))\|u\|_{\Hh^{1/2}};
\end{equation}
substituting this into~\eqref{l:ultimate-internal} 
and removing the $\mathcal O(h\log(1/h))$ error, we obtain~\eqref{e:main-estimate}.\qed
%
%%%%%%%%%%%%%%%%%%%%%%%%%%%%%%% END PROOF %%%%%%%%%%%%%%%%%%%%%%%%%%%%%%%
%

%
%%%%%%%%%%%%%%%%%%%%%%%%%%%%%% BEGIN PROOF %%%%%%%%%%%%%%%%%%%%%%%%%%%%%%
%
\subsection{Proof of Lemma~\ref{l:main}(2)}
  \label{s:ultimate2}

We follow the proof of part~(1), but with an additional
positive term coming from $\Imag z>0$. Let $Z$ be as in~\eqref{e:Z};
applying~\eqref{e:big} to $\Real z$ instead of $z$ and dropping the negative terms on the right hand side,
we get
\begin{equation}\label{e:im-aux}
\Imag\langle\widetilde P_t(\Real z)u,Zu\rangle\leq
\mathcal O(h\tilde h)\|u\|_{L^2}^2
+\mathcal O(h)\|\Psi_E u\|_{\Hh^{1/2}}^2
+\mathcal O(\tilde h)\|\widetilde P_t(\Real z)u\|_{\Hh^{-1/2}}\|u\|_{L^2}. 
\end{equation}
Now, by Lemma~\ref{l:vasy-1}(6)
$$
P(z)-P(\Real z)=\mathcal O(|\Imag z|)_{\Psi^1},
$$
and so
\begin{equation}\label{e:main21}
T_s(P(z)-P(\Real z))T_s^{-1}=i\Imag z \,\partial_zP(0)
+\mathcal O(|\Imag z|^2+h|\Imag z|)_{\Psi^1}.
\end{equation}
The conjugation by $e^{tF}$ maps  $\Psi^1$
to $\Psi^1+\Psie$ continuously, by the Bony--Chemin
theorem (see the proof of Lemma~\ref{l:conjugation}). Moreover,
we have by Lemma~\ref{l:kangaroo}(3)
\begin{equation}\label{e:main22}
[F,\partial_z P(0)]=\mathcal O(h^{1/2}\tilde h^{1/2}|\Imag z|)_{\Psie}.
\end{equation}
Using $\widetilde P_t(z)-\widetilde P_t(\Real z) = P_t(z) - P_t(\Real z)$, the expansion \eqref{e:ptexpansion}, \eqref{e:main21}, \eqref{e:main22} and 
proceeding as in the proof of Lemma~\ref{l:conjugation}, we get
$$
\widetilde P_t(z)-\widetilde P_t(\Real z)=i\Imag z \,\partial_zP(0)
+\mathcal O(|\Imag z|^2+h^{1/2}\tilde h^{1/2}|\Imag z|)_{\Psi^1+\Psie}.
$$
Therefore, by~\eqref{e:im-aux} and since $\Imag z\geq C_0h$
$$
\begin{gathered}
\Imag\langle\widetilde P_t(z)u, Zu\rangle\leq 
\Imag z\Real\langle \partial_z P(0)u,Zu\rangle
+ \mathcal O(h)\|\Psi_Eu\|_{\Hh^{1/2}}^2\\
+\mathcal O(\tilde h)\|\widetilde P_t(z)u\|_{\Hh^{-1/2}}\|u\|_{L^2}
+\mathcal O(|\Imag z|^2+\tilde h|\Imag z|)\|u\|_{\Hh^{1/2}}^2.
\end{gathered}
$$
Now Lemma~\ref{l:microlocal-partition}(3) and Lemma~\ref{l:microlocal-partition} (9)
imply that $\Psi_E$ is elliptic on $\{\Real(\lxir^{-1}\sigma (Z^*\partial_z P(0))) \ge 0\}$, so by the non-sharp G\r arding inequality \eqref{e:nonsharpg}  we get
$$
\Real\langle \partial_z P(0)u,Zu\rangle\leq -C^{-1} \|u\|_{\Hh^{1/2}}^2
+C\|\Psi_E u\|_{\Hh^{1/2}}^2.
$$
Thererefore, since $ C_0 h \le \Imag z \le \varepsilon$,
$$
\|u\|_{\Hh^{1/2}}^2\leq C(\|\Psi_E u\|_{\Hh^{1/2}}^2+
(\Imag z)^{-1}\|\widetilde P_t(z) u\|_{\Hh^{-1/2}}\|u\|_{\Hh^{1/2}})
+\mathcal O(\tilde h+|\Imag z|)\|u\|_{\Hh^{1/2}}^2.
$$
Combining this with~\eqref{e:elll} and the fact that $\tilde h$ and $\Imag z$
are small, we get~\eqref{e:improved-estimate}.\qed
%
%%%%%%%%%%%%%%%%%%%%%%%%%%%%%%% END PROOF %%%%%%%%%%%%%%%%%%%%%%%%%%%%%%%
%

%
%%%%%%%%%%%%%%%%%%%%%%%%%%%%%% BEGIN PROOF %%%%%%%%%%%%%%%%%%%%%%%%%%%%%%
%
\subsection{Proof of Lemma~\ref{l:main}(3)}
  \label{s:ultimate3}

By Lemma~\ref{l:approximation} and~\eqref{e:finally},
it suffices to show that
$$
\begin{gathered}
V_R=\{\exp(tH_{\hat p})(x,\xi)\mid
|t|\leq R,\
(x,\xi)\in (\supp\tilde a\cap \hat p^{-1}(0))+B_{\hat p^{-1}(0)}(R(h/\tilde h)^{1/2})\}
\end{gathered}
$$
has, as a subset of $\hat p^{-1}(0)$, $2n-1$ dimensional volume
$\mathcal O((h/\tilde h)^{n-1-\nu})$.
Here $2\nu+1$ is bigger than the upper Minkowski
dimension of $K$, or equal to it in the case of a trapped set of pure dimension.
By the definition~\eqref{e:a-tilde} of $\tilde a$ and the fact that $\supp \chi_1 \subset (-\infty,C_{\hat f}^{-1})$ together with Lemma~\ref{l:f-hat},
we see that $\supp\tilde a\subset
\iota(\widetilde K)+B(C_{\hat f}(h/\tilde h))^{1/2}$, so that
$$
V_R\subset\{\exp(tH_{\hat p})(x,\xi)\mid
|t|\leq R,\
(x,\xi)\in\hat p^{-1}(0)\cap
(\iota(\widetilde K)+B((R+C_{\hat f})(h/\tilde h)^{1/2}))\}.
$$
However, note that $\iota(\widetilde K)$ is invariant under $\exp(tH_{\hat p})$;
therefore, there exists a constant $R'$ depending on $R$ such that
$$
V_R\subset {\hat p}^{-1}(0)\cap(\iota(\widetilde K)+B(R'(h/\tilde h)^{1/2})).
$$
By~\eqref{e:p-p0} and since $p_0+1$ is a homogeneous polynomial of degree 2 in the fibers,
$\widetilde K$ is diffeomorphic to the product of $K=\widetilde K\cap p^{-1}(0)$ and an interval;
therefore, for some $R''$,
$$
V_R\subset \iota(K)+B_{\hat p^{-1}(0)}(R''(h/\tilde h)^{1/2}).
$$
By the definition~\eqref{e:minkdef}, the volume of an $\varepsilon$-neighborhood of $\iota(K)$ is $\mathcal O(\varepsilon^{2(n-1-\nu)})$; thus
$$
\Vol_{\hat p^{-1}(0)}(V_R)\leq C(h/\tilde h)^{n-1-\nu}.\qed
$$
%
%%%%%%%%%%%%%%%%%%%%%%%%%%%%%%% END PROOF %%%%%%%%%%%%%%%%%%%%%%%%%%%%%%%
%

\appendix
%%%%%%%%%%%%%%%%%%%%%%%%%%%%%%%%%%%%%%%%%%%%%%%%%%%%%%%%%%%%%%%%%%%%%%%%%%%%%%%%
%                                   APPENDIX                                   %
%%%%%%%%%%%%%%%%%%%%%%%%%%%%%%%%%%%%%%%%%%%%%%%%%%%%%%%%%%%%%%%%%%%%%%%%%%%%%%%%
\section{Quasifuchsian convex cocompact groups}
\label{s:quasifuchsian}

In this Appendix we describe in more detail the construction of the groups used
in the examples in Figure~\ref{f:puredim}. Recall that a finitely generated discrete group 
of M\"obius 
transformations of the Riemann sphere $\mathbb C \cup \{\infty\}$
is \textit{Fuchsian} if it keeps invariant some disk or half-plane.
Let $\Gamma_0$ be the Fuchsian group generated by 
$\{\mathcal A_1,\mathcal  B_1,\mathcal  A_2,\mathcal  B_2\}$, where all the transformations
preserve the unit disk, and $\mathcal A_1$ maps the exterior of the disk
$C_1$ onto the interior
of the disk $C_3$, $\mathcal B_1$ maps the exterior of $C_2$ onto the interior of $C_4$, and so on
(see Figure~\ref{f:puredim2}). If $\Gamma_0$ acts on the unit disk model of $\mathbb H^2$, then  
$\Gamma_0 \backslash \mathbb H^2$ is a compact surface of genus $2$ 
(see e.g.~\cite[\S4.3, Example C]{k}).

%
%%%%%%%%%%%%%%%%%%%%%%%%%%%%%% BEGIN FIGURE %%%%%%%%%%%%%%%%%%%%%%%%%%%%%%
\begin{figure}[h]
\includegraphics[width=6cm]{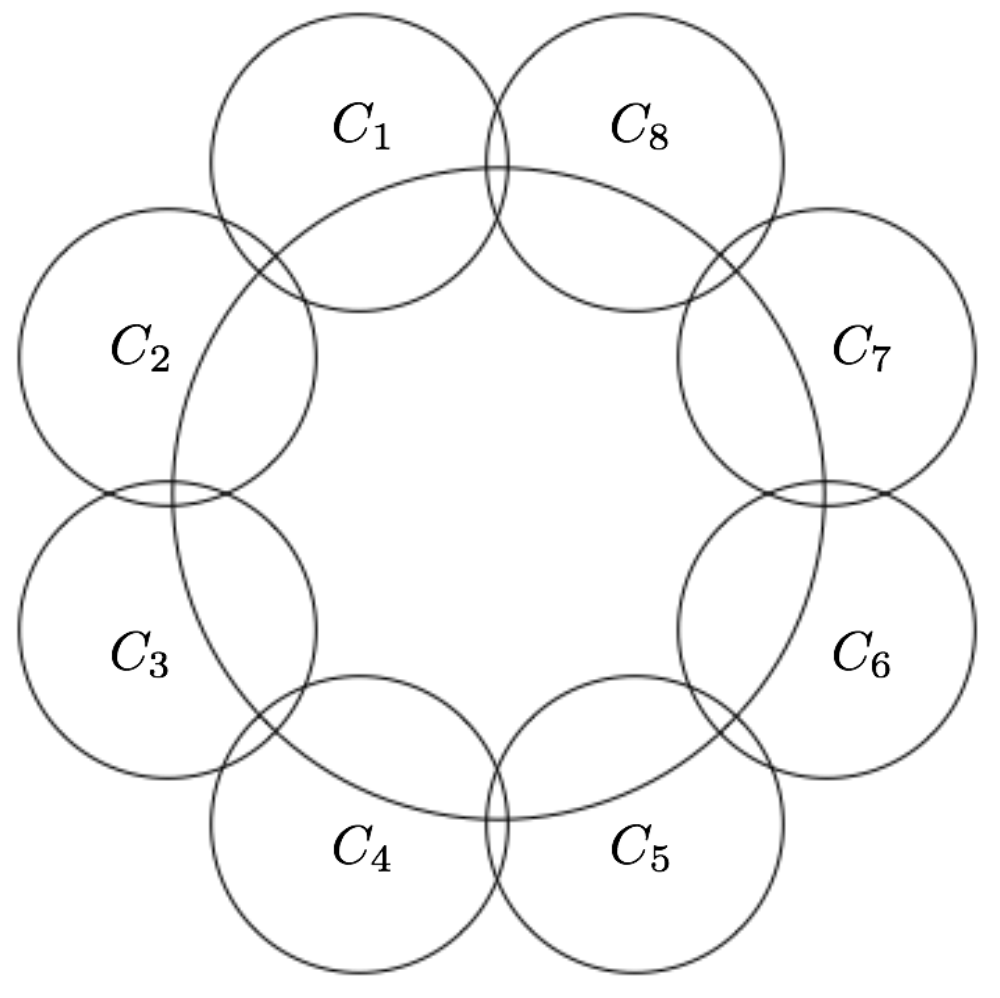}
\hspace{1cm}
\includegraphics[width=6cm]{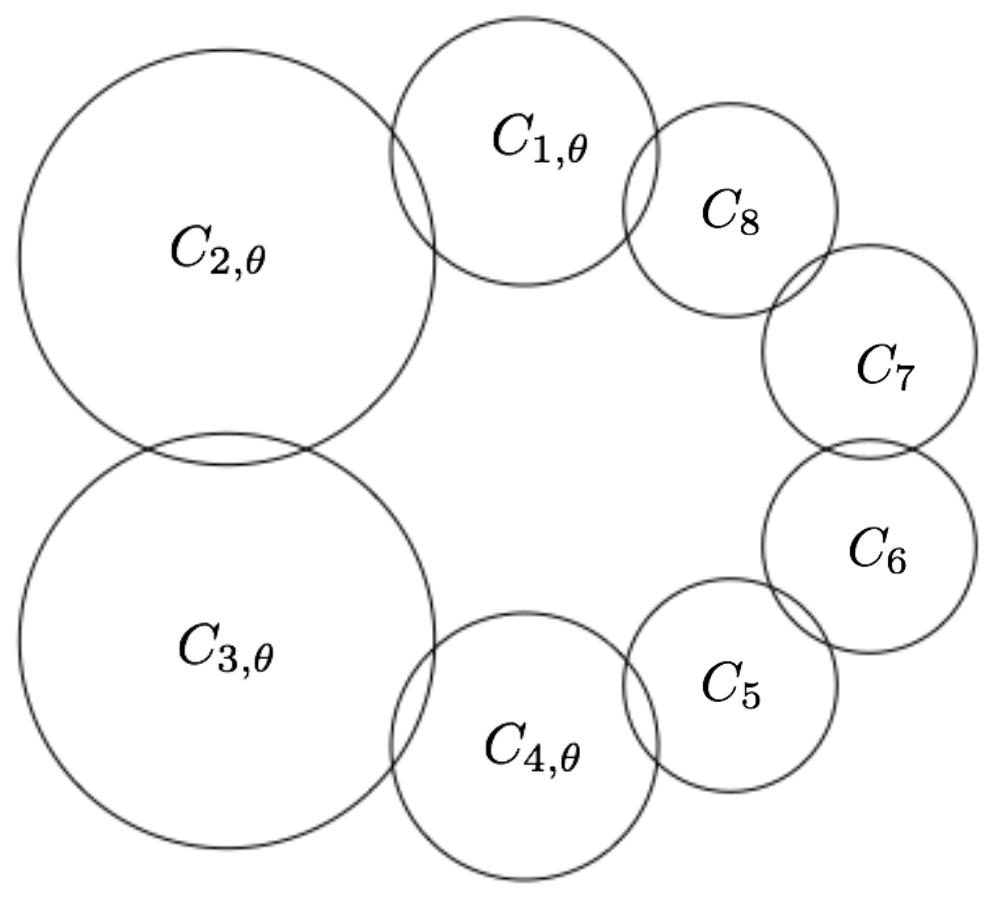}
\caption{(Quasifuchsian bending.)
The circles $C_1,\dots,C_8$ are orthogonal to the unit circle 
and enclose a regular hyperbolic octagon of area $4 \pi$. The 
circles $C_{1,\theta},\dots C_{4,\theta}$ are the images of 
$C_1,\dots,C_4$ under $M_\theta$.}
\label{f:puredim2}
\end{figure}
%
%%%%%%%%%%%%%%%%%%%%%%%%%%%%%%% END FIGURE %%%%%%%%%%%%%%%%%%%%%%%%%%%%%%%

If  $\Gamma_0$ instead acts on  
$\mathbb H^3$ by M\"obius transformations 
on the sphere at infinity ($\mathcal A_1$ extends to an isometry of $\mathbb H^3$
by mapping the half space whose boundary at infinity is the exterior of $C_1$ 
to the  half space whose boundary at infinity is the interior of $C_3$, and so on)
then
$\Gamma_0 \backslash \mathbb H^3$ is isometric to  
$\mathbb R \times \Gamma_0 \backslash \mathbb H^2$ with metric
$dr^2 + (\cosh^2 r )dS$, where $dS$ is the metric on $\Gamma_0 \backslash \mathbb H^2$,
 and  $\Gamma_0 \backslash \mathbb H^3$ is
convex cocompact with
limit set the unit circle, 
and $\delta_{\Gamma_0} = 1$.

 Let 
$\Gamma_\theta = \langle M_\theta \mathcal A_1 M_\theta^{-1}, 
M_\theta \mathcal B_1 M_\theta^{-1}, \mathcal A_2, \mathcal B_2 \rangle$,
where $M_\theta$ is the rotation of 
the Riemann sphere
$\mathbb C \cup \{\infty\}$
 by angle $\theta$ which 
fixes $\{i, -i\}$ and moves $\mathbb R$
 to the left (so that $M_\theta \mathcal A_1 M_\theta^{-1}$ maps the exterior of 
 $C_{1,\theta}$ onto the interior
of $C_{3,\theta}$, and so on). This is a \textit{quasifuchsian
bending} of $\Gamma_0$ in the sense of
\cite[\S VIII.E.3]{mas}.  When 
$\theta \ne 0$ the group is no longer Fuchsian because e.g 
 $\operatorname{Tr} M_\theta \mathcal A_1 M_\theta^{-1} \mathcal A_2 \not \in \mathbb R$, and hence 
$\Gamma_\theta$
is not a group of isometries of $\mathbb H^2$. Nonetheless, it is still
a  group of isometries of $\mathbb H^3$ and,
for $|\theta| \ne 0$ small enough,
$\Gamma_\theta \backslash \mathbb H^3$ is still convex cocompact 				and
diffeomorphic to $\mathbb R \times \Gamma_0 \backslash \mathbb H^2$,
 but the metric is no longer a warped product and the limit
set $\Lambda_{\Gamma_\theta}$ is now a \textit{quasicircle} with 
dimension
$\delta_{\Gamma_\theta} \in(1,2)$~\cite{bowen,su2,bj}.
In Figure~\ref{f:puredim} we plot $\Lambda_{\Gamma_\theta}$
for $\theta  = 0.5$,
 using Mathematica code based on that of \cite[Appendix]{geerlings}.

%%%%%%%%%%%%%%%%%%%%%%%%%%%%%%%%%%%%%%%%%%%%%%%%%%%%%%%%%%%%%%%%%%%%%%%%%%%%%%%%
%                                   APPENDIX                                   %
%%%%%%%%%%%%%%%%%%%%%%%%%%%%%%%%%%%%%%%%%%%%%%%%%%%%%%%%%%%%%%%%%%%%%%%%%%%%%%%%
\section{The hyperbolic cylinder}
  \label{s:picture}

In this Appendix we consider the hyperbolic cylinder
$M=(-1,1)_r\times \mathbb S^1_{\tilde y}$ with  metric 
$$
g={dr^2\over (1-r^2)^2}+{d\tilde y^2\over 1-r^2}.
$$
We explain how this asymptotically hyperbolic manifold fits
into the general framework of~\cite{v2} and why Figure~\ref{f:vasy}
represents the phase space picture for the modified operator.

Theorems~\ref{l:theorem-weak} and~\ref{l:theorem-strong} will apply with 
$\nu=\delta_\Gamma = 0$: note that
$M \simeq \langle z \mapsto e^{2\pi}z \rangle \backslash \mathbb H^2$, where we use the upper 
half plane model of $\mathbb H^2$. In this case the resonances are actually known to
lie on a lattice~\cite[Appendix]{gzjfa}. More generally, when as in this case the trapped 
set consists of a single hyperbolic orbit, the resonances are asymptotic to a 
lattice~\cite{g-s}.

First, note that we can bring the metric to the form~\eqref{e:as-hyp}
near $\{r=\pm 1\}$ by taking
$$
\tilde x=2\sqrt{1\mp r\over 1\pm r},\
g={d\tilde x^2\over\tilde x^2}+\bigg(1+{\tilde x^2\over 4}\bigg)^2{d\tilde y^2\over \tilde x^2}.
$$
Then a boundary defining function of $\overline M_{\even}$ is given by
$$
\mu=1-r^2.
$$
(Strictly speaking, for the calculations in \S\ref{s:ah} and \cite{v2} to go through without changes,
we need $\mu = \tilde x^2$ near the conformal
boundary; however, our $\mu$ makes the formulas  simpler and as
$\mu=\tilde x^2(1+\mathcal O(\tilde x^2))$, the analysis is the same.)
The Laplacian on $M$ is
$$
\Delta_{g}=(1-r^2)^2D_r^2+ir(1-r^2)D_r+(1-r^2)D_{\tilde y}^2.
$$
To simplify the formula for the modified Laplacian~\eqref{e:p1def},
we put $e^\phi=\mu^{1/2}$.
We have
$$
\begin{gathered}
P(z)=\mu^{-5/4}\mu^{i(z+1)/(2h)}(h^2(\Delta_g-1/4)-(z+1)^2)\mu^{-i(z+1)/(2h)}\mu^{1/4}\\
=\mu(hD_r)^2+2(z+1)r(hD_r)+D_{\tilde y}^2-(z+1)^2+\mathcal O(h)_{\Psi^1}.
\end{gathered}
$$
This operator extends to $X=\mathbb R_r\times \mathbb S^1_{\tilde y}$
(in the rest of the paper, and in \cite{v2}, $X$ is compact, but we will not need this here).
Note that for $\mu > 0$ it is elliptic (Laplacian-like) but for
$\mu<0$ it is hyperbolic (d'Alembertian-like).
Take coordinates $(r,\tilde y,\zeta,\tilde \eta)$  on $T^*X$, with $\zeta$ dual to $r$ and
 $\tilde \eta$ dual to $\tilde y$.
We use the momentum $\zeta$ instead of the momentum
$\tilde \xi=-\zeta/(2r)$, dual to $\mu$, to avoid a coordinate singularity at $r=0$.
The principal symbol of $P(0)$ is
\begin{equation}\label{e:pex}
p=\mu\zeta^2+2r\zeta+\tilde\eta^2-1,
\end{equation}
and the Hamiltonian flow is
\begin{equation}
  \label{e:hpex}
H_p=2(\mu\zeta+r)\partial_r
+2\zeta(r\zeta-1)\partial_\zeta
+2\tilde\eta \partial_{\tilde y}.
\end{equation}
Our phase space $\overline{T}^*X$ is as in \S\ref{s:prelim.basics} and is a ball bundle on $X$, and we denote a typical point by $(x,\xi)$. We now study the characteristic set $\{\lxir^{-2}p=0\}$ and the rescaled Hamiltonian
flow $\lxir^{-1}H_p$, beginning with the behavior near  fiber
infinity (the prefactors $\lxir^{-2}$ and $\lxir^{-1}$ make the symbol and flow extend smoothly to fiber infinity, $S^*X = \partial \overline{T}^*X)$. We use the  coordinates $\check\zeta=\lxir^{-1}\zeta$, 
$\check\eta=\lxir^{-1}\tilde\eta$ on the fibers in $\overline T^*X$;
then $(\check\zeta,\check\eta)$ lies on the circle of radius $ \lxir^{-1} |\xi|$ and
$$
\lxir^{-2}p=\mu\check\zeta^2+2r\lxir^{-1}\check\zeta+\check\eta^2-\lxir^{-2}.
$$
On $S^*X$, we have $\lxir^{-1}=0$ and thus the characteristic set is given
by
$$
\{\lxir^{-2}p=0\}\cap S^*X=\{\mu\check\zeta^2+\check\eta^2=0\}.
$$
For $\mu>0$, this equation has no solutions, which corresponds to the 
characteristic set not touching  fiber infinity. For $\mu\leq 0$,
we have
$$
\begin{gathered}
\{\lxir^{-2}p=0\}\cap S^*X=\Sigma'_+\cup \Sigma'_-,\\
\Sigma'_\pm=\{\mu\leq 0,\ \lxir^{-1}=0,\ \check\zeta=\mp \sgn r/\sqrt{1-\mu},\
\check\eta^2=-\mu/(1-\mu)\},
\end{gathered}
$$
because $\check\zeta^2 + \check\eta^2 = 1$ on $S^*X$.
In particular, we have
$$
L_\pm=\Sigma'_\pm\cap \{\mu=0\}=\{\mu=0,\ \lxir^{-1}=0,\ \check\zeta=\mp\sgn r,\
\check\eta=0\}.
$$
On $\{\check\zeta\neq 0\} \supset \Sigma'_+ \cup \Sigma'_-$, we can pass to the
system of coordinates
$$
(r,\tilde y,\tilde\rho=|\zeta|^{-1},\hat\eta=\tilde\rho\tilde\eta).
$$
Near $\Sigma'_\pm$ we have $\zeta=\mp \sgn r \tilde\rho^{-1}$ and thus, 
using $\partial_\zeta = \pm \sgn r \tilde \rho(\tilde \rho \partial_{\tilde \rho} + \hat \eta 
\partial_{\hat \eta})$,
\begin{equation}\label{e:hpcyl}
\tilde\rho H_p=\mp2\sgn r((\mu\mp |r|\tilde\rho)\partial_r
-(r \pm \tilde \rho \sgn r)(\tilde\rho \partial_{\tilde\rho}+\hat\eta \partial_{\hat\eta}))
+2\hat\eta \partial_{\tilde y}.
\end{equation}
Since $\mu = \tilde\rho = \hat\eta= 0$ on $L_\pm$, we see that $L_\pm$ consists 
of fixed points for $\tilde\rho H_p$.
We also get
$$
\tilde\rho H_p\mu|_{\Sigma'_\pm}=\pm 4|r|\mu.
$$
Therefore, the flow lines on $\Sigma'_+\cap\{\mu<0\}$ go to $\mu=-\infty$ in the
forward direction and to $L_+$ in the backward direction, while the flow
lines on $\Sigma'_-\cap\{\mu<0\}$ go to $\mu=-\infty$ in the backward direction
and to $L_-$ in the forward direction. This is displayed in Figure~\ref{f:funny}.
%
%%%%%%%%%%%%%%%%%%%%%%%%%%%%%% BEGIN FIGURE %%%%%%%%%%%%%%%%%%%%%%%%%%%%%%
%
\begin{figure}
\includegraphics{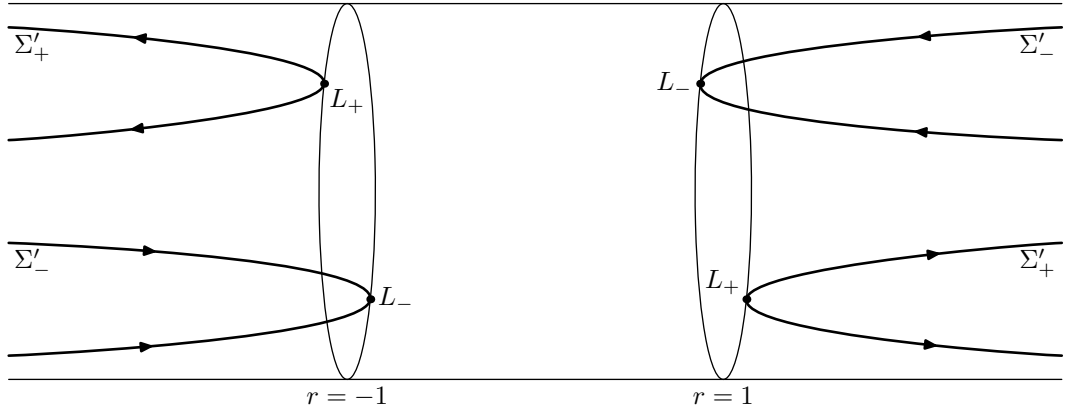}
\caption{Global dynamics of the Hamiltonian flow on the fiber infinity.}
\label{f:funny}
\end{figure}
%
%%%%%%%%%%%%%%%%%%%%%%%%%%%%%%% END FIGURE %%%%%%%%%%%%%%%%%%%%%%%%%%%%%%%
%

We can also use \eqref{e:hpcyl} study the dynamics of the flow of $H_p$ on $\overline T^*M$, not just on
$S^*M$, near $L_\pm$. Omitting the $\tilde y$ variable as the flow does not depend on it,
we find that
$$
\tilde\rho H_p
\begin{pmatrix}\mu\\\tilde\rho\\\hat\eta\end{pmatrix}=
\begin{pmatrix}
\pm 4&-4&0\\
0&\pm 2&0\\
0&0&\pm 2
\end{pmatrix}
\begin{pmatrix}\mu\\\tilde\rho\\\hat\eta\end{pmatrix}
+\mathcal O(\mu^2+\tilde\rho^2+\hat\eta^2).
$$
We see that $L_+$ is a source and $L_-$ is a sink. In fact, the 
eigenvectors of the linearized flow at $L_\pm$ are
$\partial_\mu$ with eigenvalue $\pm 4$,
$2 \partial_\mu\pm \partial_{\tilde\rho}$ with eigenvalue $\pm 2$,
and $\partial_{\hat\eta}$ with eigenvalue $\pm 2$.
The behavior of the linearized system is pictured on Figure~\ref{f:vasy}
(the horizontal coordinate is $r$, and the vertical coordinate 
$\tilde \zeta = \zeta / \langle \zeta \rangle$ is a compactification of
$\zeta$).

We now study the semiclassical behavior, that is, dynamics in the
interior $T^*X$ of $\overline T^*X$. First of all, we fix $r$ and study the
set of solutions in $(\zeta,\tilde\eta)$ to the characteristic equation
$$
\mu\zeta^2+2r\zeta+\tilde\eta^2-1 =  \mu(\zeta + r/\mu)^2+\tilde\eta^2-(1 + r^2/\mu)=0.
$$
This is an ellipse when $\mu>0$, a parabola when $\mu=0$, and a hyperbola when $\mu<0$.
Therefore, in $T^*X$ the characteristic set has one connected component for $\mu\geq 0$
and two components for $\mu<0$; on the other hand, in $\overline T^*X$ it has one connected
component for $\mu>0$ and two components for $\mu\leq 0$, the additional connected
component for $\mu=0$ being exactly $L_-$, and the intersections of the connected
components with $S^*X$ being exactly $\Sigma'_\pm$. We then see, as in Figure~\ref{f:chset},
that the characteristic
set $\{\lxir^{-2}p=0\}\subset \overline T^*X$ can be split into two components
$\Sigma_+$ and $\Sigma_-$ (the latter consisting of two pieces, corresponding
to $\pm r\geq 1$), so that $\Sigma'_\pm=\Sigma_\pm\cap S^*X$
and $\Sigma_-\subset \{\mu\leq 0\}$. More precisely, we note that near $\{\mu=0\}$,
the characteristic set does not intersect the surface
$$
\{\tilde\xi=-1/2\}=\{\zeta=r\}\subset \overline T^*X,
$$
where $\tilde\xi = -\zeta/(2r)$ is the dual variable to $\mu$.
Therefore, for $\varepsilon$ small enough, we can define $\Sigma_\pm$ in 
$\{|\mu|<\varepsilon\}$ as
$$
\Sigma_\pm\cap\{|\mu|<\varepsilon\}=\{\pm (\tilde \xi+1/2)>0\}\cap \{\lxir^{-2}p=0\}
=\{\mp r(\zeta-r)>0\}\cap \{\lxir^{-2}p=0\}.
$$
With that definition, $\Sigma_- \subset \{-\varepsilon< \mu \le 0\}$ for some $\varepsilon > 0$; hence,
we can extend $\Sigma_\pm$ to $\{\mu>-\varepsilon\}$ by requiring that
$\Sigma_-\cap \{\mu\geq\varepsilon\}=\emptyset$
and $\Sigma_+\cap \{\mu\geq\varepsilon\}=\{\lxir^{-2}p=0\}\cap\{\mu\geq\varepsilon\}$.
%
%%%%%%%%%%%%%%%%%%%%%%%%%%%%%% BEGIN FIGURE %%%%%%%%%%%%%%%%%%%%%%%%%%%%%%
%
\begin{figure}
\includegraphics{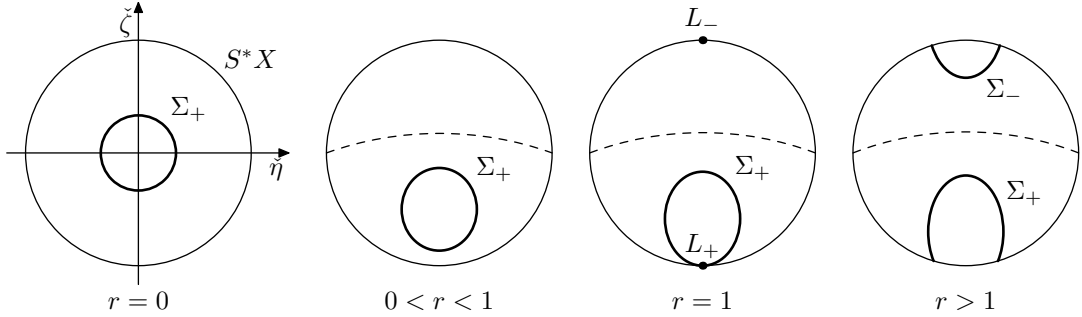}
\caption{The characteristic set in the radial compactified fiber
for various values of $r$. The dashed line is $\{\zeta=r\}$, separating
$\Sigma_+$ from $\Sigma_-$.}
\label{f:chset}
\end{figure}
%
%%%%%%%%%%%%%%%%%%%%%%%%%%%%%%% END FIGURE %%%%%%%%%%%%%%%%%%%%%%%%%%%%%%%
%

It remains to study the Hamiltonian flow in $T^*X$. If $\mu=0$,
we have $H_p\mu=-4r^2<0$; therefore, the flow lines on 
$\Sigma_+\setminus L_+ = \{\lxir^{-2}p=0\} \setminus  S^*X$
only cross $\{\mu=0\}$ in the direction of decreasing $\mu$. Together with the behavior
of the linearized flow near $L_\pm$ studied before, this gives the behavior
of the flow near $\{\mu=0\}$, as in Lemma~\ref{l:vasy-2}. It remains to analyze the behavior
of the flow in $\{\mu>0\}$. This can be related to the geodesic flow on the original
manifold; we then see that the trapped set corresponds to two trapped trajectories
$$
\iota(K)=\{r=\zeta=0,\ |\tilde\eta|=1\}
$$
while the incoming/outgoing tails (in the sense of \eqref{e:gammadef}) are given by
$$
\begin{gathered}
\iota(\Gamma_+)=\{\zeta = 0,\ |\tilde\eta|=1\},\\
\iota(\Gamma_-)=\{\zeta=-2 r/\mu,\ |\tilde\eta|=1\}.
\end{gathered}
$$
Note that $\iota(\Gamma_+)$ continues smoothly across $\{\mu=0\}$ and
 $\iota(\Gamma_-)$ converges backwards to $L_+$.

All components of Figure~\ref{f:vasy} are now in place. The horizontal
direction corresponds to $r$ (with $r$ increasing as we move to the right),
with two vertical lines marking $\partial \overline M_{\even}=\{r=\pm 1\}$, the conformal
boundary. The vertical direction corresponds to a compactification of the
momentum $\zeta$; the corresponding coordinate
$\tilde\zeta=\zeta/\langle\zeta\rangle$ is well-defined away from $\{\lxir^{-1}=\check\zeta=0\}$
on $\overline T^*X$, and thus on the whole characteristic set. The top and
bottom edges of the picture then correspond to $\{\pm\check\zeta>0,\ \lxir^{-1}=0\}\subset S^*X$.
 
The explicit formulas~\eqref{e:pex} and~\eqref{e:hpex} for $p$ and $H_p$ allow us to
define the escape functions $f_0$ and $\hat f$, constructed in the general case
in Lemmas~\ref{l:escape} and~\ref{l:f-hat}, more directly and explicitly here. 
We begin with $\hat f$, where we will use the fact that no conditions are imposed outside of
a small neighborhood of $\iota(\widetilde K)$. Observe first that
\[
\iota(\widetilde K) = \{r=\zeta = 0\},
\]
so we will be interested in estimates valid for $r$ and $\zeta$ sufficiently small.
Following~\cite[\S 4.2]{w-z},~\cite[\S 7]{sj-z},~\cite[\S 5]{sj}, let
\[
\varphi_+ = \zeta^2, \qquad \varphi_- = (\mu\zeta + 2r)^2,
\]
be functions measuring the distance squared 
to $\iota(\Gamma_+)$ and $\iota(\Gamma_-)$ respectively.
We then have
\[
\begin{gathered}
H_p \varphi_+ = 4\zeta^2(r\zeta - 1) = - 4  \varphi_+
 (1 +  \mathcal O(r^2 + \zeta^2)),\\
H_p \varphi_- = 4(\mu\zeta + 2r)^2(1-r\zeta) = 4\varphi_-(1 + \mathcal O(r^2 + \zeta^2)).
\end{gathered}
\]
(Near $\iota(\widetilde K)$ such estimates can be deduced from the
hyperbolicity of $\iota(\widetilde K)$ but in this  example we
can compute the derivatives directly).
Consequently
\begin{equation}\label{e:varphidiff}
H_p ( \varphi_- - \varphi_+ ) =  4 ( \varphi_- +
   \varphi_+ ) +  \mathcal O(r^4 + \zeta^4) .
\end{equation}
Hence, near $\iota(\widetilde K)$, $H_p ( \varphi_- -  \varphi_+ )  \ge (\varphi_- +   \varphi_+)/C $, 
which is 
a nonnegative function vanishing precisely on $\iota(\widetilde K)$.
To obtain a  lower bound of $1/C$ off
 a neighborhood of size $(h/\tilde h) ^{1/2}$ of $\iota(\widetilde K)$,
as asserted in Lemma~\ref{l:f-hat},
 we take the 
following `logarithmic flattening' of 
$\varphi_- - \varphi_+$:
\[
\hat f = \log((h/\tilde h) +  \varphi_-) - \log((h/\tilde h) +\varphi_+).
\]
Then, for $r$ and $\zeta$ sufficiently small,
\[
\begin{gathered}
H_p \hat f = 4\frac{\varphi_-(1+ \mathcal O(r^2 + \zeta^2))}{(h/\tilde h) +  \varphi_-}
+ 4 \frac {\varphi_+
 (1 +  \mathcal O(r^2 + \zeta^2))}{(h/\tilde h) +\varphi_+}\\
 \ge 2\frac{\varphi_-}{(h/\tilde h)+\varphi_-}
+ 2 \frac {\varphi_+
 }{(h/\tilde h)+\varphi_+},
\end{gathered}
\]
and this is uniformly bounded from below off of a 
neighborhood of size $(h/\tilde h) ^{1/2}$ of $\iota(\widetilde K)$.
The upper bounds on derivatives in  Lemma~\ref{l:f-hat} follow
from similar arguments.

To construct $f_0$ we begin with a near-global escape function
based on $\varphi_- - \varphi_+$:
\[
f_{00} = r \zeta + \frac 1 2 r^2.
\]
This function is strictly increasing along flowlines everywhere in 
$T^*X \cap \{|\mu| \le 1\} \setminus \iota(\widetilde K)$:
\[
H_p f_{00} = 2 \zeta^2  + 2 \mu \zeta r + 2 r^2 \ge \zeta^2 + r^2.
\]
 To obtain
$f_0$ we  precompose and
multiply by a smooth cutoff as in Lemma~\ref{l:escape}.

%%%%%%%%%%%%%%%%%%%%%%%%%%%%%%%%%%%%%%%%%%%%%%%%%%%%%%%%%%%%%%%%%%%%%%%%%%%%%%%%
%%%%%%%%%%%%%%%%%%%%%%%%%%%%%%%%%%%%%%%%%%%%%%%%%%%%%%%%%%%%%%%%%%%%%%%%%%%%%%%%
\noindent\textbf{Acknowledgments.}
We are very grateful to Maciej Zworski for suggesting this
project and for helpful and encouraging conversations, as well as
for many comments on earlier drafts of this paper. Thanks
also to Andr\'as Vasy for several detailed discussions about his
method of analysis on asymptotically hyperbolic spaces, to
St\'ephane Nonnenmacher for his interest in our proof and for many
questions asked during our visit to CEA Saclay, and to Colin
Guillarmou for his help with Appendix~\ref{s:quasifuchsian}. Thanks finally to the two anonymous referees for many helpful comments and suggestions.

We are also grateful for NSF support from
a postdoctoral fellowship (KD) and grant DMS-1201417 (SD), and for
the hospitality of the Universit\'e Paris Nord in Spring
2011.

%%%%%%%%%%%%%%%%%%%%%%%%%%%%%%%%%%%%%%%%%%%%%%%%%%%%%%%%%%%%%%%%%%%%%%%%%%%%%%%%
%%%%%%%%%%%%%%%%%%%%%%%%%%%%%%%%%%%%%%%%%%%%%%%%%%%%%%%%%%%%%%%%%%%%%%%%%%%%%%%%

% arXiv bibliography macro
\def\arXiv#1{\href{http://arxiv.org/abs/#1}{arXiv:#1}}

\end{document}